\documentclass[pdflatex,sn-mathphys-num]{sn-jnl}

\usepackage{graphicx}%
\usepackage{multirow}%
\usepackage{amsmath,amssymb,amsfonts}%
\usepackage{amsthm}%
\usepackage{mathrsfs}%
\usepackage[title]{appendix}%
\usepackage{threeparttable}
\usepackage{xcolor}%
\usepackage{textcomp}%
\usepackage{manyfoot}%
\usepackage{booktabs}%
\usepackage{algorithm}%
\usepackage{algorithmicx}%
\usepackage{algpseudocode}%
\usepackage{listings}%
\usepackage{placeins}
\usepackage{nicefrac}
\usepackage{float}
\usepackage{subfig}
\usepackage{caption}
\usepackage{makecell}

\theoremstyle{thmstyleone}%
\newtheorem{theorem}{Theorem}
\newtheorem{corollary}{Corollary}
\newtheorem{lemma}{Lemma}
\theoremstyle{thmstyletwo}%
\newtheorem{remark}{Remark}%

\theoremstyle{thmstylethree}%
\newtheorem{assumption}{Assumption}%

\newcommand{\order}[1]{\mathcal{O}\left(#1\right)}
\newcommand{\torder}[1]{\tilde{\mathcal{O}}\left(#1\right)}
\newcommand{\torderi}[1]{\tilde{\mathcal{O}}(#1)}
\newcommand{\prt}[1]{\left(#1\right)}
\newcommand{\brk}[1]{\left[#1\right]}
\newcommand{\crk}[1]{\left\{#1\right\}}

\newcommand{\orderi}[1]{\mathcal{O}(#1)}
\newcommand{\prti}[1]{(#1)}
\newcommand{\brki}[1]{[#1]}
\newcommand{\crki}[1]{\{#1\}}
\newcommand{\abs}[1]{\left|#1\right|}

\newcommand{\inpro}[1]{\left\langle #1 \right\rangle}
\newcommand{\inproi}[1]{\langle #1 \rangle}
\newcommand{\condE}[2]{\E\brk{#1\middle|#2}}
\newcommand{\condEi}[2]{\E[#1|#2]}
\newcommand{\norm}[1]{\left\Vert #1 \right\Vert}
\newcommand{\normi}[1]{\Vert #1 \Vert}
\newcommand{\bs}[1]{\boldsymbol{#1}}

\newcommand{\R}{\mathbb{R}}
\newcommand{\E}{\mathbb{E}}

\newcommand{\I}{\mathbf{I}}
\newcommand{\x}{\mathbf{x}}
\newcommand{\z}{\mathbf{z}}
\newcommand{\g}{\mathbf{g}}
\newcommand{\1}{\mathbf{1}}

\newcommand{\T}{\intercal}
\newcommand{\sumn}{\sum_{i=1}^n}
\newcommand{\cN}{\mathcal{N}}
\renewcommand{\d}{\mathbf{d}}
\newcommand{\y}{\mathbf{y}}

\newcommand{\cF}{\mathcal{F}}

\newcommand{\cL}{\mathcal{L}}
\newcommand{\0}{\mathbf{0}}

\newcommand{\tpi}{\tilde{\Pi}}

\newcommand{\tW}{\tilde{W}}

\newcommand{\trw}{\tilde{\rho}_w}
\newcommand{\cR}{\mathcal{R}}

\newcommand{\tx}{\tilde{\mathbf{x}}}
\newcommand{\ty}{\tilde{\mathbf{y}}}

\newcommand{\td}{\tilde{\mathbf{d}}}

\newcommand{\cC}{\mathcal{C}}
\newcommand{\cH}{\mathcal{H}}

\newcommand{\cA}{\mathcal{A}}

\newcommand{\KTN}{K_T^{(\text{NCVX})}}
\newcommand{\KTP}{K_T^{(\text{PL})}}
\newcommand{\bc}{\mathbf{c}}

\newcommand{\teta}{\hat{\eta}}

\newcommand{\tPi}{\tilde{\Pi}}
\newcommand{\tcL}{\tilde{\cL}}
\newcommand{\br}{\mathbf{r}}
\newcommand{\tcH}{\tilde{\cH}}

\newcommand{\bavg}{\underline{\bar{\nabla}}}
\newcommand{\ceil}[1]{\left\lceil#1\right\rceil}
\newcommand{\ceili}[1]{\lceil#1\rceil}

\newcommand{\floor}[1]{\left\lfloor#1\right\rfloor}
\newcommand{\floori}[1]{\lfloor#1\rfloor}

\newcommand{\DSMTL}{LMT}

\usepackage{pifont}

\definecolor{cuhkpl}{RGB}{152,24,147}
\definecolor{cuhkp}{RGB}{255,1,255}

\usepackage{amsmath}
\allowdisplaybreaks[4]

\begin{document}

\title[Distributed Stochastic Momentum Tracking with Local Updates]{Distributed Stochastic Momentum Tracking with Local Updates: Achieving Optimal Communication Complexities}


\author[1]{\fnm{Kun} \sur{Huang}}\email{kunhuang@link.cuhk.edu.cn}

\author*[1]{\fnm{Shi} \sur{Pu}}\email{pushi@cuhk.edu.cn}

\affil[1]{\orgdiv{School of Data Science (SDS)}, \orgname{The Chinese University of Hong Kong, Shenzhen}, \orgaddress{\street{2001 Longxiang Boulevard}, \city{Shenzhen}, \postcode{518172}, \state{Guangdong}, \country{China}}}




\abstract{
    We propose Local Momentum Tracking (LMT), a novel distributed stochastic gradient method for solving distributed optimization problems over networks. To reduce communication overhead, LMT enables each agent to perform multiple local updates between consecutive communication rounds. Specifically, LMT integrates local updates with the momentum tracking strategy and the Loopless Chebyshev Acceleration (LCA) technique.
    We demonstrate that LMT achieves linear speedup with respect to the number of local updates as well as the number of agents for minimizing smooth objective functions with and without the Polyak-{\L}ojasiewicz (PL) condition. 
    {In particular, \DSMTL~outperforms existing distributed stochastic gradient methods with local updates, and notably attains the optimal communication complexities with sufficiently many local updates $Q\geq Q^*$.}
    {For a moderate number of local updates $Q=\order{1}$, the performance of \DSMTL~matches the best known result for $Q=1$ under smooth nonconvex objectives, and improves upon it under the PL condition. For $Q=1$, we further establish minimax performance lower bounds over two comparison families. For smooth objectives, LMT matches the lower bound. For PL objectives, LMT attains the minimax-optimal dependence on the spectral gap, up to logarithmic factors.} 
    To our knowledge, LMT is the first distributed stochastic gradient method with local updates that enjoys such properties.
}




\maketitle

\section{Introduction}

We consider a group of networked agents $\cN:=\crki{1,2,\ldots, n}$ collaborating to solve the following distributed optimization problem:
\begin{equation}
    \label{eq:P}
    \min_{x\in\R^p} f(x):= \frac{1}{n}\sumn f_i(x),
\end{equation}
where each agent $i$ has access only to stochastic gradients of its local objective function $f_i:\R^p\rightarrow\R$. Such a formulation is common in numerous applications, including signal processing \cite{Dimakis2010gossip}, distributed estimation \cite{suresh2017distributed}, and machine learning \cite{chang2020distributed}. 

Communication overhead is a primary bottleneck for solving Problem~\eqref{eq:P} over multiple agents, especially in large-scale settings. 
For example, in federated mean estimation, each agent holds a $p$-dimensional local vector that must be transmitted to a central server in every communication round \cite{suresh2017distributed}. This results in a communication cost of $\orderi{np}$ per round, which becomes a performance bottleneck when either $p$ or $n$ is large.
Distributed optimization over networks, or decentralized optimization, is a compelling alternative to mitigate this issue \cite{nedic2018network}. 
However, even in sparse networks, frequent transmission of high-dimensional models remains prohibitively expensive. 

To reduce communication costs in decentralized optimization, two primary strategies have been developed. The first strategy accelerates the algorithmic convergence, thereby decreasing the total number of communication rounds required to achieve a solution with the desired accuracy.
For example, the Distributed Stochastic Momentum Tracking (DSMT) method combines momentum tracking with the Loopless Chebyshev Acceleration (LCA) technique to achieve an accelerated convergence rate \cite{huang2025accelerated}.
However, DSMT follows a ``one-update, one-communication'' pattern, which can result in suboptimal overall communication complexity.

The second common strategy performs multiple local (stochastic) gradient updates between consecutive communication rounds. By skipping communications periodically, this approach has been proven effective in reducing the total number of communication rounds \cite{liu2024decentralized,alghunaim2024local,koloskova2020unified,yang2024accelerating}.
In particular, several studies have established the {\em linear speedup} property with respect to both the number of agents and the number of local updates, provided the number of iterations is sufficiently large.

This work introduces a novel method, Local Momentum Tracking (\DSMTL), that synthesizes the above two strategies. 
Specifically, \DSMTL~integrates multiple local updates into the framework of DSMT through a novel redesign that harmonizes local updates with the momentum tracking technique and the LCA mechanism. This redesign is crucial, as directly combining local updates with DSMT fails to attain the accelerated convergence rates (see Remark~\ref{rem:fail}).
We show that, with sufficiently many iterations, \DSMTL~achieves linear speedup with respect to the number of agents and the number of local updates for minimizing smooth objective functions, both with and without the Polyak-{\L}ojasiewicz (PL) condition. {In particular, \DSMTL~outperforms existing distributed stochastic gradient methods that incorporate local updates, and notably attains the optimal communication complexity when the number of local updates $Q$ is sufficiently large, i.e., $Q \geq Q^*$. For a moderate number of local updates $Q=\order{1}$, the iteration complexity of \DSMTL~matches the best known result for $Q=1$ under smooth nonconvex objectives, and improves upon it under the PL condition.}
{For $Q=1$, we further establish minimax performance lower bounds for two comparison families: LCA momentum-tracking and SPARKLE \cite{zhu2024sparkle}. In the nonconvex setting, the lower bound matches the upper rate achieved by \DSMTL. Under the PL condition, \DSMTL~attains the minimax-optimal dependence on the spectral gap, up to logarithmic factors.}
    To the best of our knowledge, LMT is the first distributed stochastic gradient method with local updates that achieves these results.

\subsection{Related Works}

Previous works on distributed stochastic gradient methods have primarily focused on accelerating the convergence rate while following the ``one-update, one-communication'' pattern. Under this setup, recent works have shown that distributed stochastic gradient methods can asymptotically achieve performance comparable  to their centralized counterparts \cite{pu2020asymptotic}. The key performance metric is the {\em transient time}, i.e., the number of iterations required for a distributed method to exhibit such comparable convergence.
For minimizing smooth objective functions, the transient time has been reduced from $\orderi{n^3/(1-\lambda)^4}$ \cite{lian2017can} to $\orderi{n^3/(1-\lambda)^2}$ \cite{alghunaim2021unified,huang2023cedas}, and then to $\orderi{n^{5/3}/(1-\lambda)}$ \cite{huang2025accelerated}, where $(1-\lambda)$ denotes the spectral gap of the mixing matrix associated with the underlying communication network.
When the objective function further satisfies the PL condition, the transient time has been improved from $\orderi{n/(1-\lambda)^2}$ \cite{pu2021sharp} to $\orderi{n/(1-\lambda)}$ \cite{huang2021improving,alghunaim2021unified}, and more recently to $\orderi{\sqrt{n/(1-\lambda)}}$ \cite{huang2025accelerated} when the spectral gap is known.\footnote{A different line of work achieves the optimal iteration complexity by incorporating multiple inner loops of communication and large batches of stochastic gradients \cite{lu2021optimal,yuan2022revisiting,xin2021stochastic}, which may not be applicable in practice \cite{ardakani2024slimfit,marek2026small}.}

In the presence of communication bottleneck, a natural mitigation strategy is to perform multiple local (stochastic) gradient steps between consecutive communication rounds. 
For example, the works in \cite{koloskova2020unified,li2019communication} combine local updates with the Distributed Stochastic Gradient Descent (DSGD) \cite{lian2017can,pu2021sharp} algorithm, but such methods can exacerbate the effects of data heterogeneity \cite{huang2023distributed}, potentially slowing convergence and offsetting the benefits of reduced communication.
To address this limitation, subsequent research has focused on integrating local updates with more sophisticated strategies that alleviate data heterogeneity, such as Distributed Stochastic Gradient Tracking (DSGT) \cite{pu2021distributed} and Exact Diffusion (ED)/$D^2$~\cite{yuan2018exact,tang2018d}. 
For instance, the works in \cite{ge2023gradient,huang2023computation,liu2024decentralized} combine local updates with the gradient tracking technique \cite{xu2015augmented,di2016next,nedic2017achieving}, where the K-GT method \cite{liu2024decentralized} particularly enhances computational efficiency by performing gradient tracking only during communication steps.
The Local Exact Diffusion (LED) method \cite{alghunaim2024local} incorporates local updates into ED and recovers the convergence rate of ED when a single local update is performed. The work in  \cite{yang2024accelerating} shows that, with sufficiently many local updates to solve a subproblem, the optimal communication complexity can be achieved for minimizing smooth objective functions. However, this method can increase sample complexity and does not support an arbitrary number of local updates.

Momentum \cite{polyak1964some} is a classical technique for accelerating first-order methods, and its integration with distributed optimization has been an active research direction \cite{yuan2021decentlam,wang2021distributed,huang2025accelerated}. 
Nevertheless, effectively integrating momentum with local updates remains challenging. Recent attempts to merge these techniques face limitations: for instance, the work in \cite{gao2020periodic} does not establish theoretical benefits from momentum, while the methods in \cite{du2024unified} cannot drive the optimality measure to zero.

Chebyshev Acceleration (CA) is a classical technique for accelerating the consensus rate among networked agents \cite{liu2011accelerated}. In previous works, CA has been applied in several distributed optimization methods to improve convergence with respect to the spectral gap \cite{lu2021optimal,yuan2022revisiting,kovalev2020optimal}. However, CA requires executing inner loops of multiple communication rounds between local updates, which can lead to unsatisfactory practical performance \cite{qu2019accelerated,xiao2023one}. By contrast, the work in \cite{song2021optimal} introduces the Loopless Chebyshev Acceleration (LCA) technique, which relieves from such constraints.

\subsection{Main Contribution}

The main contribution of this paper is four-fold.

First, we propose the \DSMTL~method, which achieves state-of-the-art communication {complexities} for minimizing smooth (smoothness modulus $L$) objective functions, both with and without the PL condition (modulus $\mu$). Notably, these results are achieved without requiring large batches of stochastic gradients. As summarized in Tables~\ref{tab:comp_ncvx}~and~\ref{tab:comp_scvx},  \DSMTL~outperforms existing distributed stochastic gradient methods that incorporate local updates. The improvements arise from a novel algorithmic design that integrates the local update scheme, momentum tracking, and the LCA technique.

Second, we establish that \DSMTL~achieves the optimal communication {complexities}\footnote{ Here, ``optimal communication complexities'' refers to the lower bounds in \cite{yuan2022revisiting,lu2021optimal}, which hold for zero-respecting decentralized first-order methods that communicate via partial averaging with a doubly stochastic matrix $W$ satisfying $\normi{\Pi W}_2=\lambda<1$, where each $f_i$ is $L$-smooth and admits unbiased stochastic gradients with variance bounded by $\sigma^2$. In the PL setting, we further assume that $f$ satisfies the $\mu$-PL condition.} when sufficiently many local updates are performed. Specifically, to obtain an $\varepsilon$-solution\footnote{$\sum_{t=0}^{T-1}\E\brki{\normi{\nabla f(\bar{x}_t)}^2}/T\leq \varepsilon^2$ ($\bar{x}_t$ stands for the average solution among agents).}, \DSMTL~attains the optimal communication complexity
\begin{align*}
    \order{\frac{L\Delta_f }{\sqrt{1-\lambda}\varepsilon^2}}
\end{align*}
when {$Q = \ceili{\sqrt{1-\lambda}\sigma^2/\varepsilon^2}$},
where $\Delta_f := f(\bar{x}_0) - \inf_x f(x)$, and $\sigma^2$ bounds the stochastic gradient variance. Under the additional PL condition, \DSMTL~attains {a communication complexity with optimal dependence on the spectral gap:}
\begin{align*}
    \torder{\frac{{\kappa}}{\sqrt{1-\lambda}}\log\frac{1}{\varepsilon}}
\end{align*}
to reach an $\varepsilon$-solution\footnote{$\sumn\E\brki{f(x_{i,T}) - f^*}/n\leq \varepsilon$.}. 
{To our knowledge, for smooth nonconvex objectives,
\DSMTL\ is the first distributed stochastic gradient method to achieve the optimal communication complexity, without any logarithmic factor, even when compared with methods that use large stochastic-gradient batches or multiple gossip rounds \cite{xin2021stochastic}. For objectives satisfying the PL condition, \DSMTL\ is also the first to achieve a communication complexity with optimal dependence on the spectral gap.
}

Third, for a moderate number of local updates $Q = \orderi{1}$, \DSMTL\ matches the previous best transient time $\orderi{{n^{5/3}}/\prt{1-\lambda}}$, achieved by DSMT with $Q = 1$ under smooth nonconvex objectives \cite{huang2025accelerated}. Under smooth PL objectives, \DSMTL~further improves the previous best transient time of DSMT from $\torderi{\sqrt{n/(1-\lambda)}}$ to $\torderi{n^{1/3}/\sqrt{1-\lambda}}$ whenever $1/(1-\lambda) = \Omega(n^{4/3})$, e.g., rings \cite{nedic2018network}. These results show that LMT generalizes and strengthens the performance guarantees of DSMT.

Last, focusing on the single-local-update regime $Q=1$, we establish minimax performance lower bounds over two comparison families: the LCA momentum-tracking and the SPARKLE~\cite{zhu2024sparkle}, both with prescribed polynomial network
operators of degree at most a fixed integer $D\geq1$. Together, these families provide a broad yet structured benchmark for comparison. The former includes \DSMTL~with $Q=1$ and permits admissible retuning of its
polynomial network operators and LCA parameter, while the latter captures several representative decentralized first-order methods. For smooth
nonconvex objectives, we establish lower bounds for both families that
match the upper bound of \DSMTL. For smooth PL objectives, \DSMTL\ attains the minimax-optimal
dependence on the spectral gap over both comparison classes, up to logarithmic factors and constants depending only on $D$.

\begin{table}[htbp]
    \centering
    \setlength{\tabcolsep}{3pt}
\begin{tabular}{@{}cccc@{}}
\toprule
Method                                               & Communication Complexity                                                                                                                                                      & \makecell[c]{Communication\\ Complexity \\ (Large $Q$)}  & \makecell[c]{{Transient} \\ {Time}\\ {($Q = \orderi{1}$)}} \\ \midrule
LSGT\footnotemark[1] \cite{ge2023gradient}                          & $\order{\frac{\sigma^2}{(1-\lambda)^8 Q\varepsilon^4} + \frac{1}{(1-\lambda)^{8/3}n Q^{1/3}\varepsilon^{4/3}} }$                                                              & /                                                      & /                              \\
PD-SGDM \cite{gao2020periodic}                       & $\order{\frac{\sigma^2}{nQ\varepsilon^4} + \frac{G }{(1-\lambda)\varepsilon^3} }$                                                                                             & $\order{\frac{G }{(1-\lambda)\varepsilon^3}}$          &     {$\order{\frac{n^3}{(1-\lambda)^4}}$}                          \\
K-GT \cite{liu2024decentralized}                     & $\order{\frac{\sigma^2}{nQ\varepsilon^4} + \frac{\sigma}{(1-\lambda)^2\sqrt{Q}\varepsilon^3} + \frac{1}{(1-\lambda)^2\varepsilon^2} }$                                        & $\order{\frac{1}{(1-\lambda)^2\varepsilon^2}}$         & {$\order{\frac{n^3}{(1-\lambda)^8}}$}                              \\
\makecell[c]{Local DSGD \cite{koloskova2020unified}} & $\order{\frac{\sigma^2}{nQ\varepsilon^4} + \prt{\frac{\sigma}{\sqrt{Q(1-\lambda)}} + \frac{\zeta}{1-\lambda} }\frac{1}{\varepsilon^3} + \frac{1}{(1-\lambda)\varepsilon^2} }$ & $\order{\frac{\zeta}{(1-\lambda)\varepsilon^3} }$      & {$\order{\frac{n^3}{(1-\lambda)^4}}$}                              \\
LED \cite{alghunaim2024local}                        & $\order{\frac{\sigma^2}{nQ\varepsilon^4} + \frac{\sigma}{\sqrt{Q(1-\lambda)}\varepsilon^3} + \frac{1}{(1-\lambda)\varepsilon^2} }$                                            & $\order{\frac{1}{(1-\lambda)\varepsilon^2}}$           & {$\order{\frac{n^3}{(1-\lambda)^2}}$}                              \\ \midrule
\makecell[c]{\textbf{\DSMTL~}\\ (This work)}         & $\bs{\order{\frac{\sigma^2}{nQ\varepsilon^4} + {\frac{\sigma}{\sqrt{Q}(1-\lambda)^{1/4}\varepsilon^3}} + \frac{1}{\sqrt{1-\lambda}\varepsilon^2}}}$                                                                                     & $\bs{\order{\frac{1}{\sqrt{1-\lambda}\varepsilon^2}}}$ & {$\bs{\order{\frac{n^{5/3}}{1-\lambda}}}$}                              \\ \bottomrule
\end{tabular}
\footnotetext[1]{LSGT cannot handle arbitrary large $Q$.}
\caption{Performance comparison of recent distributed stochastic gradient algorithms with local updates. The objective functions are assumed to have Lipschitz-continuous gradients. The parameter $\sigma^2$ bounds the stochastic gradient variance, i.e., $\E\brki{\normi{g_i(x;\xi) - \nabla f_i(x)}^2|x}\leq \sigma^2$, and $\zeta^2$ bounds the data heterogeneity, i.e., $\sumn\normi{\nabla f(x) - \nabla f_i(x)}^2/n\leq \zeta^2$. The parameter $G$ bounds the stochastic gradient norm, with $\E\brki{\normi{g_i(x;\xi)}^2|x}\leq G^2$. 
}
\label{tab:comp_ncvx}
\end{table}

\begin{table}[htbp]
    \centering
    \setlength{\tabcolsep}{3pt}
\begin{tabular}{@{}cccc@{}}
\toprule
Method                                       & Communication Complexity                                                                                                                     & \makecell[c]{Communication \\ Complexity \\ (Large $Q$)}              & \makecell[c]{{Transient} \\ {Time} \\ {($Q=\orderi{1}$)}} \\ \midrule
FlexGT \cite{huang2023computation}           & $\torder{\frac{\sigma^2}{nQ\varepsilon} + \frac{\sigma}{\sqrt{(1-\lambda)^3 \varepsilon}} }$                                                 & $\torder{\frac{\sigma}{\sqrt{(1-\lambda)^3 \varepsilon}}}$          & {$\torder{\frac{n}{(1-\lambda)^3}}$}                              \\
Local DSGD \cite{koloskova2020unified}       & $\torder{\frac{\sigma^2}{nQ\varepsilon} + \prt{\frac{\sigma}{\sqrt{(1-\lambda) Q}} + \frac{\zeta}{1-\lambda} }\frac{1}{\sqrt{\varepsilon}}}$ & $\torder{\frac{\zeta}{(1-\lambda)\sqrt{\varepsilon}}}$              & {$\torder{\frac{n}{(1-\lambda)^2}}$}                              \\
LED \cite{alghunaim2024local}                & $\torder{\frac{\sigma^2}{nQ\varepsilon} + \frac{\sigma}{\sqrt{Q(1-\lambda) \varepsilon}} }$                                                  & $\torder{\frac{1}{1-\lambda}\log\frac{1}{\varepsilon}}$             & {$\torder{\frac{n}{1-\lambda}}$}                             \\ \midrule
\makecell[c]{\textbf{\DSMTL~}\\ (This work)} & $\bs{\torder{\frac{\sigma^2}{nQ\varepsilon} + {\frac{\sigma}{\sqrt{Q\varepsilon}(1-\lambda)^{1/4}}} + \frac{1}{\sqrt{1-\lambda}}\log\frac{1}{\varepsilon}}}$                                                     & $\bs{\torder{\frac{1}{\sqrt{1-\lambda}}\log\frac{1}{\varepsilon}}}$ & {$\bs{\torder{\frac{n^{1/3}}{\sqrt{1-\lambda}}+n}}$}                          \\ \bottomrule
\end{tabular}
\caption{Performance comparison of recent distributed stochastic gradient algorithms with local updates. The objective functions are assumed to have Lipschitz-continuous gradients and satisfy the PL condition. 
}
\label{tab:comp_scvx}
\end{table}

\subsection{Notation and Assumptions}

Throughout this paper, vectors are assumed to be column vectors unless stated otherwise. We denote by $x_{i,t}^\ell \in\R^p$ the iterate of agent $i$ at the $\ell$-th local update following the $t$-th round of communication. For clarity of notation and presentation, we introduce the following stacked variables:
\begin{align*}
    \x_t^\ell&:= (x_{1,t}^\ell, x_{2,t}^\ell, \ldots, x_{n,t}^\ell)^{\T}\in\R^{n\times p},\\
    \nabla F (\x_t^\ell) &:= \prt{\nabla f_1(x_{1,t}^\ell),\nabla f_2(x_{2,t}^\ell),\ldots, \nabla f_n(x_{n,t}^\ell)}^{\T}\in\R^{n\times p},\quad 
    A_{\#}:=\begin{pmatrix}
        A\\
        A
    \end{pmatrix} \in\R^{2n\times p}.
\end{align*}
We denote by $\bar{x}\in\R^p$ the average of variables across agents. For instance, the variable $\bar{x}_t:= 1/n\sumn x_{i, t}$ represents the average of all the agents' iterates at the $t$-th communication round. 
We use $\normi{\cdot}$ to denote the Frobenius norm for a matrix and the $\ell_2$ norm for a vector. The notation $\inproi{a, b}$ denotes the inner product of two vectors $a, b\in\R^{p}$. For two matrices $A, B\in\R^{n\times p}$, the inner product $\inproi{A, B}$ is defined as $\inproi{A, B} := \sum_{i=1}^n\inproi{A_i, B_i}$, where $A_i$ (and $B_i$) represents the $i$-row of $A$ (and $B$), respectively.

We next introduce the standing assumptions.

\begin{assumption}
    \label{as:abc}
    Each agent $i\in\cN$ has access to an unbiased stochastic gradient $g_i(x;\xi_i)$ of $\nabla f_i(x)$, i.e., $\condEi{g_i(x;\xi_{i})}{x} = \nabla f_i(x)$, and there exists $\sigma\geq 0$ such that for any $i\in\cN$,
    \begin{align}
        \label{eq:abc}
        \condE{\norm{g_{i}(x;\xi_{i}) - \nabla f_i(x)}^2}{x}&\leq 
        \sigma^2.
    \end{align} 
    In addition, the stochastic gradients are independent across different agents at each $t\geq 0$. 
\end{assumption}
Note that Assumption~\ref{as:abc} can be relaxed to the more general ABC condition \cite{khaled2020better,lei2019stochastic,huang2023distributed}. The convergence results for LMT under this extension follow from procedures similar to those in \cite{huang2025accelerated}.

Assumption \ref{as:smooth} is standard that requires the objective functions to be smooth and lower bounded.

\begin{assumption}
    \label{as:smooth}
    Each $f_i(x):\R^p\rightarrow\R$ is $L$-smooth, i.e., 
    \begin{align*}
        \norm{\nabla f_i(x) - \nabla f_i(x')}\leq L\norm{x - x'}, \ \forall x,x'\in\R^p.
    \end{align*}
    In addition, 
    $f(x)$ is bounded below, i.e., $f(x)\geq f^*:= \inf_{x\in\R^p} f(x)>-\infty$ for any $x\in\R^p$.
\end{assumption}

The next assumption is standard in the distributed optimization literature and specifies the properties of the communication network. Suppose the agents are connected via a graph $\mathcal{G} = (\cN, \mathcal{E})$ with $\mathcal{E}\subseteq \cN\times \cN$ representing the set of edges connecting the agents. In particular, $(i,i)\in\mathcal{E}$ for all $i\in\mathcal{N}$. The set of neighbors of agent $i$ is denoted by $\mathcal{N}_i=\{j\in \mathcal{N}:(i,j)\in \mathcal{E}\}$.
The element $w_{ij}$ in the weight matrix $W\in\mathbb{R}^{n\times n}$ denotes the weight of the edge between agents $i$ and $j$.

\begin{assumption}
    \label{as:graph_p}
    The graph $\mathcal{G}$ is undirected and connected, i.e., there exists a path between any two nodes in $\mathcal{G}$. There is a direct link between $i$ and $j$ $(i\neq j)$ in $\mathcal{G}$ if and only if $w_{ij}>0$ and $w_{ji}>0$; otherwise, $w_{ij}=w_{ji}=0$. The mixing matrix $W$ is nonnegative, stochastic, and symmetric: $W\1=\1$ and $W^{\T} = W$, and $W$ is positive semidefinite.
\end{assumption} 

Note that Assumption~\ref{as:graph_p} guarantees that $W$ is doubly stochastic and the spectral norm $\lambda$ of the matrix $(W - \1\1^{\T}/n)$ satisfies $\lambda<1$.

Assumption~\ref{as:PL} characterizes a specific nonconvex condition known as the Polyak-{\L}ojasiewicz (PL) condition. Overparameterized models often satisfy this condition \cite{song2023fedavg}; notably, strong convexity implies the PL condition \cite{karimi2016linear}.

\begin{assumption}
    \label{as:PL}
    There {exists} $\mu>0$ such that the global objective $f(x)=\frac{1}{n}\sum_{i=1}^n f_i(x)$ satisfies
    \begin{align}
        \label{eq:PL}
        2\mu\prt{f(x)-f^*} \leq \norm{\nabla f(x)}^2,
    \end{align}
    for all $x \in \R^p$, where $f^*= \inf_{x\in\R^p}f(x)$.
\end{assumption}

\section{The Local Momentum Tracking Method}
\label{sec:pre_alg}

In this section, we introduce the proposed algorithm, Local Momentum Tracking (\DSMTL), which combines three key components: momentum tracking, local updates, and the LCA technique. We first elaborate on the motivation and design of each component.

\textbf{Momentum Tracking}. 
Without local updates, the core idea of momentum tracking is captured by the update rule \eqref{eq:dsmt_z_nolca} \cite{huang2025accelerated}, where we define the tracking variables $\y_t:= \prti{y_{1,t}, y_{2,t},\ldots, y_{n,t}}^{\T}\in\R^{n\times p}$ and the momentum variables $\z_t:= \prti{z_{1,t}, z_{2,t},\ldots, z_{n,t}}^{\T}\in\R^{n\times p}$.
It follows from $\1^{\T} W = \1^{\T}$ that the tracking property $\sumn y_{i,t} = \sumn z_{i,t}$ holds for all $t\geq 0$.
 \begin{align}
        \z_{t + 1} &= \beta\z_t + (1-\beta)\g_{t + 1}\label{eq:dsmt_z}\\
        \y_{t + 1} 
        &= W\y_t + \z_{t + 1} - \z_t,\; \y_0 = \z_0. \label{eq:dsmt_z_nolca}
    \end{align}
Compared with the standard gradient tracking method, the momentum parameter $\beta$ enables finer control over the stochastic gradient variance. {When chosen appropriately, it improves the dependence of the consensus rate on the spectral gap $1-\lambda$.} Here, $\g_t:= \prti{g_1(x_{1,t};\xi_{1,t}), \ldots, g_n(x_{n,t};\xi_{n,t})}^{\T}\in\R^{n\times p}$ stacks the stochastic gradients of all agents. Substituting \eqref{eq:dsmt_z} into \eqref{eq:dsmt_z_nolca} yields
\begin{align}
        \y_{t + 1} 
        &= \g_{t + 1} + \beta\prt{\z_t - \g_{t + 1}} + W\y_t - \z_t,\label{eq:dsmt_g_nolca}
    \end{align}
which highlights the role of $\y_t$ as corrected stochastic gradients. This observation motivates the subsequent local update procedure, wherein each agent updates multiple times under a corrected stochastic gradient direction before the next communication round.

\textbf{Local Updates}. 
Building on the interpretation of $\y_t$ as corrected stochastic gradients, we design the local update step as follows. {A similar design also appears in \cite{liu2024decentralized}.} The agents perform $Q$ local steps using the stochastic gradients $\g_t^\ell:= \prti{g_1(x_{1,t}^\ell;\xi_{1,t}^\ell), \ldots, g_n(x_{n,t}^\ell;\xi_{n,t}^\ell)}^{\T}$ along with the correction terms $\bc_t:= \prti{c_{1,t}, c_{2,t},\ldots, c_{n,t}}^{\T}$:
\begin{align}
    \x_t^{\ell + 1} &= \x_t^\ell - \eta_a\prt{\g_t^\ell + \bc_t},\quad \x_t^0 = \x_t, \qquad \ell = 0,1,\ldots, Q-1. \label{eq:mot_xtell_local}
\end{align}
After completing all $Q$ local updates, the momentum variables $\z_t$ are updated by aggregating the local stochastic gradients:
\begin{align}
    \z_{t + 1} &= \beta \z_t + \frac{1-\beta}{Q}\sum_{\ell=0}^{Q-1}\g_t^\ell,\; \z_0 = \mathbf{0},\label{eq:mot_dsmtl_localz}
\end{align}
where the averaged stochastic gradients can be computed as $[\prti{\x_t - \x_t^Q}/(\eta_a Q) - \bc_t]$. 
{The averaging in \eqref{eq:mot_dsmtl_localz} yields the same $\orderi{1/Q}$ variance reduction as a mini-batch estimator of size $Q$. Unlike mini-batch DSMT, which evaluates all $Q$ {stochastic} gradients at a
fixed iterate, \DSMTL\ evaluates them sequentially along the evolving local trajectory $\{\x_t^\ell\}_{\ell=0}^{Q-1}$ and therefore introduces additional drift caused by local updates. Our analysis controls this drift (Lemma~\ref{lem:sum_ub}). Nevertheless, under the assumptions considered here, we do not establish a strict worst-case complexity advantage over mini-batch DSMT. Whether such a separation can be established under stronger structural conditions, such as Hessian similarity \cite{karimireddy2020scaffold}, remains open for decentralized stochastic gradient methods with local updates, although corresponding results are available in the centralized setting \cite{patel2024limits,luo2025revisiting}.}

Motivated by the reformulation in \eqref{eq:dsmt_g_nolca}, we design the corrections $\bc_t$ to incorporate both the momentum terms and the information exchanged between neighboring nodes. Specifically, $\bc_t$ are updated using the local momentum information as well as the communicated variables:
\begin{equation}
    \label{eq:mot_ytct}
    \begin{aligned}
        \y_t = \z_{t + 1} + \bc_t,\quad
        \bc_{t + 1} = \bc_t - \y_t + W\y_t.
    \end{aligned}
\end{equation} 
The design in \eqref{eq:mot_ytct} enables $\y_t$ to track the momentum variables:
\begin{equation}
    \label{eq:mot_yt}
    \begin{aligned}
        \y_{t + 1} &= \z_{t + 2} + \bc_{t + 1}=  \bc_t -\y_t + W\y_t + \z_{t + 2} - \z_{t + 1} + \z_{t + 1}= W\y_t + \z_{t + 2} - \z_{t + 1}.
    \end{aligned}
\end{equation}
This formulation preserves the structure of the original momentum tracking update~\eqref{eq:dsmt_z_nolca}, ensuring that the momentum parameter $\beta$ contributes to the consensus process via the coefficient $(1-\beta)$ in the difference term $\z_{t + 1} - \z_{t} = (1-\beta)[\sum_{\ell=0}^{Q-1}\g_t^\ell /Q - \z_t]$ (see \cite[Section 2.1]{huang2025accelerated} for details). 

Finally, each agent performs an approximate gradient step to update the model parameters:
\begin{align}
    \label{eq:mot_xt}
    \x_{t + 1} = W\prt{\x_t - Q\eta_s\eta_a \y_t}.
\end{align}
\begin{remark}
    To effectively integrate local updates with momentum tracking, it is essential that the momentum variables $\z_{t}$ are updated after all local updates, using the average of the local stochastic gradients. This design is crucial for preserving the coefficient $(1-\beta)$. 
A naive alternative that updates the momentum variable at each local step disrupts the acceleration mechanism by eliminating the factor $(1-\beta)$
(see Remark~\ref{rem:fail} for a detailed discussion). 
\end{remark}

\textbf{Loopless Chebyshev Acceleration.} 
To further accelerate the consensus rate, we integrate the LCA technique into the proposed scheme \eqref{eq:mot_ytct} and \eqref{eq:mot_xt}. LCA achieves an accelerated consensus rate of $\orderi{1-\sqrt{1-\lambda^2}}$, as formalized in Lemma~\ref{lem:lca} \cite{song2021optimal}.

\begin{lemma}
    \label{lem:lca}
    Given a symmetric and positive semidefinite mixing matrix $W$, define $\eta_w := 1/(1 + \sqrt{1 - \lambda^2})$. Then $\trw:= \sqrt{\eta_w}$ with $1-\trw\sim\orderi{\sqrt{1-\lambda^2}}$, and for any $A\in\R^{n\times p}$ and $k\geq 0$, we have
    \begin{align*}
        \norm{\tpi\tW^k\tpi A_{\#}}^2 \leq c_0 \trw^{2k}\norm{\Pi A}^2,
    \end{align*}
	where 
	\begin{align*}
    c_0 = 14,\ \
		\tpi:= \begin{pmatrix}
			\Pi & \0 \\
			\0 & \Pi
		\end{pmatrix},\ \Pi:= I - \frac{\1\1^{\T}}{n},\ A_{\#}:= \begin{pmatrix}
			A\\
			A
		\end{pmatrix},\ \tW := \begin{pmatrix}
			(1 + \eta_w)W & -\eta_w I\\
			I & \0
		\end{pmatrix}.
	\end{align*}
\end{lemma}

Applying LCA to \eqref{eq:mot_ytct} is nontrivial: standard LCA accelerates a linear iteration of the form $\y_{t + 1} = W \y_t$ by introducing a memory term $\y_{t - 1}$, resulting in $\y_{t + 1} = (1 + \eta_w) W\y_t -\eta_w \y_{t - 1}$. However, the update in \eqref{eq:mot_ytct} does not conform to this standard structure, making direct application of LCA challenging.
To resolve this, we introduce an auxiliary variable $\y_t^{(l)}$ serving as the memory term required for LCA:
\begin{equation}
    \label{eq:ytlp1}
    \begin{aligned}
        \y_{t}^{(l)} &= \z_{t + 1} + \bc_{t-1}.
    \end{aligned}
\end{equation}

Combining \eqref{eq:ytlp1} with the update $\y_t = \z_{t + 1} + \bc_t$ from \eqref{eq:mot_ytct} (with LCA), we obtain the following scheme:
\begin{equation}
    \label{eq:typ1}
    \begin{aligned}
        \ty_{t + 1} &= \begin{pmatrix}
            (1 + \eta_w)W & -\eta_w I\\
            I & \mathbf{0}
        \end{pmatrix}
        \begin{pmatrix}
            \y_t \\
            \y_t^{(l)}
        \end{pmatrix} + \begin{pmatrix}
            \z_{t + 2} - \z_{t + 1}\\
            \z_{t + 2} - \z_{t + 1}
        \end{pmatrix},\quad \ty_t:= \begin{pmatrix}
        \y_t\\
        \y_{t}^{(l)}
    \end{pmatrix},
    \end{aligned}
\end{equation}
which allows us to directly utilize Lemma~\ref{lem:lca}.
  
\textbf{The Local Momentum Tracking Method.} 
Integrating the three components discussed above, we present the complete \DSMTL~method. The updates for the $t$-th communication round are given in the compact form below:
\begin{subequations}
    \label{eq:dsmtl}
    \begin{align}
        \x_t^{\ell + 1} &= \x_t^\ell - \eta_a\prt{\g_t^\ell + \bc_t},\quad \x_t^0 = \x_t, \qquad\ell = 0,1,\ldots, Q-1. \label{eq:xtell_local}\\
        \br_{t} &= \frac{1}{\eta_a Q}\prt{\x_t - \x_t^Q} - \bc_t \\
        \z_{t + 1} &= \beta \z_t + (1-\beta)\br_t,\quad \z_0 = \mathbf{0}\label{eq:dsmtl_localz}\\
        \y_t &= \z_{t + 1}  + \bc_t \label{eq:yt}\\
        \y_t^{(l)} &= \z_{t + 1} + \bc_{t -1}\label{eq:ytl}\\
        \bc_{t + 1} &= \bc_t - \y_{t} + (1 + \eta_w)W \y_t - \eta_w \y_{t}^{(l)}\label{eq:bct}\\
        \x_{t + 1} &= (1 + \eta_w)W\prt{\x_t - Q\eta_s\eta_a\y_t} - \eta_w \prt{\x_t^{(l)} - Q\eta_s\eta_a\y_t}\label{eq:xtp1}\\
        \x_{t + 1}^{(l)} &= \x_t - Q\eta_s\eta_a\y_t.\label{eq:xtlp1}
    \end{align}
\end{subequations}
Updates \eqref{eq:xtp1} and \eqref{eq:xtlp1} can be combined into a single accelerated consensus step:
\begin{equation}
    \label{eq:txp1}
    \tx_{t + 1} = \tW\prt{\tx_t - Q\eta_s\eta_a \prt{\y_t}_{\#}},\quad \tx_t:=\begin{pmatrix}
        \x_t\\
        \x_t^{(l)}
    \end{pmatrix}.
\end{equation}

{
\begin{remark}
    The corrected local update mechanism in \eqref{eq:xtell_local} and \eqref{eq:bct} follows the principle of K-GT \cite{liu2024decentralized}, whereas the momentum update \eqref{eq:dsmtl_localz} and the associated tracking structure involving $\y_t$ {are inherited} from DSMT \cite{huang2025accelerated}.
    For $Q>1$, \DSMTL\ couples these two structures by first aggregating stochastic gradients evaluated along the local updates as $\br_t=\sum_{\ell=0}^{Q-1}\g_t^\ell/Q$, and then performing a single momentum update after all $Q$ local steps. This placement preserves the momentum tracking property 
    \begin{align*}
        \sumn y_{i,t} = \sumn z_{i, t+ 1},\quad \z_{t + 1} - \z_t = (1-\beta)\prt{\frac{1}{Q}\sum_{\ell=0}^{Q-1}\g_t^\ell - \z_t}.
    \end{align*}
    In contrast, naively applying the momentum update at every local step would not preserve this relation, as detailed in Remark~\ref{rem:fail}.

    For $Q=1$, \DSMTL\ reduces to DSMT with $\alpha_{\text{DSMT}} = \eta_a\eta_s Q$. When $\beta=0$ and $\eta_w=0$, \DSMTL\ recovers K-GT \cite{liu2024decentralized}.
    Hence, \DSMTL\ jointly extends these two methods by coupling local updates with momentum tracking and LCA.

\end{remark}
}

The formal description of \DSMTL~is given in Algorithm~\ref{alg:dsmtl}. At the $t$-th communication round, agent $i$ first performs $Q$ local steps using the corrected stochastic gradient (Line~\ref{line:dsmtl_localx}). The resulting local stochastic gradients are averaged to update the momentum variable $z_{i,t + 1}$ (Line~\ref{line:zt}). 
{Then, agent $i$ forms the tracking variables $y_{i,t}$ and $y_{i,t}^{(l)}$ (Line~\ref{line:yt}) as well as the intermediate iterates $x_{i,t+\frac{1}{2}}$ and $x_{i,t+\frac{1}{2}}^{(l)}$ (Line~\ref{line:xt}).
Subsequently, each agent exchanges the tuple $(y_{i,t},\,x_{i,t+\frac{1}{2}})$ with its neighbors in a single communication round, after which the correction $c_{i,t+1}$ (Line~\ref{line:ct}) and the accelerated iterates $x_{i,t+1}$ together with its memory $x_{i,t+1}^{(l)}$ (Line~\ref{line:xtp1}) are computed locally without further communication.}

\begin{algorithm}[h]
	\begin{algorithmic}[1]
		\State Initialize $x_{i,0} = x_{i,0}^{(l)}\in\R^p$ for all agent $i\in\mathcal{N}$, determine $W = [w_{ij}]\in\R^{n\times n}$, parameters $Q$, $\eta_a$, $\eta_s$, and $\beta$. 
        Set $z_{i,0} = \mathbf{0}$ and $c_{i,0} = c_{i,-1} = \mathbf{0}$ for any $i\in\cN$. Input $\eta_w$.
		\For{$t=0, 1, 2, \ldots, T-1$}
		\For{Agent $i = 1, 2, \ldots, n$ in parallel}
        \State $x_{i,t}^0 = x_{i,t}$
        \For{Local update $\ell = 0,1,\ldots,Q-1$}
        \State Acquires a stochastic gradient $g_{i,t}^{\ell} = g_i(x_{i, t}^{\ell};\xi_{i, t}^{\ell})\in\R^p$\label{line:dsmtl_g}
		\State $x_{i,t}^{\ell + 1} = x_{i,t}^\ell - \eta_a(g_{i,t}^\ell + c_{i,t})$\label{line:dsmtl_localx}
        \EndFor
        \State $r_{i,t} = (x_{i,t} - x_{i,t}^Q)/(\eta_aQ) - c_{i,t}$\label{line:rt}
		\State $z_{i,t + 1} = \beta z_{i,t} + (1 - \beta)r_{i,t}$ \label{line:zt}
        \State $y_{i,t} = z_{i,t + 1} + c_{i,t}$, $y_{i,t}^{(l)} = z_{i,t + 1} + c_{i, t-1}$\label{line:yt}
        \State $x_{i,t + \frac{1}{2}} = x_{i,t} - Q\eta_s\eta_a y_{i,t}$,\; $x_{i,t+\frac{1}{2}}^{(l)}= x_{i,t}^{(l)} - Q\eta_s\eta_a y_{i,t}$\label{line:xt}
        \State $c_{i,t + 1} = c_{i,t} - y_{i,t} + (1 + \eta_w)\sum_{j\in\cN_i}w_{ij}y_{j,t} - \eta_w y_{i,t}^{(l)}$\label{line:ct}
        \State $x_{i, t + 1} = (1 + \eta_w)\sum_{j\in\cN_i}w_{ij} x_{j,t + \frac{1}{2}}-\eta_w x_{i,t + \frac{1}{2}}^{(l)}$,\; $x_{i,t + 1}^{(l)} = x_{i, t + \frac{1}{2}}$\label{line:xtp1}
		\EndFor
		\EndFor
	\end{algorithmic}
	\caption{Local Momentum Tracking (\DSMTL)}
	\label{alg:dsmtl}
\end{algorithm}

\section{Preliminary Analysis}
In this section, we present several preliminary results that form the foundation for proving the main convergence theorems. The analysis extends the framework developed for DSMT, with novel contributions addressing two distinguishing features of the \DSMTL~method: the local update scheme (Lemma~\ref{lem:sum_ub}) and the aggregated momentum variables (Lemmas~\ref{lem:barz} and \ref{lem:ztf}).
We organize the analysis into three parts.
First, we establish recursions for the averaged variables $\bar{x}_t$ and $\bar{z}_t$ (Lemmas~\ref{lem:avg}-\ref{lem:sum_ub}).
Then, we derive recursions for the consensus error terms (Lemmas~\ref{lem:tPitx} and \ref{lem:tPi_cRy}) and the momentum variables (Lemma~\ref{lem:ztf}). 
Finally, we construct a Lyapunov function by combining the preceding results and establish its approximate descent property (Lemma~\ref{lem:tcL}). All results in this section rely on Assumptions~\ref{as:abc}-\ref{as:graph_p}.

We begin by formalizing the update rules for the averaged variables $\bar{x}_t$, $\bar{y}_t$, and $\bar{z}_t$ in Lemma~\ref{lem:avg}.

\begin{lemma}
    \label{lem:avg}
    Let Assumption~\ref{as:graph_p} hold, and define $\teta:=\eta_a\eta_sQ$. We have for all $t\geq 0$ that 
    \begin{equation}
        \label{eq:xt_avg}
        \bar{x}_{t + 1} = \bar{x}_t - \teta \bar{y}_t,
    \end{equation}
    and 
    \begin{equation}
        \label{eq:yt_tracking}
        \bar{y}_t = \bar{y}_t^{(l)} = \bar{z}_{t + 1},\quad \bar{z}_{t + 1} = \beta\bar{z}_t + \frac{1-\beta}{Q}\sum_{\ell=0}^{Q-1}\bar{g}_t^\ell.
    \end{equation}
\end{lemma}

\begin{proof}
    See Appendix \ref{app:avg}.
\end{proof}

Lemma~\ref{lem:avg} shows that the averaged iterate $\bar{x}_t$ follows an update rule resembling SGD with momentum (SGDM). This observation enables us to leverage established techniques from the SGDM literature, e.g., \cite{liu2020improved,qiu2024convergence}. Specifically, we introduce an auxiliary sequence $\crki{\bar{d}_t}$ to facilitate the analysis, defined as:
\begin{equation}
    \label{eq:dt}
    \begin{aligned}
        \bar{d}_{t} := \begin{cases}
            \bar{x}_t, & t = 0\\
            \frac{1}{1-\beta}\bar{x}_t - \frac{\beta}{1-\beta}\bar{x}_{t -1}, & t\geq 1.
        \end{cases}
    \end{aligned}
\end{equation}

\begin{lemma}
    \label{lem:dt}
    Let Assumption \ref{as:graph_p} hold. We have for all $t\geq 0$ that
    \begin{align}
        \bar{d}_{t + 1} = \bar{d}_t - \frac{\teta}{Q}\sum_{\ell=0}^{Q-1}\bar{g}_t^\ell,\label{eq:dt_update}
    \end{align}
    and 
    \begin{align}
        \label{eq:dt_xt}
       \bar{d}_{t} - \bar{x}_t = -\frac{\teta\beta}{1-\beta}\bar{z}_{t}.
    \end{align}
\end{lemma}

\begin{proof}
    See Appendix \ref{app:dt}.
\end{proof}

Lemma~\ref{lem:dt} shows that the auxiliary variable $\bar{d}_t$ evolves similarly to an SGD iterate. This motivates the study of the approximate descent property of $\E\brki{f(\bar{d}_t) - f^*}$, as stated in Lemma~\ref{lem:descent} below.

\begin{lemma}
    \label{lem:descent}
     Let Assumptions~\ref{as:abc},~\ref{as:smooth},~and~\ref{as:graph_p} hold. Denote $\bavg_t:= \sum_{\ell=0}^{Q-1}\sumn\nabla f_i(x_{i,t}^\ell)/(nQ)$. We have for all $t\geq 0$ that
     \begin{equation}
        \label{eq:descent}
        \begin{aligned}
            \E\brk{f(\bar{d}_{t + 1})} &\leq \E\brk{f(\bar{d}_t)} - \frac{\teta}{2}\E\brk{\norm{\nabla f(\bar{x}_t)}^2} - \frac{\teta}{4}\prt{1 - 4\teta L}\E\brk{\norm{\bavg_t}^2}\\
            &\quad + \frac{\teta^3\beta^2 L^2}{(1-\beta)^2}\E\brk{\norm{\bar{z}_t}^2}  + \frac{\teta L^2}{2nQ}\sum_{\ell=0}^{Q-1}\sumn\E\brk{\norm{x_{i,t}^\ell - \bar{x}_t}^2} + \frac{\teta^2 L\sigma^2}{nQ}.
        \end{aligned}
     \end{equation}
\end{lemma}

\begin{proof}
    See Appendix \ref{app:descent}.
\end{proof}

The primary goal of this section is to construct and analyze a Lyapunov function, starting from the recursion \eqref{eq:descent} in Lemma~\ref{lem:descent}. Since recursion \eqref{eq:descent} contains several error terms on its right-hand side, we first derive recursions for these terms to establish an overall approximate descent property. Accordingly, we begin by developing a recursion for the averaged momentum term in Lemma~\ref{lem:barz} and deriving an upper bound for the local update errors in Lemma~\ref{lem:sum_ub}. 

\begin{lemma}
    \label{lem:barz}
    Let Assumptions~\ref{as:abc},~\ref{as:smooth},~and~\ref{as:graph_p} hold. We have for all $t\geq 0$ that
    \begin{equation}
        \label{eq:barz}
        \begin{aligned}
             \E\brk{\norm{\bar{z}_{t + 1}}^2} &\leq \beta \E\brk{\norm{\bar{z}_t}^2} + \frac{2(1-\beta)^2\sigma^2}{nQ}+ 3(1-\beta)\E\brk{\norm{\bavg_t}^2}.
        \end{aligned}
    \end{equation}
\end{lemma}

\begin{proof}
    See Appendix \ref{app:barz}.
\end{proof}

\begin{lemma}
    \label{lem:sum_ub}
    Let Assumptions~\ref{as:abc},~\ref{as:smooth},~and~\ref{as:graph_p} hold. Let $\eta_a\leq 1/(6\sqrt{2} QL)$. We have for all $t\geq 0$ that
    \begin{equation}
        \label{eq:sum_ub}
        \begin{aligned}
            &\frac{1}{nQ}\sum_{\ell=0}^{Q-1}\sumn \E\brk{\norm{x_{i,t}^\ell - \bar{x}_t}^2}\leq \frac{9}{n}\E\brk{\norm{\Pi\x_t}^2} + 9\eta_a^2 Q \sigma^2 + \frac{3\eta_a^2Q^2}{n}\E\brk{\norm{\Pi\y_t}^2} \\
            &\quad + 3\eta_a^2Q^2\E\brk{\norm{\nabla f(\bar{x}_t)}^2} + \frac{3\eta_a^2Q^2\beta^2}{n}\E\brk{\norm{\z_t - \nabla
             F(\1\bar{x}_t^{\T})}^2}.
        \end{aligned}
    \end{equation}
\end{lemma}

\begin{proof}
    We begin by analyzing the term $\x_t^\ell - \1\bar{x}_t^{\T}$. 
    From \eqref{eq:xtell_local}, we obtain
    \begin{equation}
        \label{eq:xtell2xt}
        \begin{aligned}
            \x_t^\ell = \x_t - \eta_a \sum_{s=0}^{\ell - 1} \g_t^s  - \eta_a\ell \bc_t.
        \end{aligned}
    \end{equation}
    Let $\boldsymbol{\Delta}_t^s:= \g_t^s - \nabla F(\x_t^s)$. Noting from \eqref{eq:yt} that $\bc_t = \y_t - \z_{t + 1}$, we derive from \eqref{eq:xtell2xt}:
    \begin{equation}
        \label{eq:sum_s1}
        \begin{aligned}
            &\x_t^\ell - \1\bar{x}_t^{\T} 
            = \Pi\x_t - \eta_a \sum_{s=0}^{\ell-1}\brk{\bs{\Delta}_t^s + \nabla F(\x_t^s) - \nabla F(\1\bar{x}_t^{\T})} - \eta_a\ell\1\brk{\nabla f(\bar{x}_t)}^{\T}\\
            &\quad - \eta_a\ell\Pi\crk{\y_t - \beta\brk{\z_t - \nabla F(\1\bar{x}_t^{\T})} -\frac{1-\beta}{Q}\sum_{s=0}^{Q-1}\brk{\bs{\Delta}_t^s + \nabla F(\x_t^s) - \nabla F(\1\bar{x}_t^{\T})}},
        \end{aligned}
    \end{equation}
    where we used $I = \Pi + \1\1^{\T}/n$ and $ \frac{\1\1^{\T}}{n}\brki{\y_t - \z_{t + 1} + \nabla F(\1\bar{x}_t^{\T})} = \1\brki{\nabla f(\bar{x}_t)}^{\T}$ based on \eqref{eq:yt_tracking}.
    From \eqref{eq:sum_s1}, we obtain
    \begin{equation}
        \label{eq:sum_s2}
        \begin{aligned}
            &\frac{1}{8}\condE{\norm{\x_t^\ell - \1\bar{x}_t^{\T}}^2}{\cF_t^0} 
            \leq \norm{\Pi\x_t}^2 + 2\eta_a^2\ell n \sigma^2 + 2\eta_a^2L^2\ell \sum_{s=0}^{Q-1}\condE{\norm{\x_t^s-\1\bar{x}_t^{\T}}^2}{\cF_t^0} \\
            &\quad + \eta_a^2\ell^2\norm{\Pi\y_t}^2 + n\eta_a^2\ell^2\norm{\nabla f(\bar{x}_t)}^2 + \eta_a^2\ell^2\beta^2\norm{\z_t - \nabla F(\1\bar{x}_t^{\T})}^2,
        \end{aligned}
    \end{equation}
    where the term $\condEi{\normi{\sum_{s=0}^{\ell-1}\boldsymbol{\Delta}_t^s}^2}{\cF_t^0}$ is bounded similarly to the derivation of \eqref{eq:avg_var}.
    Averaging both sides of \eqref{eq:sum_s2} over $\ell = 0,1,\ldots, Q-1$, collecting the terms involving $\sum_{\ell=0}^{Q-1}\condEi{\normi{\x_t^\ell - \1\bar{x}_t^{\T}}^2}{\cF_t^0}/Q$, letting $\eta_a \leq 1/(6\sqrt{2}Q L)$, we obtain the desired result after taking the full expectation.
\end{proof}

Lemma~\ref{lem:sum_ub} leads us to derive recursions for the consensus errors $\E\brki{\normi{\Pi\x_t}^2}$ and $\E\brki{\normi{\Pi\y_t}^2}$. To this end, we construct auxiliary sequences $\cR_t^x$ and $\cR_t^y$ 
that upper bound these errors, following procedures similar to those in \cite{huang2025accelerated,song2021optimal}.
The recursions for the auxiliary sequences are stated in Lemmas \ref{lem:tPitx} and \ref{lem:tPi_cRy} below.

\begin{lemma}
    \label{lem:tPitx}
    Let Assumptions~\ref{as:abc},~\ref{as:smooth},~and~\ref{as:graph_p} hold. Consider the sequence $\{\cR_t^x\}$ given by
    \begin{equation}
        \label{eq:cRx}
        \begin{aligned}
           \cR_0^x & = c_0 \norm{\Pi\x_0}^2,  \text{ and } \cR_{t + 1}^x
            = \trw\cR_t^x + \frac{c_0\trw^2\teta^2}{1-\trw}\E\brk{\norm{\Pi\y_t}^2}, \qquad t\ge0,
        \end{aligned}
    \end{equation}
    We have
    \begin{equation}
        \label{eq:Pix2cRx}
        \E\brk{\norm{\Pi\x_t}^2}\leq \cR_t^x, \qquad t\geq 0.
    \end{equation}
  
\end{lemma}

\begin{proof}
    See Appendix~\ref{app:tPitx}.
    
\end{proof}

\begin{lemma}
    \label{lem:tPi_cRy}
    Let Assumptions~\ref{as:abc},~\ref{as:smooth},~and~\ref{as:graph_p} hold. 
    Consider the sequence $\{\cR_t^y\}$ given by
    \begin{equation}
        \label{eq:cRy}
        \begin{aligned}
           \cR_0^y & = c_0\E\brk{\norm{\Pi\y_0}^2}, \text{ and } \\
           \cR_{t + 1}^y &= \trw\cR_t^y + \frac{5c_0(1-\beta)^2 L^2}{Q(1-\trw)}\sum_{\ell=0}^{Q-1}\E\brk{\norm{\x_{t + 1}^\ell - \1\bar{x}_{t + 1}^{\T}}^2} \\
            &\quad + \frac{5c_0(1-\beta)^2}{1-\trw}\E\brk{\norm{\z_{t+1} - \nabla F(\1\bar{x}_{t + 1}^{\T})}^2} + \frac{5c_0(1-\beta)^2n\sigma^2}{Q}  , \qquad t\ge 0.
        \end{aligned}
    \end{equation}
    We have
    \begin{equation}
        \label{eq:Piy2cRy}
        \E\brk{\norm{\Pi\y_t}^2}\leq \cR_t^y, \qquad t\geq 0.
    \end{equation}
\end{lemma}

\begin{proof}
    See Appendix~\ref{app:tPi_cRy}.
\end{proof}

The recursion for $\cR_t^y$ in \eqref{eq:cRy} necessitates bounding the error term $\E\brki{\normi{\z_t - \nabla F(\1\bar{x}_t^{\T})}^2}$, which is provided in Lemma~\ref{lem:ztf}. 

\begin{lemma}
    \label{lem:ztf}
    Let Assumptions~\ref{as:abc},~\ref{as:smooth},~and~\ref{as:graph_p} hold. Let $\eta_a\leq 1/(6\sqrt{2}QL)$. We have for all $t\geq 0$ that
    \begin{equation}
        \label{eq:ztf}
        \begin{aligned}
            &\E\brk{\norm{\z_{t + 1} - \nabla F(\1\bar{x}_{t+1}^{\T})}^2}\leq \frac{1 + \beta}{2}\E\brk{\norm{\z_t - \nabla F(\1\bar{x}_t^{\T})}^2} + 45(1-\beta)L^2\E\brk{\norm{\Pi\x_t}^2} \\
            &\quad + \frac{15\teta^2 L^2n\beta^2}{1-\beta}\E\brk{\norm{\bar{z}_t}^2} + 15\eta_a^2Q^2L^2(1-\beta)\E\brk{\norm{\Pi\y_t}^2}\\
            &\quad + \frac{3(1-\beta)^2 n\sigma^2}{Q}\prt{1 + \frac{15\eta_a^2Q^2L^2}{1-\beta} + \frac{5\teta^2 L^2}{n(1-\beta)}} + 15\teta^2 L^2(1-\beta)n\E\brk{\norm{\bavg_t}^2}\\
            &\quad + 15\eta_a^2Q^2L^2(1-\beta)n\E\brk{\norm{\nabla f(\bar{x}_t)}^2}.
        \end{aligned}
    \end{equation}
\end{lemma}

\begin{proof}
    See Appendix \ref{app:ztf}.
\end{proof}

\begin{remark}
    \label{rem:fail}
    We justify our design choice in \eqref{eq:mot_dsmtl_localz} by showing why a naive alternative fails. Consider performing multiple local momentum updates as in \eqref{eq:local_fail} below:
    \begin{equation}
    \label{eq:local_fail}
    \begin{aligned}
        \x_t^{\ell + 1} &= \x_t^\ell - \eta_a\prt{\z_t^\ell + \bc_t},\quad \z_t^0 = \z_{t-1}^Q,\\
        \z_{t}^{\ell + 1} &= \beta\z_t^\ell + (1-\beta)\g_t^\ell,\qquad \ell=0,1,\ldots, Q-1.
    \end{aligned}
\end{equation}
    It follows that $\y_t$ track the averaged momentum variables $\sum_{\ell = 0}^{Q-1} \z_t^{\ell} / Q$, leading to the following decomposition:
    \begin{equation}
    \label{eq:fail}
    \begin{aligned}
        \frac{1}{Q}\sum_{\ell = 0}^{Q-1} \z_t^{\ell} &= \frac{1}{Q}\sum_{\ell = 0}^{Q-1} \beta^{\ell}\z_t^0 + \frac{1-\beta}{Q}\sum_{\ell = 1}^{Q-1} \sum_{s=0}^{\ell - 1}\beta^{\ell - 1 - s} \g_t^{s}\\
        &= \frac{1-\beta^Q}{Q(1-\beta)}\z_t^0 + \sum_{\ell=0}^{Q-2}\frac{1-\beta^{Q-1-\ell}}{Q} \g_t^\ell.
    \end{aligned}
    \end{equation}
    Consequently, we would not obtain the desired $(1-\beta)^2\sigma^2$ term on the right-hand side of \eqref{eq:ztf}.
\end{remark}

Combining Lemmas~\ref{lem:descent}-\ref{lem:ztf}, we introduce the Lyapunov function $\tcL_t$.
\begin{equation}
    \label{eq:tcL_def}
    \begin{aligned}
        \tcL_t&:= \E\brk{f(\bar{d}_t)} - f^* + \frac{4\teta^3 L^2}{(1-\beta)^3} \E\brk{\norm{\bar{z}_t}^2} + \frac{11\teta L^2}{n(1-\trw)}\cR_t^x \\
        &\quad + {\frac{2(5 + 22c_0)\teta^3L^2}{n(1-\trw)^3}}\cR_t^{y} + {\frac{6(1+30c_0 + 132c_0^2)\teta^3 L^2 }{n(1-\trw)^3}}\E\brk{\norm{\z_t - \nabla F(\1\bar{x}_t^{\T})}^2}.
    \end{aligned}
\end{equation}
Note that we work with $\E\brki{\normi{\z_t - \nabla F(\1\bar{x}_t^{\T})}^2}$ instead of $\E\brki{\normi{\z_t}^2}$ to avoid introducing additional terms related to data heterogeneity. 

Lemma~\ref{lem:tcL} establishes the approximate descent property for the Lyapunov function $\tcL_t$. 

\begin{lemma}
    \label{lem:tcL}
    Let Assumptions~\ref{as:abc},~\ref{as:smooth},~and~\ref{as:graph_p} hold. Set 
    \begin{align*}
        {\teta\leq \frac{1-\beta}{130c_0 L},\quad \eta_s\geq \frac{(1-\trw)^{\frac{3}{2}}}{\sqrt{1-\beta}}},\quad \teta = \eta_a\eta_s Q,\quad \beta\geq \trw.
    \end{align*}
    Then, we have for all $t\geq 0$ that
    \begin{equation}
        \label{eq:tcL}
        \begin{aligned}
            \tcL_{t + 1} 
            &\leq \tcL_t - \frac{\teta}{3}\E\brk{\norm{\nabla f(\bar{x}_t)}^2} - \frac{\teta}{5}\E\brk{\norm{\bavg_t}^2} + {\frac{2\teta^2L\sigma^2}{nQ}} +  \frac{{8}\teta^3L^2\sigma^2}{nQ(1-\beta)}\\
            &\quad + {10}\teta\eta_a^2QL^2\sigma^2 + {\frac{3c_0'\teta^3 L^2\sigma^2(1-\beta)^2}{Q(1-\trw)^3}},
        \end{aligned}
    \end{equation}
    {where $c_0':= 6 + 380c_0 + 1672c_0^2$.}

\end{lemma}

\begin{proof}
    See Appendix~\ref{app:tcL}.
\end{proof}

\section{Main Results}
\label{sec:res}

In this section, we establish the main convergence properties of \DSMTL~for minimizing smooth objective functions, both with and without the PL condition. We further show that these results imply optimal communication complexity (with sufficiently many local updates).

\subsection{General Nonconvex Case}

{ This part presents the convergence guarantees of \DSMTL\ for smooth nonconvex objectives. To jointly characterize the optimization and consensus errors, we denote $\x_{0:(T-1)}:= (\x_0, \x_1, \ldots, \x_{T-1})$ and define the stationarity measure as follows:
\begin{equation}\label{def:ncvx_metric}
    \mathfrak{E}^{\text{(NCVX)}}(\x_{0:(T-1)}):= \frac{1}{T}\sum_{t=0}^{T-1}\crk{\E\brk{\norm{\nabla f(\bar{x}_t)}^2} + \frac{L^2}{n}\E\brk{\norm{\Pi\x_t}^2}},
\end{equation}
which is a commonly used stationarity measure in the literature; see, e.g., \cite{zhang2019decentralized}.
Notably, Assumption~\ref{as:smooth} implies that 
\begin{align*}
    \frac{1}{nT}\sumn \sum_{t=0}^{T-1}\E\brk{\norm{\nabla f(x_{i,t})}^2} \leq 2\mathfrak{E}^{\text{(NCVX)}}(\x_{0:(T-1)}),
\end{align*} 
where the left-hand side is another commonly used stationarity measure in distributed nonconvex optimization; see, e.g., \cite{xin2021improved,huang2024distributed}. Thus, $\mathfrak{E}^{\text{(NCVX)}}(\x_{0:(T-1)})$ controls the average stationarity gap of the local iterates up to a constant factor. 

Lemma~\ref{lem:ncvx_cons} below bounds the consensus error term in $\mathfrak{E}^{\text{(NCVX)}}(\x_{0:(T-1)})$.
\begin{lemma}
    \label{lem:ncvx_cons}
    Let Assumptions~\ref{as:abc},~\ref{as:smooth},~and~\ref{as:graph_p} hold. Set 
    \begin{align}
        \label{eq:nvcx_stepsize_cons}
        \teta\leq \frac{1-\beta}{130c_0 L},\quad \eta_s\geq \frac{(1-\trw)^{\frac{3}{2}}}{\sqrt{1-\beta}},\quad \teta = \eta_a\eta_s Q,\quad \beta\geq \trw.
    \end{align}
    Then, for $T\ge 1$, we have
    \begin{equation}
        \label{eq:cons_sum_ub}
        \begin{aligned}
            &\frac{1}{T}\sum_{t=0}^{T-1}\E\brk{\norm{\Pi\x_t}^2} \leq  \frac{3c_0\normi{\Pi\x_0}^2}{(1-\trw)T} + \frac{\normi{\nabla F(\1\bar{x}_0^{\T})}^2}{L^2(1-\trw) T} + \frac{4c_0^2\teta^2(1-\beta)^2 n L\Delta_f}{(1-\trw)^3 T}\\
            &\quad + \brk{2\eta_a^2Q^2L^2 + \frac{12\teta^2 L^2}{(1-\beta)^2}}\frac{904c_0^2\teta^2(1-\beta)^2n}{(1-\trw)^4T}\sum_{t=0}^{T-1}\E\brk{\norm{\nabla f(\bar{x}_t)}^2}\\
            &\quad  + \frac{92c_0^2\teta^2 (1-\beta)^2n\sigma^2}{(1-\trw)^3Q} + \eta_a^2 Q n\sigma^2.
        \end{aligned}
    \end{equation}
\end{lemma}

\begin{proof}
    See Appendix \ref{app:ncvx_cons}.
\end{proof}
}
\begin{theorem}
    \label{thm:ncvx}
    Let Assumptions~\ref{as:abc},~\ref{as:smooth},~and~\ref{as:graph_p} hold. Set 
    \begin{align}
        \label{eq:nvcx_stepsize}
        {\teta\leq \frac{1-\beta}{130c_0 L},\quad \eta_s\geq \frac{(1-\trw)^{\frac{3}{2}}}{\sqrt{1-\beta}}},\quad \teta = \eta_a\eta_s Q,\quad \beta\geq \trw.
    \end{align}
    Denote $\Delta_f=f(\bar{x}_0) - f^*$. Then, for $T\ge 1$, we have
    {
        \begin{equation}
        \label{eq:ncvx_whole}
        \begin{aligned}
            &\mathfrak{E}^{\text{(NCVX)}}\prt{\x_{0:(T-1)}}\leq \frac{5\Delta_f}{\teta T} + \frac{44c_0 L^2\normi{\Pi\x_0}^2}{n(1-\trw)T} + \frac{6\normi{\nabla F(\1\bar{x}_0^{\T})}^2}{n(1-\trw) T} + \frac{8\teta L\sigma^2}{nQ}  \\
            &\quad + \frac{48c_0(5+22c_0)(1-\beta)^2\teta^2 L^2\sigma^2}{(1-\trw)^3 QT} + 16\brk{\frac{2}{nQ(1-\beta)} + \frac{c_0'(1-\beta)^2}{Q(1-\trw)^3}}\teta^2 L^2\sigma^2 \\
            &\quad +\frac{L \Delta_f}{ T} + 33\eta_a^2 Q L^2\sigma^2.
        \end{aligned}
    \end{equation}
    }

    In particular, by setting 
    {
    \begin{equation}
        \label{eq:teta_cO}
        \begin{aligned}
            \beta &=\trw,\quad
            \teta = \frac{1}{\sqrt{\frac{L\sigma^2 T}{nQ\Delta_f}} + \prt{\frac{\sigma^2 L^2 T}{Q(1-\trw)\Delta_f}}^{\frac{1}{3}} + \frac{130c_0L}{1-\beta}},\\
            \eta_a &= \frac{1}{\sqrt{\frac{L\sigma^2 QT}{\Delta_f}} + \frac{Q(1-\trw)n^{1/6}}{\teta}},\quad \eta_s = \frac{\teta}{\eta_a Q}\geq \frac{(1-\trw)^{3/2}}{\sqrt{1-\beta}},
        \end{aligned}
    \end{equation}
    }
    we obtain the following convergence rate, where $\orderi{\cdot}$ hides numerical constants:
    \begin{equation}
        \label{eq:ncvx_teta}
        \begin{aligned}
            &{\mathfrak{E}^{\text{(NCVX)}}\prt{\x_{0:(T-1)}}} = \order{\sqrt{\frac{L\Delta_f \sigma^2}{nQT}} + {\prt{\frac{L\Delta_f \sigma}{\sqrt{Q}(1-\lambda)^{1/4}T}}^{\frac{2}{3}}} \right.\\
            &\left.\quad + \frac{L^2 \norm{\Pi\x_0}^2 }{n\sqrt{1-\lambda} T} + \frac{\norm{\nabla F(\1\bar{x}_0^{\T})}^2}{n\sqrt{1-\lambda} T} + \frac{L\Delta_f}{T} + \frac{L\Delta_f}{\sqrt{1-\lambda}T}}.
        \end{aligned}
    \end{equation}

\end{theorem}

\begin{remark}
    \label{rem:optimal_communication}
    
    From \eqref{eq:ncvx_teta}, if we further assume the mild conditions $L\normi{\Pi\x_0}^2 = \orderi{n \Delta_f}$ and $\normi{\nabla F(\1\bar{x}_0^{\T})}^2 = \orderi{n L \Delta_f}$, then the communication complexity of \DSMTL~to achieve $\varepsilon$-stationarity, i.e., $\sum_{t=0}^{T-1}\E\brki{\normi{\nabla f(\bar{x}_t)}^2}/T \leq \varepsilon^2$, is given by
    \begin{align*}
        \order{\frac{L\Delta_f \sigma^2}{nQ\varepsilon^4} + {\frac{L\Delta_f \sigma}{\sqrt{Q}(1-\lambda)^{1/4}\varepsilon^{3}}} + \frac{L\Delta_f }{\sqrt{1-\lambda}\varepsilon^2} }.
    \end{align*}
    Therefore, to achieve the optimal communication complexity \cite{lu2021optimal,yuan2022revisiting}
    $$ T=\order{\frac{\Delta_f L}{\sqrt{1-\lambda}\varepsilon^2}},$$ 
    it suffices to set 
    {
    \begin{align*}
        Q = \ceil{\frac{\sqrt{1-\lambda}\sigma^2}{\varepsilon^2}}.
    \end{align*}
    }
    Then, the corresponding sample complexity is 
    {
    \begin{align*}
        QT = \order{\frac{L\Delta_f \sigma^2}{\varepsilon^4} + \frac{L\Delta_f \sigma}{(1-\lambda)^{1/4}\varepsilon^{3}}  + \frac{\Delta_f L}{\sqrt{1-\lambda}\varepsilon^2}}.
    \end{align*}
    }


    For comparison, the LED method \cite{alghunaim2024local} cannot achieve the optimal communication complexity with large $Q$. In particular, LED necessitates a sample complexity of $\orderi{\Delta_f L \sigma^2/\varepsilon^{4} + \Delta_f L\sigma/(\sqrt{1-\lambda}\varepsilon^3) + L\Delta_f/[(1-\lambda)\varepsilon^2]}$ with $Q = \ceili{(1-\lambda)\sigma^2/\varepsilon^{2}}$ in order to achieve the communication complexity $\orderi{\Delta_f L/[(1-\lambda)\varepsilon^{2}]}$. Similarly, the works in \cite{liu2024decentralized,koloskova2020unified,gao2020periodic} do not achieve the optimal communication complexity.
    The work in \cite{yang2024accelerating} attains the optimal communication complexity but results in a sample complexity greater than $\torderi{\sigma^2/[(1-\lambda)^{3}\varepsilon^{4}]}$, which depends on the network topology and generally worsens under sparse graphs. Moreover, the method does not support an arbitrary number of local updates.

\end{remark}

\begin{remark}
    \label{rem:ncvx_kt}
    When {$Q = \orderi{1}$}, \DSMTL~{recovers the transient time of DSMT \cite{huang2025accelerated}, i.e., 
    \begin{align*}
        \KTN = \order{\frac{n^{\frac{5}{3}}}{1-\lambda}},
    \end{align*}
    }
    where 
    \begin{align*}
        \KTN:=\inf_{K}\crk{{\mathfrak{E}^{\text{(NCVX)}}(\x_{0:(k-1)})}\leq \order{\frac{1}{\sqrt{nk}}},\ \forall k\geq K}.
    \end{align*}

    {This is achieved by setting $Q = C$ for some constant $C>0$ and 
    \begin{align*}
        \beta &=1 - \frac{1-\trw}{n^{1/3}},\quad \teta = \frac{1}{\sqrt{\frac{L\sigma^2 T}{nC\Delta_f}} + \frac{130c_0L}{1-\beta}},\quad
            \eta_a = \frac{1}{\sqrt{\frac{L\sigma^2 C T}{\Delta_f}} + \frac{C(1-\trw)n^{1/6}}{\teta}},
    \end{align*}
    in \eqref{eq:ncvx_whole}, which yields
    \begin{equation}
        \label{eq:ncvx_whole_order}
        \begin{aligned}
            &\mathfrak{E}^{\text{(NCVX)}}\prt{\x_{0:(T-1)}} = \order{\sqrt{\frac{L\Delta_f \sigma^2 }{nT}} + \frac{L^2\normi{\Pi\x_0}^2}{n\sqrt{1-\lambda}T} + \frac{\normi{\nabla F(\1\bar{x}_0^{\T})}^2}{n\sqrt{1-\lambda}T} + \frac{n^{1/3} L\Delta_f}{\sqrt{1-\lambda}T}}.
        \end{aligned}
    \end{equation}
    }

\end{remark}

\begin{proof}[\normalfont\textbf{Proof of Theorem \ref{thm:ncvx}}]
    We first bound $\E\brki{\normi{\Pi\y_0}^2}$ by noting that $\y_0 = (1-\beta)\sum_{\ell=0}^{Q-1}\g_0^\ell/Q$. Following a similar decomposition as in \eqref{eq:ztf_s2}, we have
    \begin{equation}
        \label{eq:Piy0}
        \begin{aligned}
            &\E\brk{\norm{\Pi\y_0}^2} \leq 3(1-\beta)^2\crk{\norm{\nabla F(\1\bar{x}_0^{\T})}^2 + \frac{ n\sigma^2}{Q} + \frac{L^2}{Q}\sum_{\ell=0}^{Q-1}\E\brk{\norm{\x_0^\ell - \1\bar{x}_0^{\T}}^2}}.
        \end{aligned}
    \end{equation}
    Invoking \eqref{eq:sum_ub} and rearranging terms yields
    \begin{equation}
        \label{eq:Piy0_ub}
        \begin{aligned}
            \frac{\E\brki{\normi{\Pi\y_0}^2}}{(1-\beta)^2} &\leq 8\norm{\nabla F(\1\bar{x}_0^{\T})}^2 + \frac{8 n\sigma^2}{Q} + 54L^2\norm{\Pi\x_0}^2 + nL\Delta_f,
        \end{aligned}
    \end{equation}
    where we invoked $\normi{\nabla f(x)}^2\leq 2L(f(x) - f^*)$.

   In light of \eqref{eq:Piy0_ub}, the approximate descent property \eqref{eq:tcL} of the Lyapunov function $\tcL_t$ leads to the following inequality:
    \begin{equation}
        \label{eq:ncvx_s1}
        \begin{aligned}
            &\frac{1}{T}\sum_{t=0}^{T-1}\E\brk{\norm{\nabla f(\bar{x}_t)}^2} 
            \leq \frac{3\tcL_0}{\teta T} + \frac{6\teta L\sigma^2}{nQ}  + {30}\eta_a^2 Q L^2 \sigma^2 \\
            &\quad + {3\brk{\frac{8}{nQ(1-\beta)} + \frac{3c_0'(1-\beta)^2}{Q(1-\trw)^3}}\teta^2 L^2\sigma^2}\\
            &\leq \frac{4\Delta_f}{\teta T} + \frac{40c_0L^2\norm{\Pi\x_0}^2}{n(1-\trw)T} + \frac{4\norm{\nabla F(\1\bar{x}_0^{\T})}^2}{n{(1-\trw)}T} + {\frac{48c_0(5+22c_0)\teta^2 L^2(1-\beta)^2\sigma^2}{(1-\trw)^3QT}}  \\
            &\quad +  \frac{6\teta L\sigma^2}{nQ} +  {3\brk{\frac{8}{nQ(1-\beta)} + \frac{3c_0'(1-\beta)^2}{Q(1-\trw)^3}}\teta^2 L^2\sigma^2} + {30}\eta_a^2 Q L^2 \sigma^2.
        \end{aligned}
    \end{equation}
    {Combining \eqref{eq:ncvx_s1} and \eqref{eq:cons_sum_ub} in Lemma~\ref{lem:ncvx_cons} yields \eqref{eq:ncvx_whole}.}

    Condition \eqref{eq:teta_cO} yields
    {
    \begin{equation}
        \label{eq:teta_inv}
        \begin{aligned}
            &\frac{\Delta_f}{\teta T} = \sqrt{\frac{L\Delta_f\sigma^2}{nQT}} + \prt{\frac{L\Delta_f \sigma}{\sqrt{Q(1-\trw)}T}}^{2/3} + \frac{130c_0L\Delta_f}{(1-\trw) T},\quad 30\eta_a^2 Q L^2 \sigma^2 \leq \frac{40L\Delta_f }{T}, \\
            &\frac{\teta^2 L^2\sigma^2}{nQ(1-\trw)}\leq \frac{L\Delta_f}{(1-\trw)T},\quad \frac{\teta^2 L^2 \sigma^2}{Q(1-\trw)}\leq \prt{\frac{L^2\Delta_f^2 \sigma^2}{Q(1-\trw)T^2}}^{1/3}.
        \end{aligned}
    \end{equation}
    }
    Substituting \eqref{eq:teta_inv} into \eqref{eq:ncvx_whole} and noting that $1-\trw\sim\orderi{\sqrt{1-\lambda}}$ yields the desired result \eqref{eq:ncvx_teta}. 
\end{proof}

\subsection{PL Condition Case}

Under the additional PL condition (Assumption~\ref{as:PL}), we derive a refined recursion for $\tcL_t$, stated below.

\begin{lemma}
    \label{lem:tcL_pl}
    Let Assumptions~\ref{as:abc},~\ref{as:smooth},~\ref{as:graph_p},~and~\ref{as:PL} hold. Set {$\teta\leq {1-\beta}/\prti{200c_0 L}$, $\eta_s\geq {(1-\trw)^{3/2}}/\prti{\sqrt{1-\beta}}$, $\teta = \eta_a \eta_s Q$, and $\beta\geq \trw$.}
    We have for all $t\geq 0$ that
    \begin{equation}
        \label{eq:tcL_pl}
        \begin{aligned}
            \tcL_{t + 1}&\leq \prt{1-\frac{\teta\mu}{3}}\tcL_t - \frac{\teta}{4}\E\brk{\norm{\bavg_t}^2} +  \frac{\teta^2L\sigma^2}{nQ} + \frac{{9}\teta^3L^2\sigma^2}{nQ(1-\beta)}\\
            &\quad +{11}\teta\eta_a^2QL^2\sigma^2 + {\frac{3c_0'\teta^3 L^2(1-\beta)^2\sigma^2}{Q(1-\trw)^3}}.
        \end{aligned}
    \end{equation}
\end{lemma}

\begin{proof}
    See Appendix \ref{app:tcL_pl}.
\end{proof}

In light of Lemma~\ref{lem:tcL_pl}, we obtain the convergence of \DSMTL~for minimizing smooth objective function satisfying the PL condition, stated in Theorem~\ref{thm:pl} below. {To jointly measure the optimality and
consensus errors, define
\begin{equation}\label{def:pl_metric}
    \mathfrak{E}^{(\text{PL})}(\x_t):= \E\brk{f(\bar{x}_t) - f^*} + \frac{L}{n} \E\brk{\norm{\Pi\x_t}^2},
\end{equation} 
which is a commonly used optimality measure in the literature; see, e.g., \cite{yuan2022revisiting}.
By Assumption~\ref{as:smooth}, we have
\begin{align*}
    \frac{1}{n}\sumn \E\brk{f(x_{i,t}) - f^*} \leq \mathfrak{E}^{(\text{PL})}(\x_t),
\end{align*}
where the left-hand side is another commonly used average local optimality gap in distributed optimization; see, e.g.,
\cite{huang2025accelerated}. Thus, $\mathfrak{E}^{(\text{PL})}(\x_T)$ controls this optimality measure.
}
\begin{theorem}
    \label{thm:pl}
    Let Assumptions~\ref{as:abc},~\ref{as:smooth},~\ref{as:graph_p},~and~\ref{as:PL} hold. Denote $\kappa:=L/\mu$. Set 
    {
        \begin{align}\label{eq:teta_pl}
            \teta\leq \frac{1-\beta}{200c_0 L},\quad \eta_s\geq \frac{(1-\trw)^{3/2}}{\sqrt{1-\beta}},\quad \teta = \eta_a\eta_s Q,\quad \beta\geq \trw.
        \end{align}
    }
    Then, for all $t\geq 1$,
    {
    \begin{equation}
        \label{eq:pl_ub}
        \begin{aligned}
            &\mathfrak{E}^{(\text{PL})}(\x_t)\leq  52\prt{1 - \frac{\teta\mu}{3}}^{t-1}\tcL_0 + \frac{120 (1-\beta)^2 L}{n(1-\trw)^2}\prt{\frac{4+\beta}{5}}^t\tcH_0 \\
            &\quad + \prt{\frac{3+\trw}{4}}^t\brk{\frac{3c_0 L\normi{\Pi\x_0}^2}{n} +\frac{45c_0^2\teta^2 L\E\brk{\normi{\Pi\y_0}^2}}{2n(1-\trw)^2}} + \frac{120\teta L\sigma^2}{ n\mu Q}  \\
                &\quad + \frac{96\teta^2L^2\sigma^2}{n\mu Q(1-\beta)} + \frac{360\eta_a^2QL^2\sigma^2}{\mu} + \frac{240c_0'\teta^2 L^2(1-\beta)^2\sigma^2}{Q\mu (1-\trw)^3}.
        \end{aligned}
    \end{equation}
    }
    where 
    \begin{equation}
        \label{eq:tcH0}
        \begin{aligned}
            \tcL_0 &\leq 2\Delta_f + \frac{2L\normi{\Pi\x_0}^2}{n} + \frac{2\normi{\nabla F(\1\bar{x}_0^{\T})}^2}{nL} + \frac{\sigma^2}{L},\\ 
            \tcH_0 &:= c_0\norm{\Pi\x_0}^2 + \frac{15c_0^2\teta^2\E\brki{\normi{\Pi\y_0}^2}}{(1-\trw)^2} + \frac{480c_0^2\teta^2\normi{\nabla F(\1\bar{x}_0^{\T})}^2}{(1-\trw)^2} \\
            &\leq 2c_0\trw \norm{\Pi\x_0}^2 + {\frac{2}{L^2}}\norm{\nabla F(\1\bar{x}_0^{\T})}^2 + \frac{8(1-\beta)n\sigma^2}{Q{L^2}} + \frac{n(1-\beta)\Delta_f}{{L}}.
        \end{aligned}
    \end{equation}

    {
        In particular, for any $T\geq 2$,
    we set $\beta=\trw$,
    \begin{align}\label{eq:teta_pl_order}
        \teta=\frac{1}{\frac{\mu (T-1)}{6\log(nQT)} + \frac{200c_0 L}{1-\trw}},\quad
        \eta_a=\frac{1}{Q\mu (T-1) + \frac{1200c_0 QL\log(nQT)}{1-\trw}},\quad
        \eta_s=6\log (nQT).
    \end{align}
    Then,
    }
    \begin{equation}
        \label{eq:pl_order}
        \begin{aligned}
            {\mathfrak{E}^{(\text{PL})}(\x_T)} &= \order{\frac{\kappa \sigma^2 {\log (n Q T)} }{n \mu Q T} +{\frac{\kappa^2\sigma^2 \log^2 (n Q T)}{\mu Q \sqrt{1-\lambda} T^2}}  \right.\\
            &\left.\quad + \tcL_0 \exp\crk{-\frac{\sqrt{1-\lambda}(T-1)}{4800c_0\kappa}} + \frac{\tcL_0}{nQT}},
        \end{aligned}
    \end{equation}
    where $\orderi{\cdot}$ hides numerical constants.
\end{theorem}

\begin{remark}
    \label{rem:pl}
    Suppose the mild conditions $L\normi{\Pi\x_0}^2 = \orderi{n \Delta_f}$, $\normi{\nabla F(\1\bar{x}_0^{\T})}^2 = \orderi{n L \Delta_f}$ hold. From \eqref{eq:pl_order}, the communication complexity of \DSMTL~to achieve an $\varepsilon$-solution, i.e., $\sumn\E\brki{f(x_{i,T}) - f^*}/n \leq \varepsilon$, is given by
    \begin{align*}
        \torder{\frac{\kappa\sigma^2}{n\mu Q \varepsilon} + {\frac{\kappa \sigma}{\sqrt{\mu Q\varepsilon}(1-\lambda)^{1/4}}} + \frac{\kappa}{\sqrt{1-\lambda}}\log\frac{1}{\varepsilon}}.
    \end{align*}
    If we set the number of local updates as 
    {
    \begin{align*}
        Q = \ceil{\frac{\sqrt{1-\lambda}\sigma^2}{n\mu \varepsilon \log(1/\varepsilon)} + \frac{\sqrt{1-\lambda}\sigma^2}{\mu\varepsilon \log^2(1/\varepsilon)}},
    \end{align*}
    }
    then the number of communication rounds required to reach an $\varepsilon$-solution scales as
    \begin{align*}
        T = \torder{\frac{\kappa}{\sqrt{1-\lambda}}\log\frac{1}{\varepsilon}},
    \end{align*}
    which is optimal with respect to the graph specifics \cite{yuan2022revisiting}.

    The corresponding sample complexity is 
    \begin{align*}
        QT = \torder{{\frac{\kappa\sigma^2}{\mu \varepsilon}}  + {\frac{\kappa\sigma}{\sqrt{\mu\varepsilon}(1-\lambda)^{1/4}}} + \frac{\kappa}{\sqrt{1-\lambda}}\log\frac{1}{\varepsilon}}.
    \end{align*}

    By comparison, the works in \cite{alghunaim2024local,koloskova2020unified} cannot achieve the optimal communication complexity with large $Q$, and the method in \cite{yang2024accelerating} requires a sample complexity greater than $\torderi{\sqrt{\kappa}\sigma^2/[\mu(1-\lambda)^{3}\varepsilon]}$ to reach optimal communication, exhibiting an unfavorable dependence on the network topology.
\end{remark}

\begin{remark}
    \label{rem:pl_kt}
     When {$Q = \orderi{1}$}, \DSMTL~{{improves} the transient time of DSMT \cite{huang2025accelerated} from $\orderi{\sqrt{n/(1-\lambda)}}$ to
    \begin{align*}
        \KTP = \torder{\max\crk{\frac{n^{1/3}}{\sqrt{1-\lambda}}, n}},
    \end{align*}
     }
    where 
    \begin{align*}
        \KTP:= \inf_{K}\crk{{\mathfrak{E}^{(\text{PL})}(\x_k)}\leq \torder{\frac{1}{nk}},\ \forall k\geq K }.
    \end{align*}
    { This is achieved by setting $Q = C$ for some constant $C>0$, $\beta = 1-(1-\trw)/n^{1/3}$, and 
    \begin{align*}
    \teta=\frac{1}{\frac{\mu (T-1)}{6\log(n C T)} + \frac{200n^{1/3}c_0 L}{1-\trw}},\quad
        \eta_a=\frac{1}{C\mu(T-1) + \frac{C(1-\trw)n^{1/6}}{\teta}},
\end{align*}
    in \eqref{eq:pl_ub} for $T\geq 2$, which yields
    \begin{equation}
        \label{eq:pl_kt_order}
        \begin{aligned}
            &\mathfrak{E}^{(\text{PL})}(\x_T)= \order{\frac{\kappa \sigma^2\log(nT)}{n\mu T}+ \exp\prt{-\frac{(T-1)\sqrt{1-\lambda}}{4800n^{1/3}c_0\kappa}}\right.\\
            &\left.\quad + \frac{\kappa^2 \sigma^2 \log^2(nT)}{\mu T^2}\prt{1 + \frac{1}{n^{2/3}\sqrt{1-\lambda}}} }.
        \end{aligned}
    \end{equation}
    
    {The improvement in transient time is particularly evident} for poorly connected graphs whenever $1/(1-\lambda) =\Omega(n)$, e.g., rings.
    }
\end{remark}

\begin{proof}[\normalfont\textbf{Proof of Theorem \ref{thm:pl}}]
    {See Appendix \ref{app:pl}.}
    \end{proof}

{
\section{Minimax Lower Bounds and Optimality of \DSMTL}

    This section establishes two nested optimality statements for \DSMTL~with $Q=1$. We first introduce the LCA momentum-tracking family and show that the worst-case order attained by \DSMTL~cannot be improved by considering broader variants of \DSMTL. We then include SPARKLE \cite{zhu2024sparkle} as an additional competitor and prove that the resulting enlarged comparison class obeys the same lower bound. Thus, enlarging the competitor set does not improve the minimax rate achieved by LMT.
    The analysis proceeds by embedding LMT into the LCA momentum-tracking family, establishing a consensus-error lower bound common to the LCA and SPARKLE families on the same mixing matrix instance, and finally combining this lower bound with optimization lower bounds. Throughout this section, $D\geq 1$ is fixed and bounds the degrees of the polynomials in $W$ used by the LCA momentum-tracking and SPARKLE families. The constants in the lower bounds may depend on $D$.

\subsection{The LCA momentum-tracking and SPARKLE Families}

    {This part introduces the LCA momentum-tracking and the SPARKLE families. The LCA momentum-tracking family generalizes \DSMTL~ with $Q=1$} by allowing { bounded-degree polynomial network operators and flexible choices of the LCA acceleration parameter.}
    The SPARKLE family includes several classical distributed first-order methods, such as DSGT \cite{pu2021distributed}, D$^2$ \cite{tang2018d}, EXTRA \cite{shi2015extra}, NIDS \cite{li2019decentralized}, Exact Diffusion \cite{yuan2018exact}, and the non-LCA form of DSMT \cite{huang2025accelerated}, thereby providing a comprehensive comparison benchmark.

\textbf{The LCA momentum-tracking Family.}
Let
\begin{equation*}
    \mathsf E:=\begin{pmatrix}I\\I\end{pmatrix},\quad
    \mathsf E_1:=\begin{pmatrix}I\\0\end{pmatrix},\quad \tx_t=\begin{pmatrix}\x_t\\\x_t^{(l)}\end{pmatrix},\quad
    \td_t:=\begin{pmatrix}\d_t\\\d_t^{(l)}\end{pmatrix},
\end{equation*}
and define the set of real polynomials\footnote{ The polynomial coefficients and scalar parameters may depend on $W$, $n$, $\theta(W)$, and a prescribed horizon.} $\Phi_D(W)$ as 
\begin{align*}
    \Phi_D(W)&:= \crk{\varphi:
    \begin{array}{l}
    \deg(\varphi)\leq D,\quad P_{\varphi}=\varphi(W),\quad \mathrm{Null}(I-P_{\varphi})=\mathrm{Span}(\1),\\
    P_{\varphi}=P_{\varphi}^\T\succeq0,\quad (P_{\varphi})_{ij}\geq0,\quad P_{\varphi}\1=\1
    \end{array}}.
\end{align*}
{Let
\begin{equation*}
    \varphi\in\Phi_D(W),\quad P=\varphi(W),\quad 
    \mathsf L_{\gamma,P}
    :=
    \begin{pmatrix}
        (1+\gamma)P&-\gamma I\\
        I&0
    \end{pmatrix}.
\end{equation*}
Given $W\in\R^{n\times n}$, $\alpha>0$, $\beta\in[0,1)$, and $\gamma\in[0,1)$, the method $\mathsf A_{\alpha,\beta,\gamma,\varphi}(W)$ of the LCA momentum-tracking family is}
\begin{subequations}
\label{eq:lca_family}
\begin{align}
    \z_{t+1}&=\beta\z_t+(1-\beta)\g_t,\quad \g_t = \prt{g_1(x_{1,t};\xi_{1,t}),\ldots, g_n(x_{n,t};\xi_{n,t})}^{\T}, \label{eq:lca_m}\\
    \y_t &=\z_{t+1}+(I-P)\mathsf E_1^{\T}\td_t,\label{eq:lca_y}\\
    \tx_{t+1}&=\mathsf L_{\gamma,P}
    \left(\tx_t-\alpha\mathsf E\y_t\right),\quad \x_{t + 1} = \mathsf{E}_1^{\T} \tx_{t+1},\label{eq:lca_x}\\
    \td_{t+1}&=\mathsf L_{\gamma,P}\td_t -(1+\gamma)\mathsf E_1\z_{t+1}, \label{eq:lca_d}
\end{align}
\end{subequations}
where $\z_0 = \mathbf{0}$, $\tx_0 = (\x_0^{\T}, \x_0^{\T})^{\T}$, and $\td_0 = \mathbf{0}$. 
{Define} 
\begin{align}
    \label{eq:LCA_alpha_beta}
    \cA_D^{\text{LCA}}(W;\alpha,\beta):= \crki{\mathsf A_{\alpha,\beta,\gamma,\varphi}(W):\gamma\in[0,1),\ \varphi\in\Phi_D(W)}.
\end{align}
The LCA momentum-tracking family {(with degree at most $D$) can be described by}
\begin{align*}
    \cA_D^{\text{LCA}}(W)&:= \bigcup_{\substack{\alpha>0, \beta\in[0,1)}} \cA_D^{\text{LCA}}(W;\alpha,\beta).
\end{align*}
For example, when $D\geq 2$, the family contains the choice $\varphi(w)=w^2$, i.e., $P=W^2$,
together with $\gamma = 1/\prti{1+\sqrt{1-\lambda^4}}$, where $\lambda=\normi{\Pi W}_2$. This variant
performs two inner communication rounds per iteration and is not covered by LMT. Conversely, LMT with
$Q=1$ is a member of $\cA_1^{\text{LCA}}(W)$, as formalized below.

\begin{lemma}
\label{lem:lca_lmt_embedding}
Let $Q=1$ in Algorithm~\ref{alg:dsmtl} and Assumption~\ref{as:graph_p} hold. With $P=W$, $\gamma=\eta_w$, $\td_0 = \mathbf{0}$, $\z_0=\mathbf{0}$, $\tx_0 = (\x_0^{\T}, \x_0^{\T})^{\T}$, and $\alpha=\eta_a\eta_s$ in \eqref{eq:lca_family}, the iterates generated by \eqref{eq:lca_family} coincide with those generated by~\eqref{eq:dsmtl}. 
{In other words, LMT with $Q=1$ is a member of $\cA_1^{\text{LCA}}(W)$.}
\end{lemma}

\begin{proof}
    See Appendix~\ref{app:lca_lmt_embedding}.
\end{proof}

\textbf{The SPARKLE Family.} {The preceding family generalizes LMT with $Q=1$. We next turn to the classical SPARKLE family that includes several representative decentralized first-order methods \cite{zhu2024sparkle}.}

Fix an integer $D\geq1$ and a symmetric mixing matrix $W$. Define
$\Phi_D^{\text{SP}}(W)$ as the set of matrix triples $(A,B,C)$ for which
there exist real polynomials $a,b,c,s$ satisfying
\begin{align*}
    &A=a(W),\quad B = b(W),\quad C=c(W),\quad B^2=s(W),\quad
    \deg(a),\deg(c),\deg(s)\leq D,\\
    &A=A^\T,\quad C=C^\T,\quad A_{ij},C_{ij}\geq0,\quad A\1=C\1=\1,\quad \mathrm{Null}(B) = \mathrm{Span}(\1),
\end{align*}
and
\begin{align*}
    \rho(\mathcal{M}) < 1,\quad \mathcal M:= \begin{pmatrix}
        \Pi C-B^2&B\\
        -B&\Pi
    \end{pmatrix},
\end{align*}
where $\rho(\cdot)$ denotes the spectral radius. 

Given $\alpha>0$, $\beta\in[0,1)$, and $(A,B,C)\in\Phi_D^{\text{SP}}(W)$, the method $\mathsf A_{\alpha,\beta,A,B,C}(W)$ of the SPARKLE family\footnote{ By introducing $\d'_{t} = B\d_t$, the recursion requires only $B^2$ instead of $B$. Hence, it suffices to impose $\mathrm{deg}(s)\leq D$.} {takes the form}
\begin{equation}
\label{eq:sparkle}
\begin{aligned}
    \z_{t+1}&=\beta\z_t+(1-\beta)\g_t,\quad \z_0 = \mathbf{0}\\
    \x_{t+1}&=C\x_t-\alpha A\z_{t+1}-B\d_t,\\
    \d_{t+1}&=\d_t+B\x_{t+1}.
\end{aligned}
\end{equation}
{Define} $\cA_D^{\text{SP}}(W;\alpha,\beta) := \{\mathsf A_{\alpha,\beta,A,B,C}(W) :  (A,B,C)\in\Phi_D^{\text{SP}}(W)\}$. The SPARKLE family {can be described by}
\begin{align*}
    \cA_D^{\text{SP}}(W)&:= \bigcup_{\substack{\alpha>0, \beta\in[0,1)}} \cA_D^{\text{SP}}(W;\alpha,\beta).
\end{align*}

\subsection{A Unified Consensus-Error Lower Bound}

We now establish a consensus-{error} lower bound shared by the two families. On the instance $f_i\equiv0$, $g_{i,t} = \xi_{i,t}$, {where $\xi_{i,t}$ is centered, temporally uncorrelated noise with bounded variance}, both families admit a common abstract dynamics.

Lemma~\ref{lem:sparkle_plus_master} below establishes a lower bound on the consensus error for {this abstract algorithmic recursion under a noise sequence} satisfying \eqref{eq:sparkle_plus_master_noise}. 

\begin{lemma}
\label{lem:sparkle_plus_master}
Let $\mathcal S$ be an index set with $M:=|\mathcal S|<\infty$. Suppose that $\{v_j\in\R^n:j\in\mathcal S\}\subseteq\mathrm{Range}(\Pi)$ is an orthonormal family. Let the noise sequence satisfy
\begin{equation}
\label{eq:sparkle_plus_master_noise}
    \E[\bs{\xi}_t]=\mathbf{0},
    \quad
    \E[\bs{\xi}_t\bs{\xi}_s^{\T}]
    =\sigma^2I_n\delta_{ts},\qquad \delta_{ts}=\begin{cases}1,&t=s,\\0,&t\neq s.\end{cases}
\end{equation}
For each $j\in\mathcal S$, define $\hat\xi_t^{(j)}:=v_j^{\T}\bs{\xi}_t$. 
Fix $T,K\in\mathbb N_+$, $\alpha>0$, and $\beta\in[0,1)$. Let $\crki{\x_t}_{t=1}^{T}\subseteq\R^n$ be random vectors such that, for every $j\in\mathcal S$, there exist deterministic $\mathsf T_j\in\R^{d_j\times d_j}$, $\mathsf h_j\in\R^{d_j}$, and $\mathsf r_j^{\T}\in\R^{d_j}$, together with sequences $\crki{z_t^{(j)}}_{t=0}^{T-1}\subseteq\R$ and $\crki{s_t^{(j)}}_{t=0}^{T-1}\subseteq\R^{d_j}$, satisfying
\begin{subequations}
\label{eq:sparkle_plus_master_realization}
\begin{align}
    z_{t+1}^{(j)}&=\beta z_t^{(j)}+(1-\beta)\hat\xi_t^{(j)},
    &z_0^{(j)}&=0,\label{eq:sparkle_plus_master_momentum}\\
    s_{t+1}^{(j)}&=\mathsf T_j s_t^{(j)}-\alpha\mathsf h_jz_{t+1}^{(j)},
    &s_0^{(j)}&=\mathbf{0}_{d_j},\label{eq:sparkle_plus_master_state}\\
    v_j^{\T}\x_{t+1}&=\mathsf r_js_{t+1}^{(j)}.\label{eq:sparkle_plus_master_readout}
\end{align}
\end{subequations}
Define $\psi_{j,k}:=\mathsf r_j\mathsf T_j^k\mathsf h_j$, $k\geq0$.
Suppose that there is a constant $\underline{\psi}>0$ such that $\psi_{j,k}\geq\underline{\psi}$, $j\in\mathcal S$, $k=0,\ldots,K-1.$
Then, with $H:=\min\{T,K\}$,
\begin{align}
    \E\brk{\norm{\Pi\x_T}^2}
    &\geq
    \frac{\alpha^2\sigma^2M\underline{\psi}^{\,2}H}{32}
    \min\crk{1,\bigl((1-\beta)H\bigr)^2}.
    \label{eq:sparkle_plus_master_finite_lb}
\end{align}
\end{lemma}

\begin{proof}
    See Appendix~\ref{app:sparkle_plus_master}.
\end{proof}

Lemma~\ref{lem:sparkle_plus_master} reduces the task of establishing the consensus-error lower bound to two ingredients: a collection $\{v_j:j\in\mathcal S\}\subseteq\mathrm{Range}(\Pi)$ and a positive lower bound {on} $\psi_{j,k}$.

The following lemma provides the first ingredient and the eigenvalue properties needed to establish the second. 
Specifically, it constructs a family of mixing matrices with bounded-degree support graphs, containing $D+1$ distinct eigenvalues, and $[n/(D+1)-1]$ orthonormal eigenvectors whose eigenvalues are close to one. 
These eigenvectors form the required collection $\{v_j:j\in\mathcal S\}$, while the $D+1$ distinct eigenvalues serve as interpolation nodes for controlling the polynomials appearing in the two families. 
The positive lower bound on $\psi_{j,k}$ is then established separately for the LCA momentum-tracking and SPARKLE families in the proofs of Theorems~\ref{thm:poly_lca_lb} and~\ref{thm:sparkle_lb}, respectively.

\begin{lemma}
\label{lem:mixing_matrix}
Fix an integer \(D\geq1\) and let $m:=D+1$. There exist constants $\theta_D>0$, $\gamma_*>0$, and an infinite set $\mathbb{N}_D\subseteq\crki{km: k\geq 2}$, all depending at most on $D$, such that the following holds.

For every $n\in\mathbb N_D$ and every $\theta\in(0,\theta_D]$, there exists a symmetric, nonnegative, positive semidefinite, and stochastic matrix
$W_{n,\theta}\in\mathbb R^{n\times n}$ satisfying the following properties:
\begin{enumerate}
    \item Let $\mathrm{supp} (W_{n,\theta})$ be the support graph with
    $\{i,j\}$ as an edge whenever $i\neq j$ and $[W_{n,\theta}]_{ij}>0$. Then, $\mathrm{supp} (W_{n,\theta})$ is connected and the maximum degree of $\mathrm{supp} (W_{n,\theta})$ is bounded by a constant depending only on $D$. 
    \item \(W_{n,\theta}\) contains $m$ distinct eigenvalues $1=\tau_1>\tau_2>\cdots>\tau_m>0$ depending only on $D$.
    \item There exist orthonormal vectors $v_2,\ldots,v_{n/m}\in\mathrm{Range}(\Pi)$ and corresponding eigenvalues \(\omega_2,\ldots,\omega_{n/m}\) such that $W_{n,\theta}v_r=\omega_rv_r$, $1-{\theta}/{\gamma_*}\leq\omega_r\leq1-\theta$, for $r=2,\ldots,n/m$, and $\omega_2 = 1-\theta$.
\end{enumerate}
Moreover, $\normi{\Pi W_{n,\theta}}_2 = 1-\theta$.
\end{lemma}

\begin{proof}
    See Appendix~\ref{app:mixing_matrix}.
\end{proof}

Theorem~\ref{thm:poly_lca_lb} {below} establishes the consensus-error lower bound for the LCA momentum-tracking family. Its proof proceeds in two steps.
First, we show that every method in this family, when applied to the mixing matrix constructed in Lemma~\ref{lem:mixing_matrix}, admits the representation in \eqref{eq:sparkle_plus_master_realization}. Second, we derive a uniform bound on the associated $\psi_{j,k}$, which allows us to apply Lemma~\ref{lem:sparkle_plus_master}.

\begin{theorem}
\label{thm:poly_lca_lb}
Fix an integer $D\geq1$. There exist an infinite set $\mathbb{N}_D$, positive constants $\theta_D^{\text{LCA}}$, $T_D^{\text{LCA}}$, and $C_D^{\text{LCA}}$ depending only on $D$ such that the following holds. 
For every $n\in\mathbb{N}_D$ and $\theta\in(0,\theta_D^{\text{LCA}}]$, let $W_{n,\theta}$ be the matrix constructed in Lemma~\ref{lem:mixing_matrix}. Consider the one-dimensional instance 
\begin{align}
    \label{eq:fi0}
    f_i(x)\equiv0,\qquad g_{i,t}(x)=\xi_{i,t},\qquad \xi_{i,t}\stackrel{\mathrm{i.i.d.}}{\sim}\mathcal N(0,\sigma^2),\qquad i=1,\ldots,n.
\end{align}
Then, for every $\alpha>0$, $\beta\in[0,1)$, every $\mathsf A\in\cA_D^{\mathrm{LCA}} (W_{n,\theta};\alpha,\beta)$ with $\x_0 = \mathbf{0}$, and every integer $T\geq T_D^{\text{LCA}}/\sqrt{\theta}$, the iterate generated by $\mathsf A$ satisfies
\begin{equation}
\label{eq:poly_lca_cons_lb}
    \E\brk{\norm{\Pi\x_T}^2} \geq \frac{C_D^{\text{LCA}}\alpha^2\sigma^2n}{(D+1)\sqrt{\theta}}
    \min\crk{1,\frac{(1-\beta)^2}{\theta}}.
\end{equation}
\end{theorem}

\begin{proof}
    See Appendix~\ref{app:poly_lca_lb}.
\end{proof}

The preceding result rules out improvements in consensus error obtained by retuning the LCA representation. The next theorem shows that the same network and stochastic instance are {also hard for} the SPARKLE family.

\begin{theorem}
\label{thm:sparkle_lb}
Fix an integer $D\geq1$. There exist an infinite set $\mathbb{N}_D$, positive constants $\theta_D^{\text{SP}}$, $T_D^{\text{SP}}$, and $C_D^{\text{SP}}$ depending only on $D$ such that the following holds. 
For every $n\in\mathbb{N}_D$ and $\theta\in(0,\theta_D^{\text{SP}}]$, let $W_{n,\theta}$ be the matrix constructed in Lemma~\ref{lem:mixing_matrix}. Consider the one-dimensional instance 
\begin{align*}
    f_i(x)\equiv0,\qquad g_{i,t}(x)=\xi_{i,t},\qquad \xi_{i,t}\stackrel{\mathrm{i.i.d.}}{\sim}\mathcal N(0,\sigma^2),\qquad i=1,\ldots,n.
\end{align*}
Then, for every $\alpha>0$, $\beta\in[0,1)$, every $\mathsf A\in\cA_D^{\text{SP}} (W_{n,\theta};\alpha,\beta)$ with zero initialization, and every integer $T\geq T_D^{\text{SP}}/\sqrt{\theta}$, the iterate generated by $\mathsf A$ satisfies
\begin{equation}
\label{eq:poly_sparkle_cons_lb}
    \E\brk{\norm{\Pi\x_T}^2} \geq \frac{C_D^{\text{SP}}\alpha^2\sigma^2n}{(D+1)\sqrt{\theta}}
    \min\crk{1,\frac{(1-\beta)^2}{\theta}}.
\end{equation}
\end{theorem}

\begin{proof}
    See Appendix~\ref{app:sparkle_lb}. 
\end{proof}

Theorems~\ref{thm:poly_lca_lb} and~\ref{thm:sparkle_lb} establish {consensus-error lower bounds} of the same order for the two families on the common mixing matrix $W_{n,\theta}$ and under the same stochastic oracle. Their conclusions therefore extend directly to the union of the two families. 

\begin{corollary}
    \label{cor:sparkle_plus_lb}
    Fix $D\geq1$, and let $\mathbb N_D$ be the infinite set in Lemma~\ref{lem:mixing_matrix}. Define $\theta_D^+:=\min\crki{\theta_D^{\mathrm{LCA}},\theta_D^{\mathrm{SP}}}$ and $T_D^+:=\max\crki{T_D^{\mathrm{LCA}},T_D^{\mathrm{SP}}}$. For every $n\in\mathbb N_D$, $\theta\in(0,\theta_D^+]$, $\alpha>0$, and $\beta\in[0,1)$, every $T\geq T_D^+/\sqrt{\theta}$, and every method $\mathsf A\in \cA_D^{\text{LCA}}(W_{n,\theta};\alpha,\beta)\cup \cA_D^{\text{SP}}(W_{n,\theta};\alpha,\beta)$ with zero initialization, applied to the instance \eqref{eq:fi0}, the iterate generated by $\mathsf A$ satisfies
    \begin{equation}
    \label{eq:cons_lb_thm}
        \E\brk{\norm{\Pi\x_T}^2} \geq  \frac{\min\crk{C_D^{\text{LCA}}, C_D^{\text{SP}}}\alpha^2\sigma^2n}{(D+1)\sqrt{\theta}}\min\crk{1,\frac{(1-\beta)^2}{\theta}}.
    \end{equation}
    \end{corollary}

    \begin{proof}
    If $\mathsf A\in\cA_D^{\text{LCA}}(W_{n,\theta};\alpha,\beta)$, we apply Theorem~\ref{thm:poly_lca_lb}; if $\mathsf A\in\cA_D^{\text{SP}}(W_{n,\theta};\alpha,\beta)$, we apply Theorem~\ref{thm:sparkle_lb}. In either case, the lower bound in \eqref{eq:cons_lb_thm} holds.
    \end{proof}

    Corollary~\ref{cor:sparkle_plus_lb} identifies the consensus-error lower bound shared by both families. For every fixed $D$, it rules out improving the worst-case order of this term either by admissibly retuning a method within the LCA momentum-tracking family or by enlarging the comparison class to include SPARKLE methods. 

\subsection{Minimax Lower Bounds for \DSMTL}

    {In this part, we first define the minimax performance over classes of objectives, stochastic oracles, networks, and algorithms. We then combine the consensus-error lower bound above with optimization lower bounds in \cite{yuan2022revisiting,lu2021optimal} to derive minimax lower bounds for \DSMTL~with $Q=1$.}

    Let $\cF_L$ denote the class of local objective collections $\{f_i\}_{i=1}^n$ {such that} each $f_i$ is $L$-smooth, the global objective function $f:=\sum_{i=1}^n f_i/n$ is bounded below, and $f(\bar{x}_0) - f^*\leq \Delta_f$, {where} $f^*=\inf_{x} f(x)$. We additionally assume the initial heterogeneity condition $\normi{\nabla F(\1\bar{x}_0^{\T})}^2=\orderi{ nL\Delta_f}$, which is implicit in \cite{yuan2022revisiting,lu2021optimal}.
    Let $\mathcal O_{\sigma^2}$ denote the corresponding class of stochastic oracles that are unbiased {estimators of} $\nabla f_i(x)$ {with} variance bounded by $\sigma^2$. 
    We define $\mathcal W_{n,\theta}$ as the class of symmetric, nonnegative, positive semidefinite, and doubly stochastic matrices $W\in\mathbb R^{n\times n}$ satisfying $\normi{\Pi W}_2=1-\theta$. 
    For the PL setting, let $\cF_{\mu,L}\subseteq\cF_L$ denote the subclass {for which the} global objective function satisfies the $\mu$-PL condition. 
    
    Regarding the algorithmic classes, for fixed $\alpha>0$ and $\beta\in[0,1)$, let $\mathfrak A_{D,n,\theta}^{\mathrm{LCA}}(\alpha,\beta)$ denote the class of algorithms $\mathsf A$ such that, for every $W\in\mathcal W_{n,\theta}$, the algorithm $\mathsf A$ applied to $W$ belongs to $\cA_D^{\mathrm{LCA}}(W;\alpha,\beta)$, as defined in \eqref{eq:LCA_alpha_beta}. Similarly, we define the algorithm class $\mathfrak A_{D,n,\theta}^{\mathrm{SP}}(\alpha,\beta)$. The overall algorithm class is then defined as their union:
    \begin{align*}
        \mathfrak A_{D,n,\theta}^{+}(\alpha,\beta):=
        \mathfrak A_{D,n,\theta}^{\mathrm{LCA}}(\alpha,\beta)
        \cup
        \mathfrak A_{D,n,\theta}^{\mathrm{SP}}(\alpha,\beta).
    \end{align*}
    
    For any algorithm $\mathsf A\in\mathfrak A_{D,n,\theta}^{+}(\alpha,\beta)$ and any mixing matrix $W\in\mathcal W_{n,\theta}$, let $\x_{0:T}^{\mathsf{A},W}:= (\x_0,\x_1,\ldots,\x_T)$ denote the iterates generated by $\mathsf A$ {with} $W$. Recall the definition of $\mathfrak{E}^{\text{(NCVX)}}\prti{\x_{0:(T-1)}}$ given in \eqref{def:ncvx_metric}. We define the corresponding nonconvex minimax performance as
    \begin{equation}
        \label{def:ncvx_lb}
        \begin{aligned}
            \mathfrak{R}_{D,T}^{\text{(NCVX)}}(n,\theta;\alpha,\beta):=\inf_{\mathsf{A}\in\mathfrak A_{D,n,\theta}^{+}(\alpha,\beta)}\sup_{W\in\mathcal{W}_{n,\theta}} \sup_{\crki{f_i}\subset \cF_{L}} \sup_{\crki{g_i}\subset\mathcal{O}_{\sigma^2}} \mathfrak{E}^{\text{(NCVX)}}\prt{\x_{0:(T-1)}^{\mathsf{A},W}}.
        \end{aligned}
    \end{equation}
    {For the PL setting}, we consider the following minimax formulation:
    \begin{equation}
        \label{def:pl_lb}
        \begin{aligned}
            &\mathfrak{R}_{D,T}^{(\text{PL})}(n,\theta;\alpha,\beta):=\inf_{\mathsf{A}\in\mathfrak A_{D,n,\theta}^{+}(\alpha,\beta)}\sup_{W\in\mathcal{W}_{n,\theta}}\sup_{\crki{f_i}\subset \cF_{\mu, L}}\sup_{\crki{g_i}\subset\mathcal{O}_{\sigma^2}} \mathfrak{E}^{(\text{PL})}\prt{\x_T^{\mathsf{A},W}},
        \end{aligned}
    \end{equation}
    where $\mathfrak{E}^{(\text{PL})}\prti{\x_T}$ is defined in \eqref{def:pl_metric}.

    In the proofs of the following theorems, the terms involving $\alpha$ and $\beta$ are established jointly through a coordinate-separable instance on the common hard network $W_{n,\theta}$. The remaining terms from \cite{yuan2022revisiting,lu2021optimal} require different hard networks. Since the minimax risk takes the supremum over $W\in\mathcal W_{n,\theta}$, these bounds can be combined {via the relation} $\max\{a,b\}\geq(a+b)/2$.

    Theorem~\ref{thm:sparkle_plus_ncvx_lb} establishes a minimax lower bound for the two comparison families in the smooth nonconvex setting. 

\begin{theorem}
\label{thm:sparkle_plus_ncvx_lb}
Fix $D\geq1$. There exist an infinite set $\mathbb{N}_D$ and constants $\theta_D^+ = \min\crki{\theta_D^{\text{LCA}}, \theta_D^{\text{SP}}}$ and $T_D^+ = \max\crki{T_D^{\text{LCA}}, T_D^{\text{SP}}}$ depending only on $D$, such that the following holds. For every $n\in\mathbb{N}_D$, $\theta\in(1-\cos(\pi/n),\theta_D^+)$, $\beta\in[0,1)$, $\alpha >0$, and $T\geq 2\ceili{\max\crki{T_D^+/\sqrt{\theta}, 1/(1-\beta)}}$, we have
\begin{equation}
\label{eq:sparkle_plus_ncvx_lb}
\begin{aligned}
    &\mathfrak{R}_{D,T}^{\text{(NCVX)}}(n,\theta;\alpha,\beta)= \Omega\prt{\sqrt{\frac{\sigma^2 L\Delta_f}{nT}}  + \min\crk{L\Delta_f, \frac{\Delta_f}{\alpha T}} + \min\crk{L\Delta_f, \frac{L\Delta_f}{(1-\beta)T}}\right.\\
    &\left. + \frac{L\Delta_f}{D\sqrt{\theta}T}  + \frac{\alpha L \sigma^2}{n}\min\crk{\frac{\alpha L}{1-\beta}, 1}  + \frac{\alpha^2 L^2\sigma^2}{(D+1)\sqrt{\theta}}\min\crk{1, \frac{(1-\beta)^2}{\theta}}},
\end{aligned}
\end{equation}
where $\Omega(\cdot)$ hides constants depending only on $D$.

\end{theorem}

\begin{proof}
    See Appendix~\ref{app:sparkle_plus_ncvx_lb}.
\end{proof}
\begin{remark}
    \label{rem:ncvx_matching}
    For fixed $D$ and horizon $T\geq 2\ceili{\max\crki{T_D^+/\sqrt{\theta}, 1/(1-\beta)}}$, consider momentum parameter $\beta=\beta(n,\theta,D)\in[0,1)$ that is independent of $T$, while allowing $\alpha>0$ to depend on $T$. Minimizing $\mathfrak{R}_{D,T}^{\text{(NCVX)}}(n,\theta;\alpha,\beta)$ over $\alpha>0$ gives 
    \begin{align*}
        \inf_{\alpha>0}\mathfrak{R}_{D,T}^{\text{(NCVX)}}(n,\theta;\alpha,\beta) &= \Omega\prt{\sqrt{\frac{L\Delta_f \sigma^2}{nT}} + \frac{L\Delta_f}{D\sqrt{\theta} T} + \frac{L\Delta_f}{(1-\beta)T} \right.\\
        &\quad\left.+ \min\crk{L\Delta_f, \brk{\frac{L^2\Delta_f^2 \sigma^2}{(D+1)\sqrt{\theta}T^2}\min\prt{1, \frac{(1-\beta)^2}{\theta}} }^{\frac{1}{3}}}}.
    \end{align*}
    Minimizing the {resulting} transient time over all $\beta\in[0,1)$ chosen independently of $T$ yields the lower bound
    \begin{align*}
        \Omega\prt{\frac{n^{5/3}}{1-\lambda}},
    \end{align*}
    which matches the upper bound in Remark~\ref{rem:ncvx_kt}.
\end{remark}

Theorem~\ref{thm:sparkle_plus_pl_lb} establishes a minimax lower bound for the two comparison families under the {additional} PL condition.

\begin{theorem}
\label{thm:sparkle_plus_pl_lb}
Fix $D\geq1$ and denote $\kappa = L/\mu$. There exist an infinite set $\mathbb{N}_D$ and constants $\theta_D^+ = \min\crki{\theta_D^{\text{LCA}}, \theta_D^{\text{SP}}}$ and $T_D^+ = \max\crki{T_D^{\text{LCA}}, T_D^{\text{SP}}}$ depending only on $D$, such that the following holds. For every $n\in\mathbb{N}_D$, $\theta\in(1-\cos(\pi/n),\theta_D^+)$, $\beta\in[0,1)$, $\alpha>0$, and $T\geq 2\ceili{\max\crki{T_D^+/\sqrt{\theta}, 1/(1-\beta)}}$, we have
\begin{align*}
    &\mathfrak{R}_{D,T}^{(\text{PL})}(n,\theta;\alpha,\beta) = \Omega\prt{\frac{\sigma^2}{\mu nT} + \frac{\Delta_f}{\kappa}\exp\prt{-C_DT \sqrt{\frac{\theta}{\kappa}}}+ \frac{\alpha \sigma^2}{n}\min\crk{\frac{\alpha L}{1-\beta}, 1}\right.\\
    &\left. + \frac{\alpha^2 L\sigma^2}{(D+1)\sqrt{\theta}}\min\crk{1, \frac{(1-\beta)^2}{\theta}} + \Delta_f \prt{\frac{\sqrt{\kappa} - 1}{\sqrt{\kappa} + 1}}^{2T}\right.\\
    &\left.  + \Delta_f \exp\prt{-8\alpha\mu T}\mathbb{I}\prt{\alpha\leq\frac{1-\beta}{8\mu}}},\quad \text{for some constant } C_D = \order{D},
\end{align*}
where $\mathbb{I}(\cdot)$ is the indicator function.

\end{theorem}

\begin{proof}
    See Appendix~\ref{app:sparkle_plus_pl_lb}.
\end{proof}

\begin{remark}
    \label{rem:pl_near_matching}
    For fixed $D$ and horizon $T\geq 2\ceili{\max\crki{T_D^+/\sqrt{\theta}, 1/(1-\beta)}}$, consider momentum parameter $\beta=\beta(n,\theta,D)\in[0,1)$ that is independent of $T$, while allowing $\alpha>0$ to depend on $T$. The term 
    \begin{align*}
        \Omega\prt{\frac{\alpha\sigma^2}{n}\min\crk{\frac{\alpha L}{1-\beta}, 1}}
    \end{align*}
    implies that every tuning $\alpha>0$ attaining $\torderi{\sigma^2/(nT)}$ must satisfy $\alpha\leq (1-\beta)/(8 L)$ for sufficiently large $T$. Hence, to derive the transient-time lower bound, it suffices to minimize $\mathfrak{R}_{D,T}^{(\text{PL})}(n,\theta;\alpha,\beta)$ over $0<\alpha\leq (1-\beta)/(8 L)$. {This yields} 
    \begin{align*}
    &\inf_{0<\alpha\leq(1-\beta)/(8 L)}
    \mathfrak R_{D,T}^{(\mathrm{PL})}(n,\theta;\alpha,\beta)=
    \tilde{\Omega}\prt{ \frac{\sigma^2}{\mu nT} +
    \frac{\Delta_f}{\kappa}
    \exp\prt{-C_DT\sqrt{\frac{\theta}{\kappa}}} \right.\\
    &\left. +
    \Delta_f
    \prt{\frac{\sqrt\kappa-1}{\sqrt\kappa+1}}^{2T}
    + \frac{\kappa\sigma^2} {\mu n(1-\beta)T^2}
    + \frac{\kappa\sigma^2} {\mu(D+1)\sqrt\theta\,T^2}
    \min\crk{1,\frac{(1-\beta)^2}{\theta}}
    },
    \end{align*}
    where $\tilde{\Omega}(\cdot)$ hides logarithmic factors. Minimizing the induced transient time over all $\beta\in[0,1)$ chosen independently of $T$ yields the transient-time lower bound
    \begin{align*}
        \tilde{\Omega}\prt{\frac{n^{1/3}}{\sqrt{1-\lambda}}},
    \end{align*}
    which matches the upper bound in Remark~\ref{rem:pl_kt} whenever $1/(1-\lambda) = \Omega(n^{4/3})$.
\end{remark}

}

\section{Numerical Examples}

This section validates our theoretical findings through two numerical examples over ring graphs ($n=100$ and $n=50$). We compare the performance of \DSMTL~with several related methods on two problems: a strongly convex logistic regression problem (satisfying the PL condition), and a nonconvex logistic regression problem.
Both problems are evaluated on the CIFAR-10  dataset \cite{krizhevsky2009learning} in the context of classifying airplane and truck images. We consider a \textit{heterogeneous} data setting, where data samples are sorted by class labels and partitioned among the agents without overlap. All reported results are averaged over $10$ independent trials.

\subsection{Logistic Regression}
    \label{subsec:sims_scvx}

    We first consider a binary classification problem using $\ell_2$-regularized logistic regression, as defined in \eqref{eq:logistic}. Each agent $i$ possesses a distinct local dataset $\mathcal{S}_i{= \crki{(u_j, v_j)}}$ selected from the global dataset $\mathcal{S}$. Here, $u_j\in\R^p$ denotes the image input and $v_j\in\R$ represents the label. The classifier is obtained by solving the following optimization problem:

    \begin{equation}
        \label{eq:logistic}
        \begin{aligned}
            &\min_{x\in\R^{p}} f(x) = \frac{1}{n}\sum_{i=1}^n f_i(x),\;f_i(x) := \frac{1}{|\mathcal{S}_i|} \sum_{j\in\mathcal{S}_i} \log\left[1 + \exp(-x^{\T}u_jv_j)\right] + \frac{\rho}{2}\norm{x}^2,
        \end{aligned}
    \end{equation}
    where $\rho$ is set to $0.2$. 

    We compare \DSMTL~against several distributed stochastic gradient methods with local updates: K-GT \cite{liu2024decentralized}, LED \cite{alghunaim2024local}, PD-SGDM \cite{gao2020periodic}, and Local DSGD \cite{koloskova2020unified}. We also include SCAFFOLD \cite{karimireddy2020scaffold} as a centralized baseline to represent the best possible performance if a parameter server exists. To ensure a fair comparison, all methods use the same ``effective'' stepsize\footnote{Specifically, the stepsize is set such that the averaged iterates $\crki{\bar{x}_t}$ follow the same update $\bar{x}_{t + 1} = \bar{x}_t - \teta \bar{g}_t$ at the $t$-th communication round.}. We evaluate each method using both $Q = 10$ and $Q=20$ local updates to investigate the impact of increased local computation.

    Fig.~\ref{fig:logistic} illustrates the performance of all methods. Under both network sizes, \DSMTL~consistently outperforms the other distributed algorithms and exhibits performance comparable to that of SCAFFOLD. This advantage is particularly pronounced on the larger, poorer-connected network (Fig.~\ref{fig:logistic_n100} and Fig.~\ref{fig:logistic_Q20_n100}), confirming the superior performance of the proposed method. 
    
    Furthermore, comparing the results in Fig.~\ref{fig:logistic} for $Q=10$ (top row) and $Q=20$ (bottom row) shows that increasing the number of local updates improves the final accuracy across all methods. This observation aligns with the theoretical result in \eqref{eq:pl_order}, which indicates that the final error is inversely proportional to  the number of local updates $Q$.

    \begin{figure}[htbp]
        \centering
        \subfloat[Ring graph, $n=50$, $Q=10$, $1-\lambda = 2.6\times 10^{-3} $.]{\includegraphics[width=0.49\textwidth]{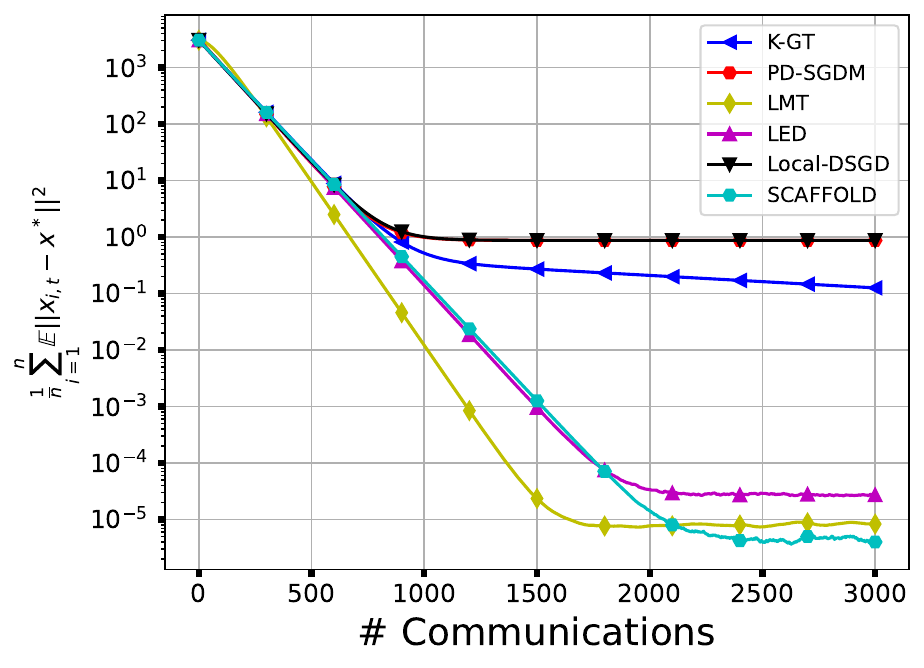}\label{fig:logistic_n50}}
        \subfloat[Ring graph, $n=100$, $Q = 10$, $1-\lambda = 6.6\times10^{-4}$.]{\includegraphics[width=0.49\textwidth]{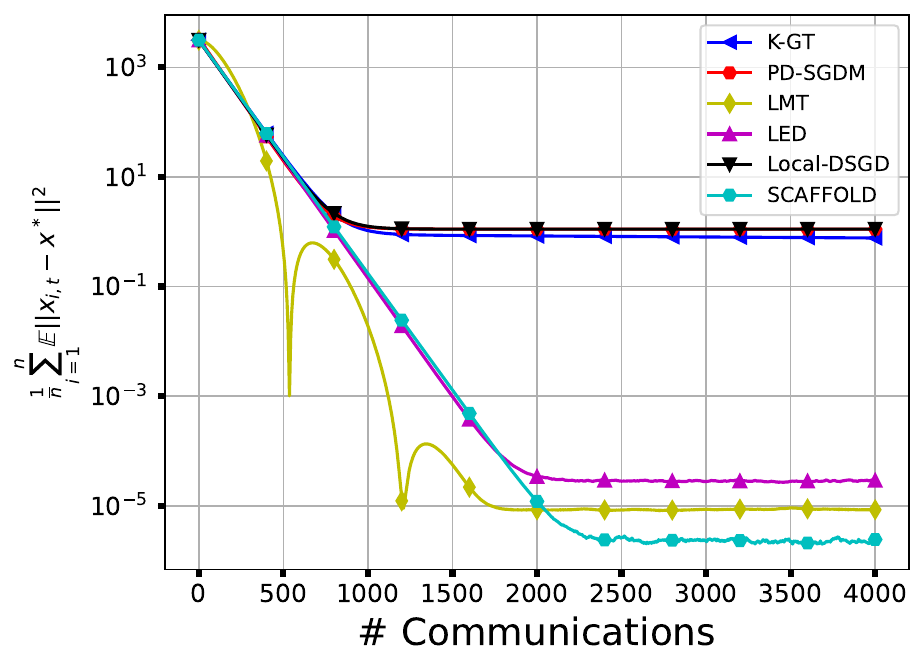}\label{fig:logistic_n100}}\\
        \subfloat[Ring graph, $n=50$, $Q=20$, $1-\lambda = 2.6\times 10^{-3}$.]{\includegraphics[width=0.49\textwidth]{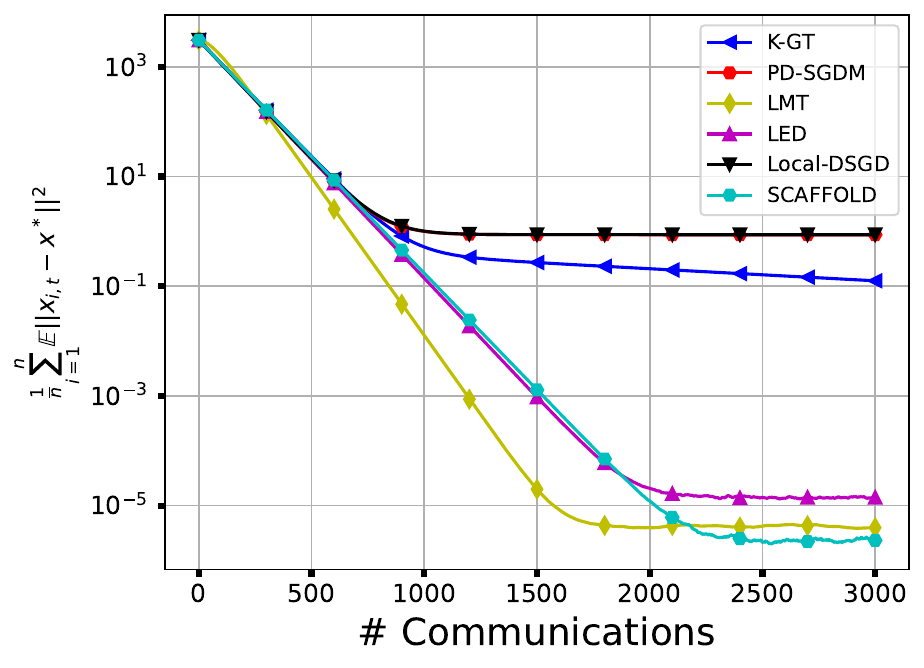}\label{fig:logistic_Q20_n50}}
        \subfloat[Ring graph, $n=100$, $Q=20$, $1-\lambda = 6.6\times10^{-4}$.]{\includegraphics[width=0.49\textwidth]{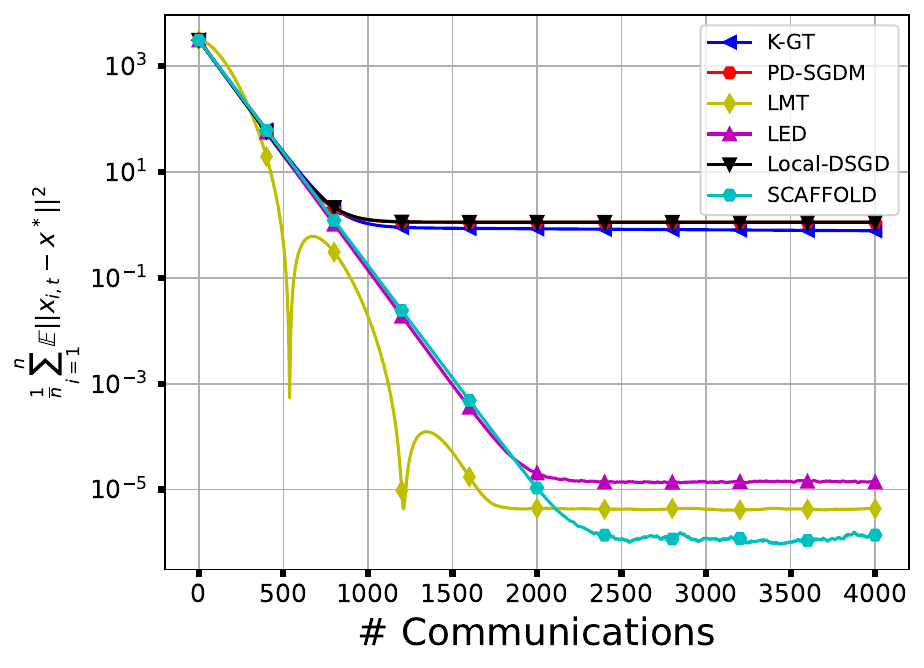}\label{fig:logistic_Q20_n100}}
        \caption{Comparison among \DSMTL, K-GT, LED, PD-SGDM, Local DSGD, and SCAFFOLD for solving Problem \eqref{eq:logistic} on the CIFAR-10 dataset using a constant stepsize. 
        For \DSMTL, K-GT, and SCAFFOLD, the stepsizes are set to $\eta_a = 0.25/Q$ (local updates) and $\eta_s = 0.1$ (outer loop). For LED, the stepsize is $\eta_a\eta_s$, and for PD-SGDM it is $\eta_a\eta_s(1-\beta)$.
        The momentum parameter is set to $\beta =\trw$ for \DSMTL~and PD-SGDM.}
        \label{fig:logistic}
    \end{figure}

    \subsection{Nonconvex Logistic Regression}
    We next consider the logistic regression problem with a nonconvex regularizer, defined in \eqref{eq:ncvx_logistic}. The parameter $\omega$ is set to $0.05$, and $[x]_q$ denotes the $q$-th element of $x\in\R^p$. All other experimental settings are identical to those in Subsection \ref{subsec:sims_scvx}. 
    \begin{equation}
		\label{eq:ncvx_logistic}
		\begin{aligned}
			& \min_{x\in\R^{p}} f(x) = \frac{1}{n}\sum_{i=1}^n f_i(x),\;f_i(x) := \frac{1}{|\mathcal{S}_i|} \sum_{j\in\mathcal{S}_i} \log\left[1 + \exp(-x^{\T}u_jv_j)\right] + \frac{\omega}{2}\sum_{q=1}^p \frac{[x]_q^2}{1 + [x]_q^2}.
		\end{aligned}
	\end{equation}

    \begin{figure}[htbp]
        \centering
        \subfloat[Ring graph, $n=50$, $Q=10$, $1-\lambda = 2.6\times 10^{-3}$.]{\includegraphics[width=0.49\textwidth]{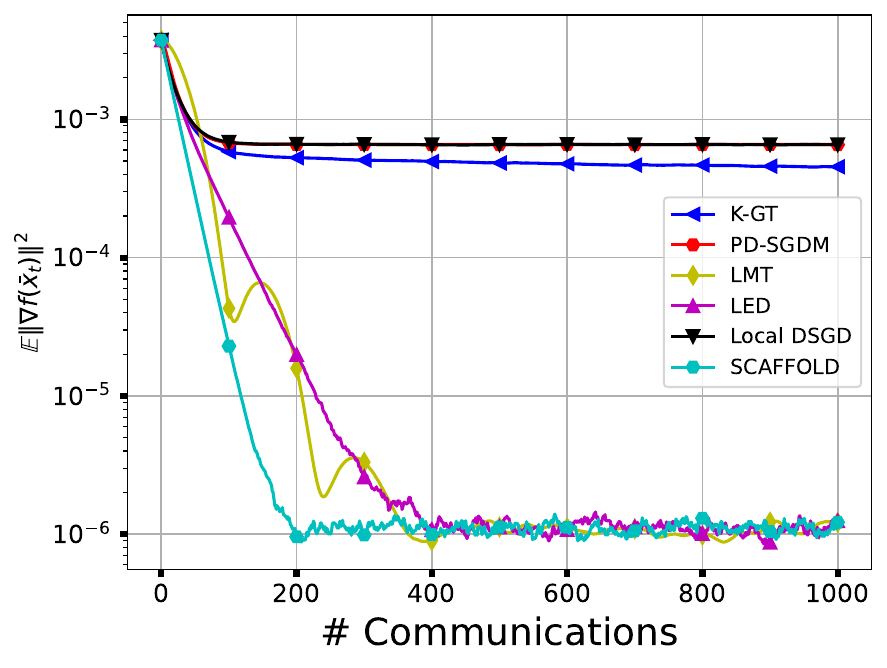}\label{fig:ncvx_logistic_n50}}
        \subfloat[Ring graph, $n=100$, $Q=10$, $1-\lambda = 6.6\times 10^{-4}$.]{\includegraphics[width=0.49\textwidth]{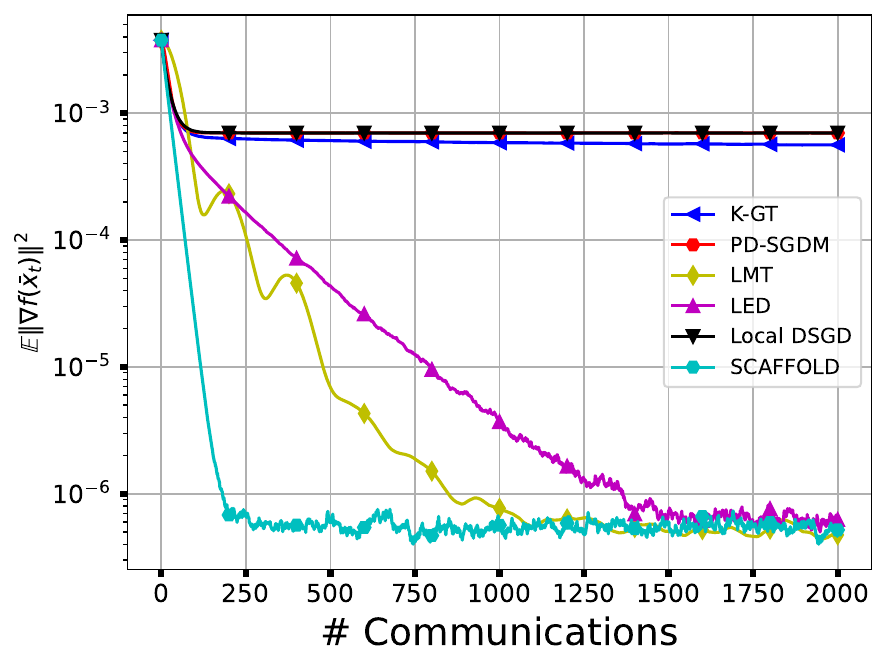}\label{fig:ncvx_logistic_n100}}\\
        \subfloat[Ring graph, $n=50$, $Q=20$, $1-\lambda = 2.6\times 10^{-3}$.]{\includegraphics[width=0.49\textwidth]{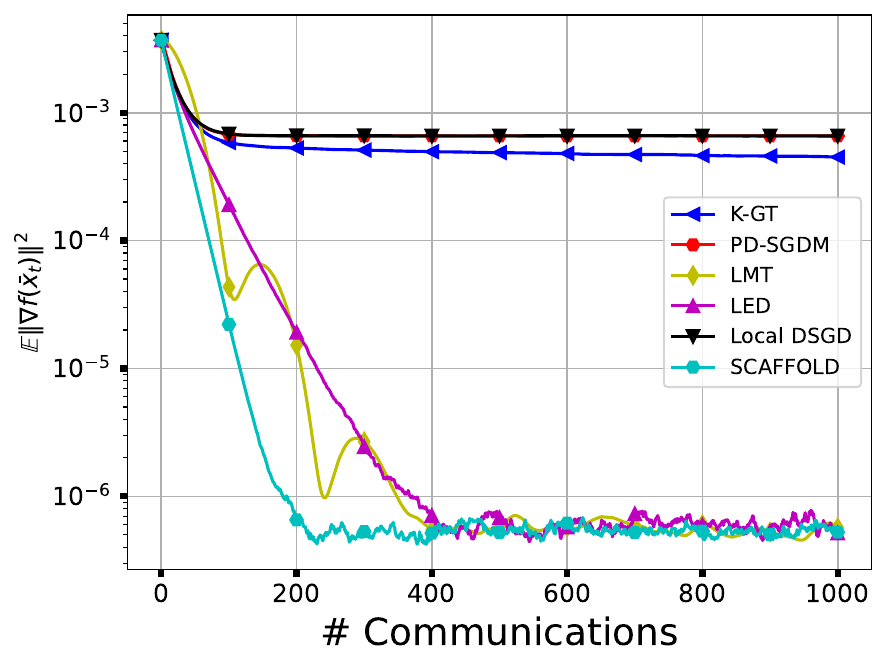}\label{fig:ncvx_logistic_Q20_n50}}
        \subfloat[Ring graph, $n=100$, $Q=20$, $1-\lambda = 6.6\times 10^{-4}$.]{\includegraphics[width=0.49\textwidth]{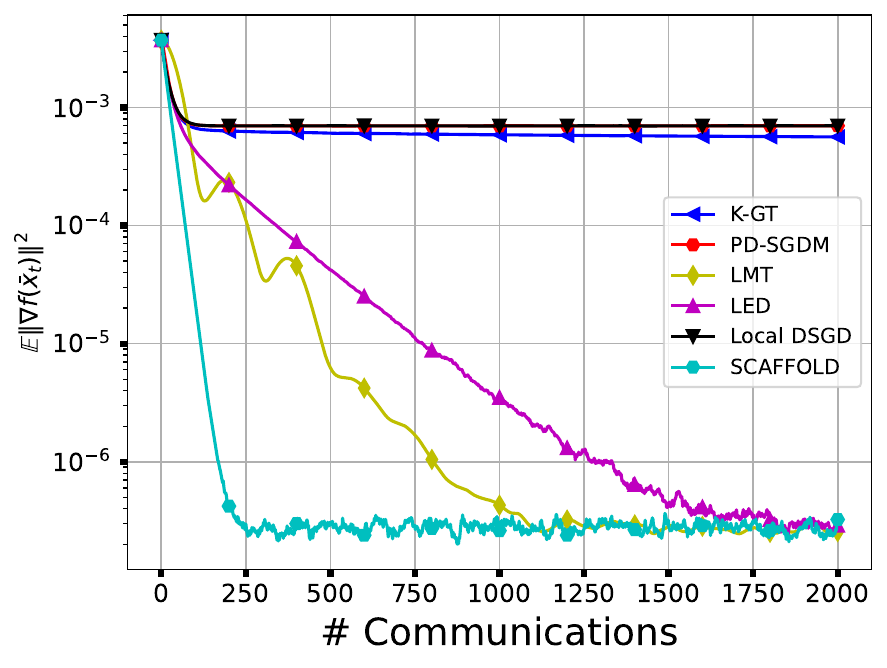}\label{fig:ncvx_logistic_Q20_n100}}
        \caption{Comparison among \DSMTL, K-GT, LED, PD-SGDM, Local DSGD, and SCAFFOLD for solving Problem \eqref{eq:ncvx_logistic} on the CIFAR-10 dataset using a constant stepsize. 
        For \DSMTL, K-GT, and SCAFFOLD, the stepsizes are set to $\eta_a = 0.5/Q$ (local updates) and $\eta_s = 1$ (outer loop). For LED, the stepsize is $\eta_a\eta_s$, and for PD-SGDM it is $\eta_a\eta_s(1-\beta)$. The momentum parameter is set to $\beta =\trw$ for \DSMTL~and PD-SGDM.}
        \label{fig:ncvx_logistic}
    \end{figure}

    The results shown in Fig.~\ref{fig:ncvx_logistic} reinforce our findings from the strongly convex case. \DSMTL~achieves the fastest convergence among all the distributed methods, especially on the large, more poorly connected network (Fig.~\ref{fig:ncvx_logistic_n100} and Fig.~\ref{fig:ncvx_logistic_Q20_n100}). With sufficiently many communication rounds, \DSMTL~and LED attain final accuracies close to that of SCAFFOLD, though LMT consistently converges faster. Increasing the number of local updates from $Q = 10$ to $Q=20$ further reduces the final error. This phenomenon corroborates the theoretical findings in \eqref{eq:ncvx_teta}, which indicates that the final error is inversely proportional to the number of local updates $Q$.

\section{Conclusion}
This paper addresses the distributed stochastic optimization problem over networks with locally updated agents. The proposed Local Momentum Tracking (\DSMTL) algorithm achieves optimal communication complexity with sufficiently many local updates {while improving the iteration complexity over existing methods. With a single local update, we establish minimax lower bounds
over the LCA momentum-tracking and SPARKLE families. These lower bounds match the transient time of \DSMTL\ in the smooth nonconvex setting and, under the PL condition, nearly match the transient time of \DSMTL\ in its dependence on the network spectral gap.}
Experimental results corroborate the theoretical findings.

\section{Statements and Declarations}
The authors declare that there is no conflict of interest. The CIFAR-10 dataset is online at \href{https://www.cs.toronto.edu/\~kriz/cifar.html}{https://www.cs.toronto.edu/$\sim$kriz/cifar.html}.

\section{Acknowledgments}
This work is supported by the National Natural Science Foundation of China (NSFC) under grant No. 62373316.
\backmatter

\begin{appendices}

\section{Proofs}
\renewcommand*{\theHequation}{app.\Alph{section}.\arabic{equation}}
\subsection{Proof of Lemma \ref{lem:avg}}
\label{app:avg}

The update rule for $\bar{x}_t$ in \eqref{eq:xt_avg} follows directly from the analysis in \cite[Lemma 2]{huang2025accelerated}, given the compact form \eqref{eq:txp1}. We now prove by induction that $\bar{y}_t = \bar{y}_t^{(l)} = \bar{z}_{t + 1}$ for any $t\geq 0$. 

For the base case ($t=0$), we have $\bar{y}_0 = \bar{z}_1 + \bar{c}_0 = \bar{z}_1$ from \eqref{eq:yt} and the initialization $\bc_0 = \mathbf{0}$. Similarly, we have $\bar{y}_0^{(l)} = \bar{z}_1$ from \eqref{eq:ytl} and the initialization $\bc_{-1} = \mathbf{0}$. 

Now, suppose $\bar{y}_t = \bar{y}_t^{(l)} = \bar{z}_{t + 1}$ for $t\geq 0$. 
From the update rule \eqref{eq:typ1}, we obtain
    \begin{equation}
        \label{eq:ytbar_s1}
        \begin{aligned}
            \bar{y}_{t + 1} &= (1 + \eta_w)\bar{y}_t - \eta_w\bar{y}_t^{(l)} + \bar{z}_{t + 2} - \bar{z}_{t + 1},\quad 
            \bar{y}_{t + 1}^{(l)} = \bar{y}_t + \bar{z}_{t + 2} - \bar{z}_{t + 1}.
        \end{aligned}
    \end{equation}
    Using the induction hypothesis that $\bar{y}_t = \bar{y}_t^{(l)}$, we have $\bar{y}_{t + 1} = \bar{y}_t + \bar{z}_{t + 2} - \bar{z}_{t + 1} = \bar{y}_{t+1}^{(l)}$.
    Therefore, $\bar{y}_{t+1} = \bar{y}_{t + 1}^{(l)} = \bar{y}_0 + \bar{z}_{t + 2} - \bar{z}_1 = \bar{z}_{t + 2}$, where we used $\bar{y}_0 = \bar{z}_1$. 
    Hence, $\bar{y}_t = \bar{y}_t^{(l)} = \bar{z}_{t + 1}$ for any $t\geq 0$.

\subsection{Proof of Lemma \ref{lem:dt}}
\label{app:dt}
For $t\geq 1$, we have 
        \begin{align*}
            \bar{d}_{t + 1} - \bar{d}_t &= \frac{1}{1-\beta}\brk{\bar{x}_{t + 1} - \beta\bar{x}_t - \bar{x}_t + \beta \bar{x}_{t - 1}} = \frac{1}{1-\beta}\prt{-\teta \bar{z}_{t + 1} + \beta\teta \bar{z}_t} = -\frac{\teta}{Q}\sum_{\ell=0}^{Q-1}\bar{g}_t^\ell,
        \end{align*}
    where we invoked \eqref{eq:mot_dsmtl_localz} and \eqref{eq:xt_avg}.
    For $t = 0$, we have 
    \begin{equation}
        \label{eq:d0}
        \begin{aligned}
            \bar{d}_1 - \bar{d}_0 &= \frac{1}{1-\beta}\brk{\bar{x}_1 - \beta\bar{x}_0 - (1-\beta)\bar{x}_0} = -\frac{\teta}{1-\beta}\bar{z}_1 = -\frac{\teta}{Q}\sum_{\ell=0}^{Q-1}\bar{g}_0^\ell,
        \end{aligned}
    \end{equation}
    where we invoked $\z_0 = \mathbf{0}$.

    We next show \eqref{eq:dt_xt}. For $t\geq 1$, $(1-\beta)(\bar{d}_t - \bar{x}_t) =\bar{x}_t - \beta\bar{x}_{t-1} - (1-\beta)\bar{x}_t = -{\teta\beta} \bar{z}_t.$
    Since $\z_0 = \mathbf{0}$, the above derivation also holds for $t = 0$.

\subsection{Proof of Lemma \ref{lem:descent}}
\label{app:descent}

Let $\bs\xi_t^\ell:= \prti{\xi_{1,t}^\ell,\xi_{2,t}^\ell,\ldots, \xi_{n,t}^\ell }^{\T}$.
We define the filtration $\crki{\cF_t^\ell}$ generated by the history of stochastic gradients up to the $(\ell-1)$-th local step of communication round $t$, as
\[
    \cF_t^\ell:= 
    \sigma \prt{
    \crki{\bs\xi_k^s:0\leq k<t,\ 0\leq s<Q}
    \cup
    \crki{\bs\xi_t^s:0\leq s<\ell}
    },
    \qquad t\geq0,\quad 0\leq\ell<Q,
\]
with $\cF_0^0:= \crki{\Omega,\phi}$ being the trivial $\sigma$-algebra, where $\Omega$ denotes the sample space, and $\phi$ is the empty set.


Define the stochastic gradient errors as $\Delta_{i,t}^s:= g_{i,t}^s - \nabla f_i(x_{i,t}^s)$ and $\bar{\Delta}_t^s:= \bar{g}_t^s - \sumn \nabla f_i(x_{i,t}^s)/n$. Recall that $\bavg_t = \sum_{\ell=0}^{Q-1}\sumn \nabla f_i(x_{i,t}^\ell)/(nQ)$. 
Applying the descent lemma to the update rule for $\bar{d}_t$ in \eqref{eq:dt_update} yields 
    \begin{equation}
        \label{eq:descent_s1}
        \begin{aligned}
            &\condE{f(\bar{d}_{t + 1})}{\cF_t^0} \leq f(\bar{d}_t) - \teta \condE{\inpro{\nabla f(\bar{d}_t), \frac{1}{Q}\sum_{\ell=0}^{Q-1}\bar{g}_t^\ell}}{\cF_t^0} + \frac{L}{2}\condE{\norm{\frac{\teta}{Q}\sum_{\ell=0}^{Q-1}\bar{g}_t^\ell}^2}{\cF_t^0}\\
            &\leq f(\bar{d}_t) -\teta\condE{\inpro{\nabla f(\bar{d}_t), \bavg_t}}{\cF_t^0} + L\teta^2\condE{\norm{\frac{1}{nQ}\sum_{\ell=0}^{Q-1}\sumn \Delta_{i,t}^\ell}^2}{\cF_t^0}\\
            &\quad  + L\teta^2\condE{\norm{\bavg_t}^2}{\cF_t^0},
        \end{aligned}
    \end{equation}
    where we invoked Assumption~\ref{as:abc} and the tower property:
        \begin{align*}
            &\condE{\inpro{\nabla f(\bar{d}_t), \frac{1}{Q}\sum_{\ell=0}^{Q-1}\bar{g}_t^\ell}}{\cF_t^0} = \inpro{\nabla f(\bar{d}_t), \frac{1}{Q}\sum_{\ell=0}^{Q-1}\condE{\bar{g}_t^\ell}{\cF_t^0}} \\
            &= \inpro{\nabla f(\bar{d}_t), \frac{1}{Q}\condE{\bar{g}_t^0}{\cF_t^0} + \frac{1}{Q}\sum_{s=1}^{Q-1}\condE{\condE{\bar{g}_t^s}{\cF_t^{s}}}{\cF_t^0}} = \condE{\inpro{\nabla f(\bar{d}_t), \bavg_t}}{\cF_t^0}.
        \end{align*}

    We now bound the inner product term in \eqref{eq:descent_s1},
        \begin{align}
            &-\teta\condE{\inpro{\nabla f(\bar{d}_t) , \bavg_t}}{\cF_t^0}=-\teta\condE{\inpro{\nabla f(\bar{d}_t) - \nabla f(\bar{x}_t) + \nabla f(\bar{x}_t), \bavg_t}}{\cF_t^0} \nonumber\\
            &\leq \teta\norm{\nabla f(\bar{d}_t) - \nabla f(\bar{x}_t)}^2   + \frac{\teta}{4}\condE{\norm{\bavg_t}^2}{\cF_t^0} - \frac{\teta}{2}\norm{\nabla f(\bar{x}_t)}^2 - \frac{\teta}{2}\condE{\norm{\bavg_t}^2}{\cF_t^0}\nonumber \\
            &\quad + \frac{\teta}{2}\condE{\norm{\bavg_t - \nabla f(\bar{x}_t)}^2}{\cF_t^0}\nonumber\\
            &\leq \frac{\teta^3L^2\beta^2}{(1-\beta)^2}\norm{\bar{z}_t}^2 - \frac{\teta}{2}\norm{\nabla f(\bar{x}_t)}^2 - \frac{\teta}{4} \condE{\norm{\bavg_t}^2}{\cF_t^0}+ \frac{\teta L^2}{2nQ}\sum_{\ell=0}^{Q-1}\condE{\norm{\x_{t}^\ell - \1\bar{x}_t^{\T}}^2}{\cF_t^0}.\label{eq:descent_s2}
        \end{align}


    We then bound the third term on the right-hand side of \eqref{eq:descent_s1},
    \begin{equation}
        \label{eq:avg_var}
        \begin{aligned}
            \condE{\norm{\sum_{\ell=0}^{Q-1}\bar{\Delta}_t^\ell}^2}{\cF_t^0}&= \sum_{\ell=0}^{Q-1}\condE{\norm{\bar{\Delta}_t^\ell}^2}{\cF_t^0} + 2\sum_{0\leq p<q\leq Q-1} \condE{\inpro{\bar{\Delta}_t^p, \bar{\Delta}_t^q}}{\cF_t^0}\\
            &= \sum_{\ell=0}^{Q-1}\condE{\norm{\bar{\Delta}_t^\ell}^2}{\cF_t^0} \leq \frac{Q\sigma^2}{n},
        \end{aligned}
    \end{equation}
    where the last equality and the last inequality follow from Assumption~\ref{as:abc}, which implies that $ \condEi{\inproi{\bar{\Delta}_t^p, \bar{\Delta}_t^q}}{\cF_t^0} = \condEi{\condE{\inproi{\bar{\Delta}_t^p, \bar{\Delta}_t^q}}{\cF_t^q}}{\cF_t^0} = \condEi{\inproi{\bar{\Delta}_t^p, \condEi{\bar{\Delta}_t^q}{\cF_t^q}}}{\cF_t^0} = 0,$
    for any $1\leq p<q\leq Q-1$, and $\condE{\inpro{\Delta_{i,t}^s, \Delta_{j,t}^s}}{\cF_t^0}=\condE{\condE{\inpro{\Delta_{i,t}^s, \Delta_{j,t}^s}}{\cF_t^s}}{\cF_t^0} = \condE{\inpro{\condE{\Delta_{i,t}^s}{\cF_t^s}, \condE{\Delta_{j,t}^s}{\cF_t^s}}}{\cF_t^0} = 0,$
    for any $0\leq s\leq Q-1$ and $i\neq j$.

    Substituting the bounds from \eqref{eq:descent_s2} and \eqref{eq:avg_var} into \eqref{eq:descent_s1}, and then taking the full expectation, yields the desired result \eqref{eq:descent}.

\subsection{Proof of Lemma \ref{lem:barz}}
\label{app:barz}
    From \eqref{eq:mot_dsmtl_localz}, we obtain $\bar{z}_{t + 1} = \beta \bar{z}_t + (1-\beta)\sum_{\ell=0}^{Q-1}\brki{\bar{\Delta}_t^\ell + \sumn\nabla f_i(x_{i,t}^\ell)/n}/Q.$
    Note that $\condEi{\inproi{\bar{z}_t, \sum_{\ell=0}^{Q-1}\bar{\Delta}_t^\ell}}{\cF_t^0} = 0.$
    Applying Young's inequality, we have
    \begin{equation}
        \label{eq:barz_s2}
        \begin{aligned}
            \condE{\norm{\bar{z}_{t + 1}}^2}{\cF_t^0} &\leq \beta \norm{\bar{z}_t}^2 + 2\condE{\norm{\frac{1-\beta}{Q}\sum_{\ell=0}^{Q-1}\bar{\Delta}_t^\ell}^2}{\cF_t^0}+ 3(1-\beta)\condE{\norm{\bavg_t}^2}{\cF_t^0}.
        \end{aligned}
    \end{equation}
    Invoking \eqref{eq:avg_var} and taking the full expectation on both sides of \eqref{eq:barz_s2} yields the desired result.

\subsection{Proof of Lemma \ref{lem:tPitx}}
\label{app:tPitx}
    From \eqref{eq:txp1} and the relation $\tPi\tW = \tPi\tW\tPi$, we have $\tPi\tx_{t + 1} = \tPi\tW\tPi \brki{\tx_t - \teta\prt{\y_t}_{\#}}$.
    { Consequently, we have $\tPi\tx_t = \tPi\tW^t\tPi\tx_0 - \teta\sum_{s=0}^{t-1}\tPi\tW^{t-s}\tPi\brki{\y_s}_{\#}$,
    where we invoked $\tPi \cdot \tPi = \tPi$. By Young's inequality \cite[Lemma 2]{song2021optimal}, we choose the weights as 
    \begin{align*}
        a_0 = \trw^t, \quad a_s = (1-\trw)\trw^{t-s}, s = 1, 2, \ldots, t,
    \end{align*}
    and obtain 
    \begin{equation}
        \label{eq:tPix_tPix0_norm}
        \begin{aligned}
            \E\brk{\norm{\tPi\tx_t}^2} &\leq \frac{c_0\trw^{2t}}{a_0}\norm{\Pi\x_0}^2 + \sum_{s=0}^{t-1}\frac{\teta^2 c_0\trw^{2(t-s)}}{a_{s+1}}\E\brk{\norm{\Pi\y_s}^2}\\
            & = \trw^t c_0\norm{\Pi\x_0}^2 + \frac{c_0\trw^2 \teta^2}{1-\trw}\sum_{s=0}^{t-1}\trw^{t-1-s}\E\brk{\norm{\Pi\y_s}^2} = \cR_t^x,
        \end{aligned}
    \end{equation}
    where we invoked Lemma~\ref{lem:lca}.
    }

    Finally, the relation $\E\brki{\normi{\Pi\x_t}^2}\leq \cR_t^x$ in \eqref{eq:Pix2cRx} follows from \eqref{eq:tPix_tPix0_norm} together with $\normi{\Pi\x_t}^2\leq \normi{\tPi\tx_t}^2$ for all $t\geq 0$.

\subsection{Proof of Lemma \ref{lem:tPi_cRy}}
\label{app:tPi_cRy}

    In light of \eqref{eq:mot_dsmtl_localz}, we obtain
    \begin{equation}
        \label{eq:ztp2}
        \begin{aligned}
            &\z_{t + 2} - \z_{t + 1} 
            = -(1-\beta)\z_{t + 1} + \frac{1-\beta}{Q}\sum_{\ell=0}^{Q-1}\g_{t+1}^\ell\\
            &= -(1-\beta)\brk{\z_{t + 1} - \nabla F(\1\bar{x}_{t + 1}^{\T})} + \frac{1-\beta}{Q}\sum_{\ell=0}^{Q-1}\brk{\Delta_{t + 1}^\ell + \nabla F(\x_{t + 1}^\ell) -  \nabla F(\1\bar{x}_{t + 1}^{\T}) }.
        \end{aligned}
    \end{equation}
    Combining \eqref{eq:typ1}, \eqref{eq:ztp2}, and the relation $\tPi\tW=\tPi\tW\tPi$, we have
    \begin{equation}
        \label{eq:tPiy_s1}
        \begin{aligned}
            \tPi\ty_{t + 1} 
            &= \tPi\tW\tPi\ty_t -(1-\beta)\tPi\brk{\z_{t + 1} - \nabla F(\1\bar{x}_{t + 1}^{\T})}_{\#} \\
            &+ \frac{1-\beta}{Q}\sum_{\ell=0}^{Q-1}\tPi\brk{\bs{\Delta}_{t + 1}^\ell + \nabla F(\x_{t  + 1}^\ell) - \nabla F(\1\bar{x}_{t + 1}^{\T})}_{\#}.
        \end{aligned}
    \end{equation}
    Noting that $\condEi{\inproi{\tPi\tW\tPi\ty_t, \sum_{\ell=0}^{Q-1}\tPi\brki{\bs{\Delta}_{t + 1}^\ell}_{\#}}}{\cF_{t+1}^0} = 0,$
    we have from Young's inequality that 
    \begin{equation}
        \label{eq:tPiy_s2}
        \begin{aligned}
            &\condE{\norm{\tPi\ty_{t + 1}}^2}{\cF_{t + 1}^0} \leq \frac{1}{\trw}\norm{\tPi\tW\tPi\ty_t}^2 + \frac{5(1-\beta)^2}{Q^2}\condE{\norm{\sum_{\ell=0}^{Q-1}\bs{\Delta}_{t + 1}^\ell}^2}{\cF_{t + 1}^0}\\
            & + \frac{10(1-\beta)^2}{1-\trw}\norm{\z_{t + 1} - \nabla F(\1\bar{x}_{t + 1}^{\T})}^2 + \frac{10(1-\beta)^2 L^2}{Q(1-\trw)}\sum_{\ell=0}^{Q-1}\condE{\norm{\x_{t + 1}^\ell - \1\bar{x}_{t + 1}^{\T}}^2}{\cF_{t + 1}^0}\\
            &\leq \frac{1}{\trw}\norm{\tPi\tW\tPi\ty_t}^2 + \frac{5c_0(1-\beta)^2 n\sigma^2}{Q}+ \frac{5c_0(1-\beta)^2}{1-\trw}\norm{\z_{t + 1} - \nabla F(\1\bar{x}_{t + 1}^{\T})}^2\\
            &+ \frac{5c_0(1-\beta)^2 L^2}{Q(1-\trw)}\sum_{\ell=0}^{Q-1}\condE{\norm{\x_{t + 1}^\ell - \1\bar{x}_{t + 1}^{\T}}^2}{\cF_{t + 1}^0},
        \end{aligned}
    \end{equation}
    where we invoked the fact that $ \normi{\tPi A_{\#}}^2 = 2\normi{\Pi A}^2$, and $c_0 = 14$.
    For the term $\frac{1}{\trw}\normi{\tPi\tW\tPi\ty_t}^2$, in light of Lemma~\ref{lem:lca}, we obtain
        \begin{align*}
           &\frac{1}{\trw} \E\brk{\norm{\tPi\tW\tPi\ty_t}^2}\leq \frac{1}{\trw^2}\E\brk{\norm{\tPi\tW^2\tPi\ty_{t-1}}^2} + \frac{5c_0\trw (1-\beta)^2n\sigma^2}{Q}
           \\
           &\quad + \frac{5c_0\trw(1-\beta)^2}{1-\trw}\E\brk{\norm{\z_{t} - \nabla F(\1\bar{x}_{t}^{\T})}^2} + \frac{5c_0\trw(1-\beta)^2 L^2}{Q(1-\trw)}\sum_{\ell=0}^{Q-1}\E\brk{\norm{\x_{t}^\ell - \1\bar{x}_{t}^{\T}}^2}.
        \end{align*}
    Repeating the above procedures and invoking Lemma \ref{lem:lca} yields for $t\geq 1$ that 
        \begin{align}
            &\E\brk{\norm{\tPi\ty_t}^2} \leq c_0\trw^t\E\brk{\norm{\Pi\y_0}^2} + \frac{5c_0(1-\beta)^2}{1-\trw}\sum_{s=1}^{t}\trw^{t-s}\E\brk{\norm{\z_{s} - \nabla F(\1\bar{x}_{s}^{\T})}^2}  \nonumber \\
            &\quad + \frac{5c_0(1-\beta)^2n\sigma^2}{Q}\sum_{s=0}^{t-1}\trw^{t-1-s} + \frac{5c_0(1-\beta)^2 L^2}{Q(1-\trw)}\sum_{s=1}^{t}\trw^{t-s}\sum_{\ell=0}^{Q-1}\E\brk{\norm{\x_{s}^\ell - \1\bar{x}_{s}^{\T}}^2}=\cR_t^y, \label{eq:tPiy_s4}
        \end{align}
    where the equality follows from \eqref{eq:cRy}.
    
    Finally, the relation $\E\brki{\normi{\Pi\y_t}^2}\leq \cR_t^y$ in \eqref{eq:Piy2cRy} follows from \eqref{eq:tPiy_s4} together with $\normi{\Pi\y_t}^2\leq \normi{\tPi\ty_t}^2$ for all $t\geq 0$.

\subsection{Proof of Lemma \ref{lem:ztf}}
\label{app:ztf}
    From \eqref{eq:mot_dsmtl_localz}, we obtain
    \begin{equation}
        \label{eq:ztf_s1}
        \begin{aligned}
            \z_{t + 1} - \nabla F(\1\bar{x}_{t + 1}^{\T})&= \beta\brk{\z_t - \nabla F(\1\bar{x}_t^{\T})} + \frac{1-\beta}{Q}\sum_{\ell=0}^{Q-1}\brk{\boldsymbol{\Delta}_t^\ell + \nabla F(\x_t^\ell) - \nabla F(\1\bar{x}_t^{\T})}\\
            &\quad + \nabla F(\1\bar{x}_t^{\T}) - \nabla F(\1\bar{x}_{t + 1}^{\T}).
        \end{aligned}
    \end{equation}
    Consequently,
    \begin{equation}
        \label{eq:ztf_s2}
        \begin{aligned}
            &\condE{\norm{\z_{t + 1} - \nabla F(\1\bar{x}_{t + 1}^{\T})}^2}{\cF_t^0}\leq \beta\norm{\z_t - \nabla F(\1\bar{x}_t^{\T})}^2 + \frac{3(1-\beta)^2n\sigma^2}{Q}\\
            &\quad + \frac{5(1-\beta)^2L^2}{Q(1-\beta)}\sum_{\ell=0}^{Q-1}\condE{\norm{\x_t^\ell-\1\bar{x}_t^{\T}}^2}{\cF_t^0} + \frac{5nL^2}{1-\beta}\condE{\norm{\bar{x}_{t + 1} - \bar{x}_t}^2}{\cF_t^0}.
        \end{aligned}
    \end{equation}
    
    We now bound the term $\normi{\bar{x}_{t + 1} - \bar{x}_t}^2$. By Lemma~\ref{lem:avg}, we have
    \begin{equation}
        \label{eq:xt_diff}
        \begin{aligned}
            \bar{x}_{t + 1} - \bar{x}_t 
            &= -\teta\beta\bar{z}_t - \frac{\teta(1-\beta)}{Q}\sum_{\ell=0}^{Q-1}\brk{\bar{\Delta}_t^\ell + \frac{1}{n}\sumn \nabla f_i(x_{i,t}^\ell)}.
        \end{aligned}
    \end{equation} 
    Taking the squared norm and conditional expectation on both sides of \eqref{eq:xt_diff} yields
    \begin{equation}
        \label{eq:xt_diff_var}
        \begin{aligned}
            \condE{\norm{\bar{x}_{t + 1} - \bar{x}_t}^2}{\cF_t^0}&\leq 3\teta^2\beta^2\norm{\bar{z}_t}^2 + \frac{3\teta^2 (1-\beta)^2\sigma^2}{nQ} + 3\teta^2(1-\beta)^2\condE{\norm{\bavg_t}^2}{\cF_t^0}.
        \end{aligned}
    \end{equation}
    Substituting \eqref{eq:sum_ub} and \eqref{eq:xt_diff_var} into \eqref{eq:ztf_s2}, and taking the full expectation on both sides, leads to the desired result.

\subsection{Proof of Lemma \ref{lem:tcL}}
\label{app:tcL}
    In light of \eqref{eq:Pix2cRx} and \eqref{eq:Piy2cRy}, we substitute \eqref{eq:sum_ub} into \eqref{eq:descent} and take the full expectation to obtain
    \begin{equation}
        \label{eq:descent1}
        \begin{aligned}
            &\E\brk{f(\bar{d}_{t + 1})} \leq \E\brk{f(\bar{d}_t)} - \frac{\teta}{2}\prt{1-3\eta_a^2Q^2L^2}\E\brk{\norm{\nabla f(\bar{x}_t)}^2} - \frac{\teta}{4}(1-4\teta L)\E\brk{\norm{\bavg_t}^2}\\
            &\quad + \frac{\teta^3\beta^2 L^2}{(1-\beta)^2}\E\brk{\norm{\bar{z}_t}^2} + \frac{9\teta L^2}{2n}\cR_t^x + \frac{3\teta \eta_a^2Q^2L^2}{2n}\cR_t^{y} \\
            &\quad + \frac{3\teta\eta_a^2 Q^2L^2\beta^2}{2n}\E\brk{\norm{\z_t - \nabla F(\1\bar{x}_t^{\T})}^2} + \frac{\teta^2 L\sigma^2}{nQ} + \frac{9\teta\eta_a^2QL^2\sigma^2}{2}.
        \end{aligned}
    \end{equation}
    Substituting \eqref{eq:sum_ub} into \eqref{eq:cRy} leads to 
        \begin{align*}
            &\brk{1 - \frac{15c_0(1-\beta)^2\eta_a^2Q^2 L^2}{1-\trw}}\cR_{t + 1}^y \leq \trw \cR_t^y + \frac{5c_0(1-\beta)^2n\sigma^2}{Q}\prt{1 + \frac{9\eta_a^2Q^2L^2}{1-\trw}} \\
            &\quad + \frac{45 c_0(1-\beta)^2L^2}{(1-\trw)}\cR_{t + 1}^x   + \frac{5c_0(1-\beta)^2}{1-\trw}\prt{1 + 3\eta_a^2Q^2L^2\beta^2}\E\brk{\norm{\z_{t + 1} - \nabla F(\1\bar{x}_{t + 1}^{\T})}^2} \\
            &\quad  + \frac{15n\eta_a^2Q^2 c_0(1-\beta)^2L^2}{1-\trw}\E\brk{\norm{\nabla f(\bar{x}_{t + 1})}^2}\\
            &\leq \prt{\trw + \frac{45c_0^2\trw^2\teta^2(1-\beta)^2L^2}{(1-\trw)^2} + \frac{90c_0\eta_a^2Q^2L^2(1-\beta)^3}{1-\trw}}\cR_t^y \\
            &\quad + \frac{45c_0(1-\beta)^2 L^2}{1-\trw}\brk{1 + 6(1-\beta)}\cR_t^x + \frac{6c_0(1-\beta)^2}{1-\trw}\E\brk{\norm{\z_t - \nabla F(\1\bar{x}_t^{\T})}^2}\\
            &\quad + \frac{90c_0\teta^2 L^2 (1-\beta)n}{1-\trw}\brk{\eta_a^2Q^2L^2(1-\beta) + 1}\E\brk{\norm{\bar{z}_t}^2}\\
            &\quad  + \frac{30c_0n\eta_a^2Q^2L^2(1-\beta)^2}{1-\trw}\brk{3(1-\beta) + 1}\E\brk{\norm{\nabla f(\bar{x}_t)}^2} \\
            &\quad + \frac{90c_0 n \teta^2 L^2(1-\beta)^3}{1-\trw}\prt{1 + \eta_a^2 Q^2L^2}\E\brk{\norm{\bavg_t}^2} \\
            &\quad +  {\frac{5c_0(1-\beta)^2n\sigma^2}{Q}\brk{1 + \frac{9\eta_a^2Q^2L^2}{1-\trw} + \frac{24(1-\beta)}{5(1-\trw)} + \frac{18\teta^2\eta_a^2Q^2 L^4(1-\beta)}{n(1-\trw)}}}.
        \end{align*}
    We let 
        \begin{align*}
            \eta_a&\leq \min\crk{\frac{1-\trw}{6\sqrt{10c_0}(1-\beta) QL}, \frac{1-\trw}{3(1-\beta)\sqrt{10c_0}QL}},\\
            \teta&\leq \min\crk{\frac{(1-\trw)^\frac{3}{2}}{18\sqrt{5}c_0(1-\beta) L}, \frac{1-\beta}{2\sqrt{15}c_0 L}},\; \beta\geq \trw.
        \end{align*}
    Then,
    \begin{align*}
        & 1 - \frac{15c_0(1-\beta)^2\eta_a^2Q^2 L^2}{1-\trw}\geq \frac{3(1+\trw)}{2(2 + \trw)},\\
        &\frac{45c_0^2\trw^2\teta^2(1-\beta)^2L^2}{(1-\trw)^2} + \frac{90c_0\eta_a^2Q^2L^2(1-\beta)^3}{1-\trw}\leq \frac{1-\trw}{2}.
    \end{align*} 
    It follows that
    \begin{equation}
        \label{eq:cRy_sim}
        \begin{aligned}
            &\cR_{t + 1}^y \leq \frac{2+\trw}{3}\cR_t^y +  \frac{8c_0(1-\beta)^2}{1-\trw}\E\brk{\norm{\z_t - \nabla F(\1\bar{x}_t^{\T})}^2} + \frac{420c_0(1-\beta)^2L^2}{1-\trw}\cR_t^{x}\\
            &\quad  + \frac{123c_0\teta^2L^2n (1-\beta)}{1-\trw} \E\brk{\norm{\bar{z}_t}^2} + {\frac{30c_0(1-\beta)^2n\sigma^2}{Q}\prt{2 + \frac{3\eta_a^2 Q^2 L^2}{1-\trw}}}\\
            &\quad + \frac{160c_0n\eta_a^2Q^2L^2(1-\beta)^2}{1-\trw} \E\brk{\norm{\nabla f(\bar{x}_t)}^2}+ \frac{123c_0n\teta^2L^2(1-\beta)^3}{1-\trw} \E\brk{\norm{\bavg_t}^2}.
        \end{aligned}
    \end{equation}

    To illustrate the derivation of $\tcL_t$, we first define $\tcL_t$ that includes undetermined positive coefficients $\cC_1$-$\cC_4$:
    \begin{equation}
    \label{eq:tcL_can}
    \begin{aligned}
        \tcL_t&:= \E\brk{f(\bar{d}_t)} - f^* + \frac{\teta^3 L^2\cC_1}{(1-\beta)^3} \E\brk{\norm{\bar{z}_t}^2} + \frac{\teta L^2\cC_2}{n(1-\trw)}\cR_t^x \\
        &\quad + {\frac{\teta^3 L^2\cC_3}{n(1-\trw)^3}}\cR_t^{y} + {\frac{\teta^3 L^2\cC_4 }{n(1-\trw)^3}}\E\brk{\norm{\z_t - \nabla F(\1\bar{x}_t^{\T})}^2}.
    \end{aligned}
    \end{equation}
    According to \eqref{eq:barz}, \eqref{eq:cRx}, \eqref{eq:ztf}, \eqref{eq:descent1}, \eqref{eq:cRy_sim}, and \eqref{eq:tcL_can}, we have 
    \small
        \begin{align}
            &\tcL_{t + 1} \leq \E\brk{f(\bar{d}_t)} - f^*  \nonumber \\
            &+ \brk{(1-\beta) + \beta\cC_1 + {\frac{123c_0\teta^2 L^2(1-\beta)^4\cC_3}{(1-\trw)^4}} + {\frac{15\teta^2L^2(1-\beta)^2\cC_4}{(1-\trw)^3}}}\frac{\teta^3L^2}{(1-\beta)^3}\E\brk{\norm{\bar{z}_t}^2}\nonumber\\
            & + \brk{\frac{9(1-\trw)}{2} + \trw \cC_2 + {\frac{420c_0\teta^2 L^2(1-\beta)^2\cC_3}{(1-\trw)^3}} + {\frac{45\teta^2L^2(1-\beta)\cC_4}{(1-\trw)^2}}} \frac{\teta L^2}{n(1-\trw)}\cR_t^x \nonumber\\
            & + {\brk{\frac{3(1-\trw)^3}{2\eta_s^2} + c_0\trw^2\cC_2(1-\trw) + \frac{2 + \trw}{3}\cC_3 + 15\eta_a^2Q^2L^2(1-\beta)\cC_4}\frac{\teta^3 L^2}{n(1-\trw)^3}}\cR_t^y\nonumber\\
            & + {\brk{\frac{3(1-\trw)^3}{2\eta_s^2} +  {\frac{8c_0(1-\beta)^2\cC_3}{1-\trw}} + \frac{1+\beta}{2}\cC_4}\frac{\teta^3 L^2}{n(1-\trw)^3}}\E\brk{\norm{\z_t - \nabla F(\1\bar{x}_t^{\T})}^2}\nonumber\\
            & - \frac{\teta}{2}\brk{1 - 3\eta_a^2Q^2L^2 - {\frac{320c_0\teta^2\eta_a^2Q^2L^4(1-\beta)^2\cC_3}{(1-\trw)^4}} - {\frac{30\eta_a^2 Q^2\teta^2 L^4(1-\beta)\cC_4}{(1-\trw)^3}}}\E\brk{\norm{\nabla f(\bar{x}_t)}^2}\nonumber\\
            & - \frac{\teta}{4}\brk{1-4\teta L - \frac{12\teta^2 L^2\cC_1}{(1-\beta)^2} - {\frac{492c_0\teta^4L^4(1-\beta)^3\cC_3}{(1-\trw)^4}} - {\frac{60\teta^4L^4(1-\beta)\cC_4}{(1-\trw)^3}}}\E\brk{\norm{\bavg_t}^2}\nonumber\\
            &  + \frac{3\teta\eta_a^2QL^2\sigma^2}{2}\brk{3  + {\frac{40c_0\eta_s^2(1-\beta)^2\cC_3}{(1-\trw)^3} + \frac{60c_0\teta^2 L^2(1-\beta)^2\cC_3}{(1-\trw)^4}} \right.\nonumber\\
            &\left. + {\frac{2\eta_s^2(1-\beta)^2\cC_4}{(1-\trw)^3}}\prt{1 + \frac{15\eta_a^2Q^2L^2}{1-\beta} + \frac{5\teta^2L^2}{n(1-\beta)}}} + \frac{\teta^2L\sigma^2}{nQ} + \frac{2\teta^3L^2\sigma^2\cC_1}{nQ(1-\beta)}. \label{eq:tcL_s1}
        \end{align}
    \normalsize
    We choose $\cC_1$-$\cC_4$ such that the following conditions hold:
    \begin{subequations}
        \label{eq:cCs}
        \begin{align}
            1-\beta + \beta\cC_1 + {\frac{123c_0\teta^2 L^2(1-\beta)^4\cC_3}{(1-\trw)^4}} + {\frac{15\teta^2L^2(1-\beta)^2\cC_4}{(1-\trw)^3}}&\leq\cC_1,\label{eq:cC1}\\
            \frac{9(1-\trw)}{2} + \trw \cC_2 + {\frac{420c_0\teta^2 L^2(1-\beta)^2\cC_3}{(1-\trw)^3}} + {\frac{45\teta^2L^2(1-\beta)\cC_4}{(1-\trw)^2}}&\leq \cC_2,\label{eq:cC2}\\
            {\frac{3(1-\trw)^3}{2\eta_s^2} + c_0\trw^2\cC_2(1-\trw) + \frac{2 + \trw}{3}\cC_3 + 15\eta_a^2Q^2L^2(1-\beta)\cC_4}  &\leq \cC_3,\label{eq:cC3}\\
            {\frac{3(1-\trw)^3}{2\eta_s^2} +  {\frac{8c_0(1-\beta)^2\cC_3}{1-\trw}} + \frac{1+\beta}{2}\cC_4} &\leq \cC_4.\label{eq:cC4}
        \end{align}
    \end{subequations}
    From \eqref{eq:cC4} {and letting $\eta_s \geq (1-\trw)^{3/2}/\sqrt{1-\beta}$}, we set $\cC_4 = 6 + 18c_0\cC_3$. Substituting $\cC_4 = 6 + 18c_0\cC_3$ into \eqref{eq:cC3} {and noting $\beta\geq \trw$} yields
    \begin{align}
        \label{eq:cC43}
        \cC_3 \geq \frac{\frac{9(1-\trw)^2}{2\eta_s^2} + 3c_0\trw^2\cC_2 + 270\eta_a^2Q^2L^2}{1 - 810\eta_a^2Q^2L^2c_0}.
    \end{align}
    Given that $\eta_a\leq 1/(130c_0QL)$ and {$\eta_s\geq (1-\trw)^{3/2}/\sqrt{1-\beta}$}, we choose {$\cC_3 = 10 + 4c_0\cC_2$} to satisfy \eqref{eq:cC43}. Substituting the expressions for $\cC_3$-$\cC_4$ and the relation $\beta\geq \trw$ into \eqref{eq:cC2} results in
    \begin{equation}
        \label{eq:cC12}
        \begin{aligned}
            \cC_2 &\geq \frac{\frac{9}{2} + \frac{4200c_0\teta^2L^2(1-\beta)^2}{(1-\trw)^4} + \frac{270(1+30c_0)\teta^2L^2(1-\beta)}{(1-\trw)^3}}{1-\frac{1680c_0^2\teta^2L^2(1-\beta)^2}{(1-\trw)^4} - \frac{3240c_0^2\teta^2L^2(1-\beta)}{(1-\trw)^3}}.
        \end{aligned}
    \end{equation}
    By letting
    \begin{align}
        \label{eq:step_cC1}
        \teta\leq \frac{1-\trw}{130c_0 L},\quad \beta\geq \trw,
    \end{align}
    we set $\cC_2 = 11$ to satisfy \eqref{eq:cC12}. Thus, we obtain {$\cC_3 = 10+4c_0\cC_2 = 10+44c_0$ and $\cC_4 = 6 + 18c_0\cC_3 = 6(1+30c_0 + 132c_0^2)$}. Substituting the expressions for $\cC_2$-$\cC_4$ into \eqref{eq:cC1} yields 
    \begin{equation}
        \label{eq:cC15}
        \begin{aligned}
            \cC_1\geq 1 + \frac{246c_0(6+22c_0)\teta^2 L^2(1-\beta)^3}{(1-\trw)^4} + \frac{90(1 + 30c_0 + 132c_0^2)\teta^2 L^2(1-\beta)}{(1-\trw)^3}.
        \end{aligned}
    \end{equation}
    It suffices to let \eqref{eq:step_cC1} hold and $\cC_1 = 4$ for \eqref{eq:cC15} to hold. 

    In light of the above results, relation \eqref{eq:tcL_s1} simplifies to
    \begin{equation}
        \label{eq:tcL_s2}
        \begin{aligned}
            \tcL_{t + 1} 
            &\leq \tcL_t - \frac{\teta}{3}\E\brk{\norm{\nabla f(\bar{x}_t)}^2} - \frac{\teta}{5}\E\brk{\norm{\bavg_t}^2} + {\frac{2\teta^2L\sigma^2}{nQ}} +  \frac{{8}\teta^3L^2\sigma^2}{nQ(1-\beta)}\\
            &\quad + {10}\teta\eta_a^2QL^2\sigma^2 + {\frac{3c_0'\teta^3 L^2\sigma^2(1-\beta)^2}{Q(1-\trw)^3}},
        \end{aligned}
    \end{equation}
    where we invoked condition \eqref{eq:step_cC1} and {$\eta_a\leq 1/(42\sqrt{c_0}QL)$} and {denote $c_0':= 20c_0\cC_3 + \cC_4 = 6 + 380c_0 + 1672c_0^2$}. This completes the proof.

{
\subsection{Proof of Lemma \ref{lem:ncvx_cons}}
\label{app:ncvx_cons}
    Averaging \eqref{eq:cRx} in Lemma~\ref{lem:tPitx} over $t$ and using $\E\brki{\normi{\Pi\x_t}^2}\leq \cR_t^x$ yields 
    \begin{equation}
        \label{eq:Pix_sum1}
        \begin{aligned}
           \frac{1 - \trw}{T}\sum_{t=0}^{T-1}\E\brk{\norm{\Pi\x_t}^2} \leq \frac{c_0\norm{\Pi\x_0}^2}{T} + \frac{c_0\trw^2\teta^2}{(1-\trw)T}\sum_{t=0}^{T-2} \E\brk{\norm{\Pi\y_t}^2}.
        \end{aligned}
    \end{equation}

    Based on \eqref{eq:barz} and $\z_0 = \mathbf{0}$, we have 
    \begin{equation}
        \label{eq:barz_sum}
        \begin{aligned}
            \frac{1}{T}\sum_{t=0}^{T-1}\E\brk{\norm{\bar{z}_t}^2}&\leq \frac{2(1-\beta)\sigma^2}{nQ} + \frac{3}{T}\sum_{t=0}^{T-2}\E\brk{\norm{\bavg_t}^2}.
        \end{aligned}
    \end{equation}

    Noting that 
    $nQ\bavg_t =  \sum_{\ell=0}^{Q-1}\sumn \brki{\nabla f_i(x_{i,t}^\ell) - \nabla f_i(\bar{x}_t)} + nQ\nabla f(\bar{x}_t),$ we obtain
    \begin{equation}
        \label{eq:bavg_ub1}
        \begin{aligned}
            \E\brk{\norm{\bavg_t}^2} &\leq \frac{2L^2}{nQ}\sum_{\ell=0}^{Q-1}\sumn \E\brk{\norm{x_{i,t}^\ell - \bar{x}_t}^2} + 2\E\brk{\norm{\nabla f(\bar{x}_t)}^2}.
        \end{aligned}
    \end{equation}
    Invoking \eqref{eq:sum_ub} yields
    \begin{equation}
        \label{eq:bavg_sum}
        \begin{aligned}
            &\frac{1}{T}\sum_{t=0}^{T-1}\E\brk{\norm{\bavg_t}^2} \leq \frac{18L^2}{nT}\sum_{t=0}^{T-1}\E\brk{\norm{\Pi\x_t}^2} + 18\eta_a^2 L^2 Q \sigma^2 + \frac{6\eta_a^2Q^2 L^2}{nT}\sum_{t=0}^{T-1}\E\brk{\norm{\Pi\y_t}^2} \\
            &\quad + \frac{6\eta_a^2Q^2 L^2 + 2}{T}\sum_{t=0}^{T-1} \E\brk{\norm{\nabla f(\bar{x}_t)}^2} + \frac{6\eta_a^2Q^2L^2\beta^2}{nT}\sum_{t=0}^{T-1}\E\brk{\norm{\z_t - \nabla
             F(\1\bar{x}_t^{\T})}^2}.
        \end{aligned}
    \end{equation}


    Summing \eqref{eq:ztf} over $t=0,\dots,T-2$ and using $\z_0=\mathbf0$, we bound the terms $\sum_{t=0}^{T-1}\E\brki{\normi{\bar{z}_t}^2}/T$ and $\sum_{t=0}^{T-1}\E\brki{\normi{\bavg_t}^2}/T$ by \eqref{eq:barz_sum} and \eqref{eq:bavg_sum}, respectively. Collecting the $\frac{1}{T}\sum_{t=0}^{T-1}\E\brki{\normi{\z_t - \nabla F(\1\bar{x}_t^{\T})}^2}/T$ terms on the left-hand side and invoking \eqref{eq:nvcx_stepsize_cons} yields
    \begin{align*}
        1 - \frac{1+\beta}{2} - \frac{360\teta^2\eta_a^2Q^2L^4}{1-\beta}\geq \frac{1-\beta}{3}.
    \end{align*}
    Consequently, we have
    \begin{equation}
        \label{eq:ztf_sum_barz}
        \begin{aligned}
            &\frac{1}{T}\sum_{t=0}^{T-1}\E\brk{\norm{\z_t - \nabla F(\1\bar{x}_t^{\T})}^2} \leq \frac{3\norm{\nabla F(\1\bar{x}_0^{\T})}^2}{(1-\beta)T}+ \frac{136L^2}{T}\sum_{t=0}^{T-2}\E\brk{\norm{\Pi\x_t}^2} + \frac{90\teta^2L^2\sigma^2}{(1-\beta)Q}\\
            &\quad + \frac{46\eta_a^2Q^2L^2}{T}\sum_{t=0}^{T-2}\E\brk{\norm{\Pi\y_t}^2}  + \frac{9(1-\beta) n\sigma^2}{Q}\prt{1 + \frac{16\eta_a^2Q^2L^2}{1-\beta} + \frac{5\teta^2 L^2}{n(1-\beta)}} \\
            &\quad + \brk{\eta_a^2Q^2L^2 + \frac{12\teta^2 L^2}{(1-\beta)^2}}\frac{45n}{T}\sum_{t=0}^{T-2}\E\brk{\norm{\nabla f(\bar{x}_t)}^2}.
        \end{aligned}
    \end{equation}


    Summing \eqref{eq:cRy} over $t=0,\dots,T-2$, we apply \eqref{eq:sum_ub} to bound $\sum_{\ell=0}^{Q-1}\E\brki{\normi{\x_t^\ell - \1\bar{x}_t^{\ell}}^2}$ and \eqref{eq:ztf_sum_barz} to bound $\sum_{t=0}^{T-1}\E\brki{\normi{\z_t - \nabla F(\1\bar{x}_t^{\T})}^2}/T$. Using $\E\brki{\normi{\Pi\y_t}^2}\leq\cR_t^y$ and collecting the resulting $\cR_t^y$ terms on the left-hand side gives
    \begin{equation}
        \label{eq:cRy_sum1}
        \begin{aligned}
            &\frac{1}{T}\sum_{t=0}^{T-1}\cR_t^y \leq \frac{2c_0\E\brki{\normi{\Pi\y_0}^2}}{(1-\trw)T} +  \frac{800c_0(1-\beta)^2 L^2}{(1-\trw)^2T}\sum_{t=1}^{T-1}\E\brk{\norm{\Pi\x_t}^2} \\
            &\quad + \brk{2\eta_a^2Q^2L^2 + \frac{12\teta^2 L^2}{(1-\beta)^2}}\frac{452c_0(1-\beta)^2n}{(1-\trw)^2T}\sum_{t=0}^{T-2}\E\brk{\norm{\nabla f(\bar{x}_t)}^2} \\
            &\quad + \prt{\frac{11}{2} + \frac{154\eta_a^2 Q^2 L^2}{1-\trw} }\frac{10c_0(1-\beta)^2n\sigma^2}{(1-\trw)Q} + \frac{32c_0(1-\beta)\normi{\nabla F(\1\bar{x}_0^{\T})}^2}{(1-\trw)^2T},
        \end{aligned}
    \end{equation}
    where we invoked the conditions in \eqref{eq:nvcx_stepsize_cons} to obtain
    \begin{align*}
        1 - \trw - \frac{246c_0 \eta_a^2 Q^2 (1-\beta)^2 L^2}{1-\trw}\geq \frac{19(1-\trw)}{20}.
    \end{align*}

    Noting $\E\brki{\normi{\Pi\y_t}^2}\leq \cR_t^y$ and substituting \eqref{eq:cRy_sum1} into \eqref{eq:Pix_sum1} leads to 

    \begin{equation}
        \label{eq:Pix_sum3}
        \begin{aligned}
            &\frac{1}{T}\sum_{t=0}^{T-1}\E\brk{\norm{\Pi\x_t}^2}\leq \frac{2c_0\normi{\Pi\x_0}^2}{(1-\trw)T} + \frac{4c_0^2\teta^2 \E\brki{\normi{\Pi\y_0}^2}}{(1-\trw)^3T}\\
                &\quad + \brk{2\eta_a^2Q^2L^2 + \frac{12\teta^2 L^2}{(1-\beta)^2}}\frac{904c_0^2\teta^2(1-\beta)^2n}{(1-\trw)^4T}\sum_{t=0}^{T-2}\E\brk{\norm{\nabla f(\bar{x}_t)}^2}\\
                &\quad + \frac{128c_0^2\teta^2(1-\beta)\normi{\nabla F(\1\bar{x}_0^{\T})}^2}{(1-\trw)^4 T} +  \prt{3 + \frac{154\eta_a^2 Q^2 L^2}{1-\trw} }\frac{20c_0^2\teta^2 (1-\beta)^2n\sigma^2}{(1-\trw)^3Q},
        \end{aligned}
    \end{equation}
    where we invoked the conditions in \eqref{eq:nvcx_stepsize_cons} to obtain
    \begin{align*}
        1 - \trw - \frac{800c_0^2\teta^2(1-\beta)^2 L^2}{(1-\trw)^3} \geq \frac{19(1-\trw)}{20}.
    \end{align*}
    Substituting \eqref{eq:Piy0_ub} into \eqref{eq:Pix_sum3} yields the desired result \eqref{eq:cons_sum_ub}.

}

\subsection{Proof of Lemma \ref{lem:tcL_pl}}
\label{app:tcL_pl}

    To utilize the PL condition in \eqref{eq:PL}, we construct recursions for $\E\brki{f(\bar{d}_{t + 1})}$ and $\E\brki{\normi{\nabla f(\bar{d}_t)}^2}$ similar to \eqref{eq:descent}. 
    To obtain the term $-\normi{\nabla f(\bar{d}_t)}^2$ in the recursion of $\E\brki{f(\bar{d}_t)}$, we bound the inner product in \eqref{eq:descent_s1} with derivations similar to \eqref{eq:descent_s2}:
    \begin{equation}
        \label{eq:descent_s2_nd}
        \begin{aligned}
            &-\teta\condE{\inpro{\nabla f(\bar{d}_t), \bavg_t}}{\cF_t^0}
            \leq -\frac{\teta}{2}\norm{\nabla f(\bar{d}_t)}^2 -\frac{\teta}{2}\condE{\norm{\bavg_t}^2}{\cF_t^0} + \frac{\teta^3\beta^2 L^2}{(1-\beta)^2}\norm{\bar{z}_t}^2 \\
            &\quad + \frac{\teta L^2}{nQ}\sum_{\ell=0}^{Q-1}\sumn\condE{\norm{x_{i,t}^\ell - \bar{x}_t}^2}{\cF_t^0},
        \end{aligned}
    \end{equation}
    where we invoked Assumption~\ref{as:smooth} and the relation in \eqref{eq:dt_xt}. 
    Then, substituting \eqref{eq:avg_var} and \eqref{eq:descent_s2_nd} into \eqref{eq:descent_s1} yields
    \begin{equation}
        \label{eq:descent_nd}
        \begin{aligned}
            \E\brk{f(\bar{d}_{t + 1})} &\leq \E\brk{f(\bar{d}_t)} - \frac{\teta}{2}\E\brk{\norm{\nabla f(\bar{d}_t)}^2} - \frac{\teta}{2}\prt{1 - 2\teta L}\E\brk{\norm{\bavg_t}^2}\\
            &\quad + \frac{\teta^3\beta^2 L^2}{(1-\beta)^2}\E\brk{\norm{\bar{z}_t}^2}  + \frac{\teta L^2}{nQ}\sum_{\ell=0}^{Q-1}\sumn\E\brk{\norm{x_{i,t}^\ell - \bar{x}_t}^2}+ \frac{\teta^2 L\sigma^2}{nQ}.
        \end{aligned}
    \end{equation}

    From \eqref{eq:dt_xt}, we obtain
    \begin{equation}
        \label{eq:nfx2nfd}
        \begin{aligned}
            \norm{\nabla f(\bar{x}_t)}^2 \leq \frac{2\teta^2 \beta^2L^2}{(1-\beta)^2}\norm{\bar{z}_t}^2 + 2\norm{\nabla f(\bar{d}_t)}^2.
        \end{aligned}
    \end{equation}
    Substituting \eqref{eq:sum_ub} into \eqref{eq:descent_nd} and invoking \eqref{eq:nfx2nfd} leads to 
        \begin{align*}
            &\E\brk{f(\bar{d}_{t + 1})} \leq \E\brk{f(\bar{d}_t)} - \frac{\teta}{2}\prt{1-12\eta_a^2Q^2L^2}\E\brk{\norm{\nabla f(\bar{d}_t)}^2} - \frac{\teta}{2}(1-2\teta L)\E\brk{\norm{\bavg_t}^2}\\
            &\quad + \frac{5\teta^3\beta^2 L^2}{4(1-\beta)^2}\E\brk{\norm{\bar{z}_t}^2} + \frac{9\teta L^2}{n}\cR_t^x + \frac{3\teta\eta_a^2Q^2L^2}{n}\E\brk{\norm{\z_t - \nabla F(\1\bar{x}_t^{\T})}^2}\\
            &\quad + \frac{3\teta \eta_a^2Q^2L^2}{n}\cR_t^{y} + \frac{\teta^2 L\sigma^2}{nQ} + 9\teta\eta_a^2QL^2\sigma^2,
        \end{align*}
    where we invoked the condition $\eta_a\leq 1/(15QL)$.

    Define {$C_0:= 1 + 30c_0 + 132c_0^2$}. Invoking \eqref{eq:nfx2nfd} and following a derivation similar to \eqref{eq:tcL_s1}, we get
    \small
    \begin{equation}
        \label{eq:tcL_pl_s1}
        \begin{aligned}
            &\tcL_{t + 1} \leq \E\brk{f(\bar{d}_t)} - f^* + {\frac{\teta^2L\sigma^2}{nQ}} +  \frac{{9}\teta^3L^2\sigma^2}{nQ(1-\beta)}+ {11}\teta\eta_a^2QL^2\sigma^2 + {\frac{3c_0'\teta^3 L^2\sigma^2(1-\beta)^2}{Q(1-\trw)^3}}\\
            & + \brk{\frac{5(1-\beta)}{16} + \beta + {\frac{123(5+22c_0)c_0\teta^2L^2(1-\beta)^4}{2(1-\trw)^4}} + {\frac{50C_0 \teta^2 L^2(1-\beta)^2}{(1-\trw)^3}} }\frac{4\teta^3L^2}{(1-\beta)^3}\E\brk{\norm{\bar{z}_t}^2}\\
            &+ \brk{\frac{9(1-\trw)}{11} + \trw + {\frac{840c_0(5 + 22c_0)\teta^2L^2(1-\beta)^2}{11(1-\trw)^3}} + {\frac{270C_0\teta^2L^2(1-\beta)}{11(1-\trw)^2}}} \frac{11\teta L^2}{n(1-\trw)}\cR_t^x\\
            &+ {\brk{\frac{3(1-\trw)^3}{2(5+22c_0)\eta_s^2} + \frac{11c_0(1-\trw)}{2(5 + 22c_0)} + \frac{2+\trw}{3} + \frac{45C_0\eta_a^2Q^2L^2(1-\beta)}{5 + 22c_0}} \frac{2(5+22c_0)\teta^3L^2}{n(1-\trw)^3}}\cR_t^y\\
            &+ {\brk{\frac{(1-\trw)^3}{2C_0\eta_s^2} + \frac{8c_0(5+22c_0)(1-\beta)^2}{3C_0(1-\trw)} + \frac{1 + \beta}{2}}\frac{6C_0\teta^3 L^2}{n(1-\trw)^3}}\E\brk{\norm{\z_t - \nabla F(\1\bar{x}_t^{\T})}^2}\\
            &- \frac{\teta}{2}\brk{1-12\eta_a^2Q^2L^2 - {\frac{50c_0'\teta^2\eta_a^2 Q^2L^4(1-\beta)}{(1-\trw)^3}}}\E\brk{\norm{\nabla f(\bar{d}_t)}^2}\\
            &- \frac{\teta}{2}\brk{1-2\teta L - \frac{24\teta^2L^2}{(1-\beta)^2} - {\frac{492c_0(5+22c_0)\teta^4L^4(1-\beta)^3}{(1-\trw)^4}} - {\frac{180C_0\teta^4L^4(1-\beta)}{(1-\trw)^3}}}\E\brk{\norm{\bavg_t}^2}.
        \end{aligned}
    \end{equation}\normalsize
    Noting that {$\teta\leq (1-\beta)/(200c_0 L)$, $\eta_s\geq (1-\trw)^{3/2}/\sqrt{1-\beta}$}, and $\beta\geq \trw$, we have 
    \begin{align*}
        \frac{5(1-\beta)}{16} + \beta + {\frac{123(5+22c_0)c_0\teta^2L^2(1-\beta)^4}{2(1-\trw)^4}} + {\frac{50C_0 \teta^2 L^2(1-\beta)^2}{(1-\trw)^3}}  &\leq 1-\frac{1-\beta}{8},\\
        \frac{9(1-\trw)}{11} + \trw + {\frac{840c_0(5 + 22c_0)\teta^2L^2(1-\beta)^2}{11(1-\trw)^3}} + {\frac{270C_0\teta^2L^2(1-\beta)}{11(1-\trw)^2}}&\leq 1- \frac{1-\trw}{21},\\
        {\frac{3(1-\trw)^3}{2(5+22c_0)\eta_s^2} + \frac{11c_0(1-\trw)}{2(5 + 22c_0)} + \frac{2+\trw}{3} + \frac{45C_0\eta_a^2Q^2L^2(1-\beta)}{5 + 22c_0}}&\leq 1-\frac{1-\trw}{21},\\
        {\frac{(1-\trw)^3}{2C_0\eta_s^2} + \frac{8c_0(5+22c_0)(1-\beta)^2}{3C_0(1-\trw)} + \frac{1 + \beta}{2}} &\leq 1 - \frac{1-\beta}{21}.
    \end{align*}
    Letting $\teta\leq\min\crki{\nicefrac{1-\trw}{7\mu}, \nicefrac{1-\beta}{8\mu}}$
    and invoking Assumption~\ref{as:PL} leads to the desired result \eqref{eq:tcL_pl}. 

\subsection{Proof of Theorem \ref{thm:pl}}
\label{app:pl}

    {We first bound the optimality gap at the averaged iterate. By Assumption~\ref{as:smooth} and \eqref{eq:dt_xt}, we have
        \begin{equation}
            \label{eq:fbarx2fbard}
            \begin{aligned}
                f(\bar{x}_t) - f^* 
                &\leq 2\brk{f(\bar{d}_t) - f^*} + \frac{3\teta^2\beta^2 L}{2(1-\beta)^2}\norm{\bar{z}_t}^2.
            \end{aligned}
        \end{equation}
        

        Letting $\beta\leq (1-\teta\mu / 3)$ and $\teta\leq (1-\beta)/(18\beta^2 L)$, we obtain from \eqref{eq:tcL_pl} and \eqref{eq:barz} that 
        \begin{equation}
            \label{eq:pl_lya}
            \begin{aligned}
                &2\tcL_{t + 1} + \frac{3\teta^2\beta^2 L}{2(1-\beta)^2}\E\brk{\norm{\bar{z}_{t + 1}}^2} \leq \prt{1 - \frac{\teta\mu}{3}}\brk{2\tcL_{t} + \frac{3\teta^2\beta^2 L}{2(1-\beta)^2}\E\brk{\norm{\bar{z}_{t}}^2} } + \frac{9\teta^2 \beta^2 L\sigma^2}{2 n Q}\\
                &\quad + \frac{2\teta^2L\sigma^2}{nQ} + \frac{18\teta^3L^2\sigma^2}{nQ(1-\beta)} + 22\teta\eta_a^2QL^2\sigma^2 + \frac{6c_0'\teta^3 L^2(1-\beta)^2\sigma^2}{Q(1-\trw)^3}.
            \end{aligned}
        \end{equation}

        Unrolling \eqref{eq:pl_lya} and invoking \eqref{eq:fbarx2fbard} yields
        \begin{equation}
            \label{eq:pl_f_ub}
            \begin{aligned}
                \E\brk{f(\bar{x}_t) - f^*} &\leq 2\prt{1 - \frac{\teta\mu}{3}}^t\tcL_0 + \frac{39\teta L\sigma^2}{2 n\mu Q}  + \frac{54\teta^2L^2\sigma^2}{n\mu Q(1-\beta)}\\
                &\quad + \frac{66\eta_a^2QL^2\sigma^2}{\mu} + \frac{18c_0'\teta^2 L^2(1-\beta)^2\sigma^2}{Q\mu (1-\trw)^3}.
            \end{aligned}
        \end{equation}
    
    }

    To upper bound $\E\brki{\normi{\Pi\x_t}^2}\leq \cR_t^x$, we consider the following Lyapunov function
    \begin{equation}
        \label{eq:tcH}
        \tcH_t := \cR_t^x + {\frac{15c_0\teta^2 }{(1-\trw)^2}}\cR_t^y + {\frac{480c_0^2\teta^2 }{(1-\trw)^2}}\E\brk{\norm{\z_t - \nabla F(\1\bar{x}_t^{\T})}^2} + \frac{\teta^2 n}{(1-\beta)^2}\E\brk{\norm{\bar{z}_t}^2}.
    \end{equation} 
    Here, the corresponding coefficients are derived following similar steps as in the proof of Lemma~\ref{lem:tcL}. Then,
    \begin{equation}
        \label{eq:tcH_s1}
        \begin{aligned}
            &\tcH_{t + 1} \leq \brk{\trw + {\frac{6300c_0^2\teta^2 L^2(1-\beta)^2}{(1-\trw)^3}} + {\frac{21600c_0^2\teta^2 L^2(1-\beta)}{(1-\trw)^2}}}\cR_t^x\\
            &\quad + \brk{{\frac{1-\trw}{15}}+ \frac{2+\trw}{3} + 480c_0\eta_a^2Q^2L^2(1-\beta)}{\frac{15c_0\teta^2}{(1-\trw)^2}}\cR_t^y\\
            &\quad + \brk{\frac{(1-\beta)^2}{4(1-\trw)} + \frac{1+\beta}{2}}{\frac{480c_0^2\teta^2}{(1-\trw)^2}}\E\brk{\norm{\z_t - \nabla F(\1\bar{x}_t^{\T})}^2}\\
            &\quad + \brk{{\frac{1845c_0^2\teta^2L^2(1-\beta)^3}{(1-\trw)^3}} + {\frac{7200c_0^2\teta^2 L^2(1-\beta)}{(1-\trw)^2}} + \beta}\frac{\teta^2n}{(1-\beta)^2}\E\brk{\norm{\bar{z}_t}^{2}}\\
            &\quad + {\prt{\frac{1-\beta}{1-\trw} + 3} \frac{2400c_0^2n\teta^2 \eta_a^2 Q^2 L^2(1-\beta)}{(1-\trw)^2}}\E\brk{\norm{\nabla f(\bar{x}_t)}^2}\\
            &\quad + \brk{{\frac{1845c_0^2\teta^2 L^2(1-\beta)^4}{3(1-\trw)^3}} + {\frac{2400c_0^2\teta^2L^2(1-\beta)^2}{(1-\trw)^2}} + 1 }\frac{3\teta^2n}{1-\beta}\E\brk{\norm{\bavg_t}^2} + \frac{2\teta^2\sigma^2}{Q}\\
            &\quad  + {\frac{90c_0^2\teta^2 (1-\beta)^2 n\sigma^2}{Q(1-\trw)^2}}\brk{{5\prt{2 + \frac{3\eta_a^2 Q^2 L^2}{1-\trw}}} + 16\prt{1 + \frac{15\eta_a^2Q^2L^2}{1-\beta} + \frac{5\teta^2L^2}{n(1-\beta)}}}.
        \end{aligned}
    \end{equation}
    {Noting that $\teta\leq (1-\beta)/(200c_0 L)$, $\eta_s\geq (1-\trw)^{3/2}/\sqrt{1-\beta}$, $\teta = \eta_a \eta_s Q$, and $\beta\geq \trw$ leads to $\eta_a\leq 1/(200c_0 Q L)$}.
    Simplifying the coefficients in \eqref{eq:tcH_s1} yields
    \begin{equation}
        \label{eq:tcH_s2}
        \begin{aligned}
            \tcH_{t + 1} &\leq \frac{3+\beta}{4}\tcH_t + \eta_a^2 Q^2Ln(1-\beta){\E\brk{f(\bar{x}_t) - f^*} }
            + \frac{4\teta^2n}{1-\beta}\E\brk{\norm{\bavg_t}^2}\\
            &\quad + \frac{3\teta^2\sigma^2}{Q} + {\frac{2340c_0^2\teta^2 (1-\beta)^2 n \sigma^2}{Q(1-\trw)^2}} + {\frac{22950c_0^2\teta^2 \eta_a^2 Q L^2 (1-\beta)n\sigma^2}{(1-\trw)^2}},
        \end{aligned}
    \end{equation}
    where we invoked $\E\brki{\normi{\nabla f(\bar{x}_t)}^2}\leq 2L\E\brki{(f(\bar{x}_t) - f^*)}$. 

    Based on \eqref{eq:bavg_ub1}, \eqref{eq:sum_ub}, \eqref{eq:tcH}, {and $\eta_s \geq (1-\trw)^{3/2}/\sqrt{1-\beta}$}, we obtain
    \begin{equation}
        \label{eq:bavg_ub}
        \begin{aligned}
            \E\brk{\norm{\bavg_t}^2} 
            &\leq \frac{30L^2}{n}\tcH_t + 10L\tcL_t + 18\eta_a^2QL^2\sigma^2.
        \end{aligned}
    \end{equation}
    
    To bound $\tcL_t$, unrolling \eqref{eq:tcL_pl} yields
    \begin{equation}
        \label{eq:tcL_ub}
        \begin{aligned}
            \tcL_t&\leq \prt{1-\frac{\teta\mu}{3}}^{t}\tcL_0  + \frac{3\teta L\sigma^2}{n\mu Q} + \frac{{27}\teta^2 L^2\sigma^2}{n\mu Q(1-\beta)}+\frac{{33}\eta_a^2QL^2\sigma^2}{\mu} + {\frac{9c_0'\teta^2 L^2(1-\beta)^2\sigma^2}{Q\mu(1-\trw)^3}}.
        \end{aligned}
    \end{equation}
    { Substituting \eqref{eq:pl_f_ub}, \eqref{eq:bavg_ub}, and \eqref{eq:tcL_ub} into \eqref{eq:tcH_s2} yields a recursion whose coefficient on $\tcH_t$ is at most $(4+\beta)/5$ under \eqref{eq:teta_pl}. Unrolling this recursion yields

    \begin{equation}
        \label{eq:tcD_ub}
        \begin{aligned}
            &\tcH_{t} \leq \prt{\frac{4+\beta}{5}}^{t}\tcH_0+ \frac{360\teta^2\eta_a^2 QL^2 n\sigma^2}{(1-\beta)^2} + \frac{15\teta^2\sigma^2}{Q(1-\beta)} \\
            &\quad + \brk{2\eta_a^2 Q^2Ln(1-\beta) + \frac{40\teta^2nL}{1-\beta}}\prt{1-\frac{\teta\mu}{3}}^{t-1}\tcL_0\cdot\sum_{j=0}^{t-1}\prt{\frac{(4+\beta)/5}{1-\teta\mu/3}}^{j}\\
            & + 5\eta_a^2 Q^2 Ln\brk{\frac{39\teta L\sigma^2}{2 n\mu Q} + \frac{54\teta^2L^2\sigma^2}{n\mu Q(1-\beta)} + \frac{66\eta_a^2QL^2\sigma^2}{\mu} + \frac{18c_0'\teta^2 L^2(1-\beta)^2\sigma^2}{Q\mu (1-\trw)^3}}\\
            & + \frac{200\teta^2nL}{(1-\beta)^2}\brk{\frac{3\teta L\sigma^2}{n\mu Q} + \frac{{27}\teta^2 L^2\sigma^2}{n\mu Q(1-\beta)}+\frac{{33}\eta_a^2QL^2\sigma^2}{\mu} + {\frac{9c_0'\teta^2 L^2(1-\beta)^2\sigma^2}{Q\mu(1-\trw)^3}}} \\
            &  + {\frac{11700c_0^2\teta^2 (1-\beta) n \sigma^2}{Q(1-\trw)^2}} + {\frac{114750c_0^2\teta^2 \eta_a^2 Q L^2 n\sigma^2}{(1-\trw)^2}}.
        \end{aligned}
    \end{equation}
    }
    Since \eqref{eq:teta_pl} implies $\teta\leq(1-\beta)/(5\mu)$, we have 
    \begin{align*}
        \sum_{j=0}^{t-1}\prt{\frac{(4+\beta)/5}{1-\teta\mu/3}}^{j}\leq \frac{15}{1-\beta}.
    \end{align*}
    Using this bound and \eqref{eq:teta_pl}, we simplify \eqref{eq:tcD_ub} to
    {
    \begin{equation}
        \label{eq:tcD_ub_sim}
        \begin{aligned}
            &\tcH_{t} \leq \prt{\frac{4+\beta}{5}}^{t}\tcH_0+  \prt{1-\frac{\teta\mu}{3}}^{t-1}\frac{n\tcL_0}{L}+ \frac{3\eta_a^2 QL n\sigma^2}{\mu} + \frac{11800c_0^2\teta^2 (1-\beta) n \sigma^2}{Q(1-\trw)^2} \\
            & + \brk{\frac{200\teta^2nL}{(1-\beta)^2} + 200\eta_a^2Q^2Ln}\cdot\brk{\frac{3\teta L\sigma^2}{n\mu Q} + {\frac{9c_0'\teta^2 L^2(1-\beta)^2\sigma^2}{Q\mu(1-\trw)^3}}}  + \frac{16\teta^2 L \sigma^2}{Q\mu (1-\beta)} .
        \end{aligned}
    \end{equation}

    To obtain a sharper bound on $\cR_t^x$, we return to the coupled recursions for $\cR_t^x$ and $\cR_t^y$. The definition of $\tcH_t$ \eqref{eq:tcH} gives
    \begin{equation}
        \label{eq:zs_ub}
        \begin{aligned}
            \E\brk{\norm{\bar{z}_t}^2} \leq \frac{(1-\beta)^2}{\teta^2 n}\tcH_t,\quad \E\brk{\norm{\z_t - \nabla F(\1\bar{x}_t^{\T})}^2} \leq \frac{(1-\trw)^2}{480c_0^2\teta^2}\tcH_t,\qquad t\geq 0.
        \end{aligned}
    \end{equation}
    Substituting \eqref{eq:bavg_ub}, \eqref{eq:zs_ub}, and $\normi{\nabla f(\bar{x}_t)}^2\leq 2L\brki{f(\bar{x}_t) - f^*}$ into \eqref{eq:cRy_sim} leads to 
    \begin{equation}
        \label{eq:cRy_noz}
        \begin{aligned}
            &\cR_{t + 1}^y \leq \frac{2+\trw}{3}\cR_t^y + \frac{2214 c_0\teta^2\eta_a^2 QL^4(1-\beta)^3n\sigma^2}{1-\trw}\\
            &\quad + \brk{\frac{(1-\beta)^2(1-\trw)}{60c_0} + \frac{123c_0\teta^2 L^2 (1-\beta)^3}{1-\trw} + \frac{3690c_0\teta^4L^4(1-\beta)^3}{1-\trw} }\frac{\tcH_t}{\teta^2}\\ 
            &\quad+ \frac{420c_0(1-\beta)^2L^2}{1-\trw}\cR_t^{x} + \frac{30c_0(1-\beta)^2n\sigma^2}{Q}\prt{2 + \frac{3\eta_a^2 Q^2 L^2}{1-\trw}}\\
            &\quad + \frac{320 c_0n\eta_a^2Q^2L^3(1-\beta)^2}{1-\trw} \E\brk{f(\bar{x}_t) - f^*}+ \frac{1230c_0n\teta^2L^3(1-\beta)^3}{1-\trw}\tcL_t.
        \end{aligned}
    \end{equation}

    Similar derivations as in \eqref{eq:tcH_s1} yield
    \begin{equation}
        \label{eq:cons_contract}
        \begin{aligned}
            &\cR_{t + 1}^x + \frac{15c_0\teta^2}{(1-\trw)^2}\cR_{t + 1}^y 
            \leq \frac{3+\trw}{4}\brk{\cR_t^x +\frac{15c_0\teta^2}{(1-\trw)^2}\cR_t^y}  + \frac{2 (1-\beta)^2\tcH_t}{(1-\trw)} \\
            &\quad + \frac{1350c_0^2\teta^2(1-\beta)^2n\sigma^2}{(1-\trw)^2Q}\prt{1 + \frac{\eta_a^2 Q^2 L^2}{1-\trw}} + \frac{18450c_0^2n\teta^4L^3(1-\beta)^3}{(1-\trw)^3}\tcL_t\\
            &\quad + \frac{4800 c_0^2n\teta^2\eta_a^2Q^2L^3(1-\beta)^2}{(1-\trw)^3} \E\brk{f(\bar{x}_t) - f^*},
        \end{aligned}
    \end{equation}
    where we invoked \eqref{eq:teta_pl}. 
    Substituting \eqref{eq:pl_f_ub}, \eqref{eq:tcL_ub},
    and \eqref{eq:tcD_ub_sim} into \eqref{eq:cons_contract}, unrolling the resulting recursion, and bounding the geometric sums under \eqref{eq:teta_pl} yields
    \begin{equation}
        \label{eq:cRx_ub}
        \begin{aligned}
            &\cR_t^x +\frac{15c_0\teta^2}{(1-\trw)^2}\cR_t^y
            \leq \prt{\frac{3+\trw}{4}}^t\brk{c_0\norm{\Pi\x_0}^2 +\frac{15c_0^2\teta^2\E\brk{\normi{\Pi\y_0}^2}}{(1-\trw)^2}} \\
            &\quad + \frac{5400c_0^2\teta^2(1-\beta)^2n\sigma^2}{(1-\trw)^3Q}\prt{1 + \frac{\eta_a^2 Q^2 L^2}{1-\trw}} + \frac{80(1-\beta)^2}{(1-\trw)^2} \prt{\frac{4+\beta}{5}}^{t-1}\tcH_0\\\
            &\quad + \frac{32n}{L}\prt{1 - \frac{\teta \mu}{3}}^{t-1}\tcL_0 + \frac{32\eta_a^2 QL n\sigma^2}{\mu} + \frac{100000c_0^2\teta^2 (1-\beta)^2 n \sigma^2}{Q(1-\trw)^3} \\
            &\quad + 2000\brk{\frac{\teta^2nL}{(1-\beta)^2} + \eta_a^2Q^2Ln}\cdot\brk{\frac{3\teta L\sigma^2}{n\mu Q} + {\frac{9c_0'\teta^2 L^2(1-\beta)^2\sigma^2}{Q\mu(1-\trw)^3}}}\\
            &\quad + \frac{128\teta^2 L \sigma^2(1-\beta)}{Q\mu (1-\trw)^2},
        \end{aligned}
    \end{equation}
    where we invoked \eqref{eq:teta_pl} and used $\beta\geq (9\trw - 1)/8$ due to $\beta\geq\trw$. Combining \eqref{eq:cRx_ub} and \eqref{eq:pl_f_ub} yields the desired result \eqref{eq:pl_ub}.

   Finally, \eqref{eq:teta_pl_order} and the inequality
    $\exp\prti{-a/(b+c)}\leq \exp\prti{-a/(2b)}+\exp\prti{-a/(2c)}$ for $a,b,c>0$ yield
    \begin{align*}
        &\prt{1 - \frac{\teta\mu}{3}}^{T-1}\leq \exp\prt{-\frac{\teta\mu (T-1)}{3}}
        \leq \frac{1}{nQT} + \exp\prt{-\frac{(T-1)(1-\trw)}{1200c_0\kappa}}.
    \end{align*}
    }

{
\subsection{Proof of Lemma \ref{lem:lca_lmt_embedding}}
\label{app:lca_lmt_embedding}
For $Q=1$, the single local step in \eqref{eq:dsmtl} gives $\br_t = \g_t$.
Hence \eqref{eq:dsmtl_localz} is precisely \eqref{eq:lca_m}. Substituting $\y_t$ and $\y_t^{(l)}$ into \eqref{eq:bct} yields
\begin{equation}
    \label{eq:bct_re}
    \bc_{t + 1} = (1 + \eta_w)W\bc_t - \eta_w\bc_{t-1} - (1 + \eta_w)(I-W)\z_{t + 1}.
\end{equation}
Multiplying \eqref{eq:lca_d} by $\mathrm{diag}(I-W,I-W)$ gives 
\begin{equation}
    \label{eq:dt_re}
    \begin{aligned}
        (I-W) \d_{t + 1} 
        &=(1 + \eta_w)W(I-W)\d_t - \eta_w(I-W)\d_t^{(l)} - (1 + \eta_w)(I-W)\z_{t + 1},\\
        (I-W)\d_{t + 1}^{(l)} &= (I-W)\d_t,
    \end{aligned}
\end{equation}
where we used the fact that $(I-W)W=W(I-W)$. Comparing \eqref{eq:bct_re} and \eqref{eq:dt_re} implies that $\bc_t = (I-W)\d_t$. The coincidence of the generated $\crki{\x_t}$ of \eqref{eq:dsmtl} and \eqref{eq:lca_family} follows directly from this observation.
}

{
\subsection{Proof of Lemma \ref{lem:sparkle_plus_master}}
\label{app:sparkle_plus_master}
    Unrolling \eqref{eq:sparkle_plus_master_momentum} and \eqref{eq:sparkle_plus_master_state} gives $z_{t+1}^{(j)} =(1-\beta)\sum_{\ell=0}^{t} \beta^{t-\ell}\hat\xi_{\ell}^{(j)}$ and $s_t^{(j)} =-\alpha\sum_{k=0}^{t-1} \mathsf T_j^{t-1-k}\mathsf h_jz_{k+1}^{(j)}$, respectively.
It follows from \eqref{eq:sparkle_plus_master_readout} that
\begin{align*}
    v_j^{\T}\x_T &= \mathsf{r}_js_T^{(j)} =-\alpha(1-\beta) \sum_{k=0}^{T-1} \mathsf{r}_j\mathsf T_j^{T-1-k}\mathsf h_j \sum_{\ell=0}^{k} \brk{\beta^{k-\ell}\hat\xi_{\ell}^{(j)}}\\
    &= -\alpha(1-\beta) \sum_{\ell=0}^{T-1}\sum_{k=\ell}^{T-1} \beta^{k-\ell} \mathsf{r}_j\mathsf T_j^{T-1-k}\mathsf h_j \hat\xi_{\ell}^{(j)},
\end{align*}
where we interchanged the order of summation. Let $q:=T-1-\ell$ and $r:=T-1-k$. Then $0\leq q\leq T-1$ and $0\leq r\leq q$, which gives
\begin{equation}
\label{eq:sparkle_plus_master_noise_convolution}
    v_j^{\T}\x_T
    =-\alpha\sum_{q=0}^{T-1}
    R_{j,q}\hat\xi_{T-1-q}^{(j)},
    \quad
    R_{j,q}:=(1-\beta)\sum_{r=0}^{q} \beta^{q-r}\mathsf{r}_j\mathsf{T}_j^r \mathsf{h}_j = (1-\beta)\sum_{r=0}^{q} \beta^{q-r}\psi_{j,r}.
\end{equation}

Because the vectors $v_j$ are orthonormal and belong to $\mathrm{Range}(\Pi)$, Bessel's inequality gives
\(
    \normi{\Pi\x_T}^2 \geq\sum_{j\in\mathcal S}(v_j^{\T}\x_T)^2.
\)
Moreover, \eqref{eq:sparkle_plus_master_noise} gives $\E\brki{\hat\xi_t^{(j)}\hat\xi_s^{(j)}} =\sigma^2\delta_{ts}.$
It follows from \eqref{eq:sparkle_plus_master_noise_convolution} that
\begin{align*}
    \E\brk{(v_j^{\T}\x_T)^2} 
    &= \alpha^2 \sum_{q=0}^{T-1}\sum_{p=0}^{T-1} R_{j,q}R_{j,p}\E\brk{\hat\xi_{T-1-q}^{(j)}\hat\xi_{T-1-p}^{(j)}}= \alpha^2\sigma^2\sum_{q=0}^{T-1}R_{j,q}^2.
\end{align*}

For $0\leq q<H = \min\crki{T,K}$, we have $q\leq K-1$.  Thus, $\psi_{j,k}\geq \underline{\psi}>0$ for $k=0,1,\ldots, K-1$ implies $R_{j,q}\geq\underline{\psi}(1-\beta)\sum_{r=0}^{q}\beta^{q-r}  =\underline{\psi}(1-\beta^{q+1}).$
Hence, 
\begin{align*}
    \E\brk{\norm{\Pi\x_T}^2} &\geq
    \alpha^2\sigma^2M\underline{\psi}^{\,2}
    \sum_{r=1}^{H}(1-\beta^r)^2 + \alpha^2 \sigma^2\sum_{j\in\mathcal{S}} \sum_{q=H}^{T-1} R_{j,q}^2\geq \alpha^2\sigma^2M\underline{\psi}^{\,2}
    \sum_{r=1}^{H}(1-\beta^r)^2.
\end{align*}

It remains to lower bound $(1-\beta^{r})$. We use the following inequalities 
\begin{equation}
\label{eq:sparkle_elementary_exponential_bounds}
    1-x\leq e^{-x},\quad
    1-e^{-x} \geq \frac{x}{1+x},\qquad x\geq0.
\end{equation}
Both follow from $e^x\geq1+x$. Moreover, $x/(1 + x)\geq \min\crki{1,x}/2$ for all $x\geq0$.
In light of \eqref{eq:sparkle_elementary_exponential_bounds}, we have
\[
    1-\beta^r =1-[1-(1-\beta)]^r
    \geq1-e^{-(1-\beta)r}
    \geq\frac{(1-\beta)r}{1+(1-\beta)r}
    \geq\frac{\min\crki{1,(1-\beta)r}}{2}.
\]
For any $r = \ceili{H/2}, \ldots, H$, we have $r \geq H/2$.
Therefore,
\begin{align*}
    1 - \beta^r \geq \frac{\min\crki{1, (1-\beta)H/2}}{2}\geq \frac{\min\crki{1,(1-\beta)H}}{4}.
\end{align*}
There are at least $H/2$ such indices, and consequently
\[
    \sum_{r=1}^{H}(1-\beta^r)^2\geq \frac{H}{2}\brk{\frac{1}{4}\min\crk{1, (1-\beta)H}}^2 = \frac{H}{32}\min\crk{1, [(1-\beta)H]^2},
\]
which proves \eqref{eq:sparkle_plus_master_finite_lb}.

    }

{
\subsection{Proof of Lemma \ref{lem:mixing_matrix}}
\label{app:mixing_matrix}

The construction of $W_{n,\theta}$ is based on a fixed symmetric, nonnegative, positive definite, and stochastic matrix $H_m\in\R^{m\times m}$ with distinct eigenvalues and an infinite family of connected $d$-regular spectral expanders $\{\mathcal{G}_\ell\}_{\ell\geq1}$ \cite{reingold2002entropy} with $k_\ell:= |\mathcal{V}(\mathcal{G}_\ell)|$.
Specifically, fix a constant $\varrho\in(0,1/4]$ and define $\varepsilon_\ell:= \theta/\gamma_\ell$ for some $\gamma_\ell$ introduced later. Let $\otimes$ denote the Kronecker product. For $\theta$ sufficiently small, the constructed $W_{n,\theta}$ is given by
\begin{equation}
\label{eq:sparkle_hard_matrix_construction}
    W_{n,\theta}
    :=
    (1-\varrho-\varepsilon_\ell)I_{k_\ell m}
    +
    \varrho(I_{k_\ell}\otimes H_m)
    +
    \varepsilon_\ell(E^{(\ell)}\otimes I_m),
    \qquad
    n=k_\ell m.
\end{equation}

We now specify the choice of $H_m$ and $\crki{E^{(\ell)}}_{\ell\geq 1}$. 
Let $H_m\in\R^{m\times m}$ be a fixed symmetric, nonnegative, positive definite, and stochastic matrix with distinct eigenvalues
\begin{equation}
    \label{eq:nu}
    1=\nu_1>\nu_2>\cdots>\nu_m>0.
\end{equation}
Such a matrix exists. For example, if $L_{P_m}$ is the graph Laplacian of the path on $m$ vertices, then
\[
    H_m:=I_m-\chi L_{P_m},\quad
    0<\chi<\frac14,
\]
has all the stated properties. 
Indeed, 
the eigenvalues of $L_{P_m}$ are
\(
    \lambda_s(L_{P_m})
    =
    2-2\cos((s-1)\pi/m),\;
    s=1,\ldots,m,
\)
see, e.g., \cite[1.4.4]{brouwer2011spectra}. Thus the eigenvalues of $H_m$ are distinct and positive. Moreover, $L_{P_m}$ is symmetric, positive semidefinite, and satisfies $L_{P_m}\1=0$. Therefore
$H_m$ is symmetric, nonnegative, positive definite, and doubly stochastic.

Let $\{\mathcal{G}_\ell\}_{\ell\geq1}$ be an infinite family of connected $d$-regular expanders, where $d$ is a constant independent of $D$, $\theta$, $\ell$, and $k_\ell=|\mathcal V(\mathcal{G}_\ell)|\to\infty$; see, e.g., \cite[Theorem 1.4]{reingold2002entropy}. Let $A_\ell$ be the adjacency matrix of $\mathcal G_\ell$. The eigenvalues of $A_\ell/d$ satisfy
\(
    1=\zeta_1^{(\ell)}
    >
    \zeta_2^{(\ell)}
    \geq\cdots\geq
    \zeta_{k_\ell}^{(\ell)}
    \geq-1,\;
    \max_{r\geq2}|\zeta_r^{(\ell)}|
    \leq\rho_E<1,
\)
where $\rho_E$ is independent of $\ell$. Define 
\[
    E^{(\ell)} := \frac12\prt{I_{k_\ell}+\frac{A_\ell}{d}}.
\]
Then, $(E^{(\ell)})^{\T} = E^{(\ell)}$, $E^{(\ell)} \1 = (\1 + d\1/d)/2 = \1$, and its eigenvalues are
\(
    \mu_r^{(\ell)}
    :=
    (1+\zeta_r^{(\ell)})/2,
    \;
    r=1,\ldots,k_\ell,
\)
and hence
\(
    1=\mu_1^{(\ell)}
    >
    \mu_2^{(\ell)}
    \geq\cdots\geq
    \mu_{k_\ell}^{(\ell)}
    \geq0.
\)
Moreover,
\begin{align}
    \label{eq:gamma_star}
    \gamma_\ell := 1-\mu_2^{(\ell)} = \frac{1-\zeta_2^{(\ell)}}{2} \geq \frac{1-\rho_E}{2}=: \gamma_*>0.
\end{align}
Thus, $E^{(\ell)}$ is symmetric, nonnegative, positive semidefinite, and doubly stochastic, with a spectral gap uniformly bounded below by $\gamma_*$. Adding the identity changes only the diagonal support. Since $\mathcal G_\ell$ is $d$-regular, each vertex has at most
$d$ distinct neighbors other than itself. Therefore, the maximum degree of the support graph of $E^{(\ell)}$, defined as $d_E:=\sup_{\ell\geq1} d_{\max}(\mathrm{supp}(E^{(\ell)}))$, satisfies $d_E\leq d<\infty$.
By the choice of the expander family, both $d_E$ and $\gamma_*$ are independent of
$D$, $\theta$, $\ell$, $k_\ell$, and $n=k_\ell m$.

We are now ready to construct the mixing matrix $W_{n,\theta}$ based on $H_m$ and $E^{(\ell)}$ as in \eqref{eq:sparkle_hard_matrix_construction}.
The set $\mathbb{N}_D := \crk{k_\ell(D+1):\ell\geq1}$ is infinite because $k_\ell\to\infty$.
To ensure that \eqref{eq:sparkle_hard_matrix_construction} is a convex combination, we require that 
\begin{align*}
    \varrho,\quad \varepsilon_\ell = \frac{\theta}{\gamma_\ell},\quad 1-\varrho-\varepsilon_\ell \in [0,1],
\end{align*}
which suffices to impose $0< \theta \leq (1-\varrho)\gamma_*$. Since \eqref{eq:sparkle_hard_matrix_construction} is a convex combination of symmetric, nonnegative, positive semidefinite, and doubly stochastic matrices, $W_{n,\theta}$ has the same properties. Moreover, since both $\mathcal G_\ell$ and $P_m$ are connected, so is $\mathrm{supp}(W_{n,\theta})$.

\textbf{Verifying the degree bound.} Concretely, the construction in \eqref{eq:sparkle_hard_matrix_construction} yields a $k_\ell \times k_\ell$ block matrix, where each block is of size $m\times m$. Denote $D_{i,i}:= (1 - \varrho - \varepsilon_\ell)I_m + \varrho H_m + \varepsilon_\ell[E^{(\ell)}]_{i,i}I_m$. Then, $W_{n,\theta}$ can be expressed as
\small
\begin{align}
    \label{eq:Wntheta}
    W_{n,\theta} = \mathrm{diag}\prt{D_{11}, D_{22},\ldots, D_{k_\ell k_\ell}} + 
    \begin{pmatrix}
        \mathbf{0} & \varepsilon_\ell[E^{(\ell)}]_{1,2}I_m & \cdots & \varepsilon_\ell[E^{(\ell)}]_{1,k_\ell}I_m\\
        \varepsilon_\ell[E^{(\ell)}]_{2,1}I_m & \mathbf{0} & \cdots & \varepsilon_\ell[E^{(\ell)}]_{2,k_\ell}I_m\\
        \vdots & \vdots & \ddots & \vdots\\
        \varepsilon_\ell[E^{(\ell)}]_{k_\ell, 1}I_m & \varepsilon_\ell[E^{(\ell)}]_{k_\ell, 2}I_m & \cdots & \mathbf{0}
    \end{pmatrix}.
\end{align}
\normalsize
To reflect the block structure of $W_{n,\theta}$, we index the vertices of $\mathrm{supp}(W_{n,\theta})$ by $(i,s)\in[k_\ell]\times[m]$, where $[m]:= \crki{1,2,\ldots, m}$. Specifically, partition the vertex set $[n]$ into $\mathcal{V}_i:= \crki{(i-1)m + s: s\in[m]}$ for $i\in[k_\ell]$. We can identify each vertex $(i-1)m+s\in\mathcal V_i$ with the pair $(i,s)\in[k_\ell]\times[m]$. For two distinct vertices $(i,s)$ and $(j,t)$, we have the following observations.
\begin{itemize}
    \item If $i=j$, i.e., the two vertices belong to the same block $\mathcal{V}_i$, then it only appears in the first part of $W_{n,\theta}$ in \eqref{eq:Wntheta}. The corresponding off-diagonal entry is $\varrho[H_m]_{s,t}$ for $s\neq t$. The number of within-block neighbors of \((i,s)\) is at most the maximum degree of $\mathrm{supp}(H_m)$, which is at most $D$.
    \item If $i\neq j$, i.e., the two vertices belong to different blocks, then it only appears in the block $\varepsilon_\ell[E^{(\ell)}]_{i,j}I_m$. The corresponding off-diagonal entry is $\varepsilon_\ell[E^{(\ell)}]_{i,j}$ for $s=t$. The number of cross-block neighbors of \((i,s)\) is at most the maximum degree of $\mathrm{supp}(E^{(\ell)})$, which is at most $d_E$.
\end{itemize}
In summary, the corresponding off-diagonal entry of $W_{n,\theta}$ is
\[
    \brk{W_{n,\theta}}_{(i,s),(j,t)}=\varrho\brk{H_m}_{s,t} \delta_{ij}+\varepsilon_\ell
    \brk{E^{(\ell)}}_{i,j}\delta_{st}.
\]
Consequently, the maximum degree of $\mathrm{supp}(W_{n,\theta})$ is bounded by $(D+d_E)$, which is independent of $n$ and $\theta$.

\textbf{Verifying the properties of eigenvalues.}
The matrices $I_{k_\ell}\otimes H_m$ and $E^{(\ell)}\otimes I_m$ commute due to $(A\otimes B)(C\otimes D) = (AC)\otimes (BD)$. If $(\bs{u}_i^{(\ell)}, \mu_i^{(\ell)})$ and $(\bs{v}_s, \nu_s)$ are orthonormal eigenpairs of $E^{(\ell)}$ and $H_m$, respectively, then the eigenvalues of $W_{n,\theta}$ are 
\begin{equation}
\label{eq:sparkle_hard_matrix_eigenvalues}
    w_{i,s}^{(\ell)}
    =
    1-\varrho(1-\nu_s)
    -\varepsilon_\ell(1-\mu_i^{(\ell)}),
    \qquad
    1\leq i\leq k_\ell,\quad 1\leq s\leq m,
\end{equation}
with corresponding eigenvectors $\bs{u}_i^{(\ell)}\otimes \bs{v}_s$.

We now discuss the magnitude of $w_{i,s}^{(\ell)}$ below. 
\begin{itemize}
    \item For $s = 1$ and $i=1$, we choose $\bs{u}_1^{(\ell)}=\1_{k_\ell}/\sqrt{k_\ell}$ and $\bs{v}_1=\1_{m}/\sqrt m$, so $\bs{u}_1^{(\ell)}\otimes \bs{v}_1=\1_{n}/\sqrt n$ is the unique consensus eigenvector and $w_{1,1}^{(\ell)}=1-\varrho(1-\nu_1)-\varepsilon_\ell(1-\mu_1^{(\ell)})=1.$ We define $\tau_1 := w_{1,1}^{(\ell)} = 1.$
    \item When $s=1$ and $i\geq2$, using $\nu_1=1$ gives $1-w_{i,1}^{(\ell)}= \varepsilon_\ell(1-\mu_i^{(\ell)}).$ Since $ \gamma_\ell = 1-\mu_2^{(\ell)} \leq 1-\mu_i^{(\ell)} \leq 1$, $\varepsilon_\ell=\theta/\gamma_\ell$, and $\gamma_\ell\geq \gamma_*$, these eigenvalues obey
    \begin{equation}
    \label{eq:sparkle_slow_eigenvalues}
        \theta \leq 1-w_{i,1}^{(\ell)} \leq \frac{\theta}{\gamma_*},\; i=2,\ldots,k_\ell.
    \end{equation}
    Moreover, we have $ w_{2,1}^{(\ell)}=1-\theta$. Thus, $W_{n,\theta}$ has $(k_\ell-1)$ eigenvalues that are close to $1$ when $\theta$ is small. Denote them as $\omega_r = w_{r,1}^{(\ell)},\; r=2,\ldots,k_\ell$. Moreover, the corresponding eigenvectors $\bs{u}_r^{(\ell)}\otimes \bs{v}_1\in \mathrm{Range}(\Pi)$ due to $\inproi{\bs{u}_r^{(\ell)}\otimes \bs{v}_1, \1_{k_\ell}\otimes \1_m} = \inproi{\bs{u}_r^{(\ell)}, \1_{k_\ell}}\inproi{\bs{v}_1, \1_m} = 0.$
    \item For $s\geq 2$ and $i = 1$, by the construction of $H_m$ and $\nu_s$ in \eqref{eq:nu}, the eigenvalues $w_{1,s}^{(\ell)}$ are distinct and independent of $n$, $\theta$, and $\ell$. We denote them as 
    \begin{equation}
    \label{eq:sparkle_spectral_anchors}
        \begin{aligned}
            \tau_s &:= w_{1,s}^{(\ell)} = 1-\varrho(1-\nu_s),\; s=2,\ldots,m.
        \end{aligned}
    \end{equation}
    Then, $1 = \tau_1>\tau_2>\tau_3>\cdots>\tau_m>0$.
    \item For $s\geq 2$ and $i\geq 2$, we have 
        \(
            w_{i,s}^{(\ell)} \leq 1-\varrho(1-\nu_s) \leq 1-\varrho(1-\nu_2)
        \)
        due to $1-\mu_i^{(\ell)}\geq 0$. Let $\theta_D:=\min\crk{(1-\varrho)\gamma_*, \varrho(1-\nu_2)}>0.$
        Then, for every $\theta\in(0,\theta_D]$, all eigenvalues with $s\geq2$ are at most $1-\theta$.
\end{itemize}

It can be seen that $\Pi W_{n,\theta}$ has eigenvalue zero on the consensus eigenvector $\bs{u}_1^{(\ell)}\otimes \bs{v}_1$ and eigenvalue $w_{i,s}^{(\ell)}\geq0$ on every other eigenvector. Consequently,
\[
    \norm{\Pi W_{n,\theta}}_2=\max_{(i,s)\neq(1,1)} w_{i,s}^{(\ell)}=w_{2,1}^{(\ell)}=1-\theta.
\]
}

{
\subsection{Proof of Theorem \ref{thm:poly_lca_lb}}
\label{app:poly_lca_lb}
    For any fixed $D\geq 1$, Lemma~\ref{lem:mixing_matrix} provides a set $\mathbb{N}_D$ and a mixing matrix $W_{n,\theta}\in\R^{n\times n}$ with $n\in\mathbb{N}_D$ and $\theta\in(0,\theta_D]$. 
    It remains to show that the LCA momentum-tracking family admits a realization of the form \eqref{eq:sparkle_plus_master_realization} and the corresponding quantities satisfy the conditions of Lemma~\ref{lem:sparkle_plus_master}. Both are established below. 

    \textbf{Dynamics of the LCA momentum-tracking family.} Define $\hat{\d}_t:= \mathrm{diag}(I-P,I-P)\td_t$. The update of $\hat{\d}_t$ can be derived from \eqref{eq:lca_d} following a similar derivation as in \eqref{eq:dt_re} and $(I-P)P = P(I-P)$. Specifically, we have
    \begin{equation}
        \label{eq:lca_hd}
        \begin{aligned}
            \hat{\d}_{t+1} 
            &= \mathsf{L}_{\gamma,P}\hat{\d}_t + (\mathsf{L}_{\gamma,P} - I_{2n})\mathsf{E}\z_{t + 1},
        \end{aligned}
    \end{equation}
    Substituting \eqref{eq:lca_y} into \eqref{eq:lca_x} and noting $(I-P)\mathsf{E}_1^{\T} = \mathsf{E}_1^{\T}\mathrm{diag}(I-P, I-P)$ leads to 
    \begin{equation}
        \label{eq:lca_tx_t}
        \tx_{t + 1} = \mathsf{L}_{\gamma,P}\tx_t - \alpha \mathsf{L}_{\gamma,P}\mathsf{E}\z_{t + 1} - \alpha \mathsf{L}_{\gamma,P}\mathsf{E}\mathsf{E}_1^{\T}\hat{\d}_t.
    \end{equation}

    Let $(\bs{u}, w)$ be an eigenpair of $W_{n,\theta}$ satisfying $\bs{u}\in\mathrm{Range}(\Pi)$ and $1-w\leq \theta/\gamma_*$. Then, $P\bs{u} = \varphi(w)\bs{u}$. Define $\tilde{\bs{u}}^{\T}=\mathrm{diag}(\bs{u},\bs{u})^{\T}$, $e_1^{\T} = (1,0)\in\R^2$, and $e^{\T} = (1,1)\in\R^2$. We have $\tilde{\bs{u}}^{\T}\mathsf{E} = e \bs{u}^{\T}$, $\bs{u}^{\T}\mathsf{E}_1^{\T} = e_1^{\T}\tilde{\bs{u}}^{\T}$, and 
    \begin{align*}
        \tilde{\bs{u}}^{\T} \mathsf{L}_{\gamma,P} &= \begin{pmatrix}
            (1 + \gamma)\varphi(w)\bs{u}^{\T} & -\gamma \bs{u}^{\T}\\
            \bs{u}^{\T} & 0
        \end{pmatrix} = \begin{pmatrix}
            (1 + \gamma)\varphi(w) & -\gamma\\
            1 & 0
        \end{pmatrix}\tilde{\bs{u}}^{\T}=:\mathsf{L}_{\gamma,\varphi}\tilde{\bs{u}}^{\T}.
    \end{align*}
    Multiplying \eqref{eq:lca_tx_t} by $\tilde{\bs{u}}^{\T}$ on both sides and noting the above relation yields
    \begin{equation}
        \label{eq:lca_txu}
        \begin{aligned}
            \tilde{\bs{u}}^{\T}\tx_{t + 1} &=  \mathsf{L}_{\gamma,\varphi}\tilde{\bs{u}}^{\T}\tx_t - \alpha \mathsf{L}_{\gamma,\varphi} e \bs{u}^{\T}\z_{t + 1} - \alpha \mathsf{L}_{\gamma,\varphi} e e_1^{\T}\tilde{\bs{u}}^{\T}\hat{\d}_t.
        \end{aligned}
    \end{equation}
    Similarly, multiplying \eqref{eq:lca_hd} by $\tilde{\bs{u}}^{\T}$ on both sides gives
    \begin{equation}
        \label{eq:lca_hdu}
        \begin{aligned}
            \tilde{\bs{u}}^{\T}\hat{\d}_{t + 1} &= \mathsf{L}_{\gamma,\varphi}\tilde{\bs{u}}^{\T}\hat{\d}_t + (\mathsf{L}_{\gamma,\varphi} - I_2) e \bs{u}^{\T}\z_{t + 1}.
        \end{aligned}
    \end{equation}

    Consider the instance in \eqref{eq:fi0}. Define $z_t(w):= \bs{u}^{\T}\z_t$, $\hat{\xi}_t(w):= \bs{u}^{\T}\xi_t$, and 
    \begin{align*}
        \tilde{s}_t(w) := \begin{pmatrix}
            \tilde{\bs{u}}^{\T}\tx_t\\
            \alpha\tilde{\bs{u}}^{\T}\hat{\d}_t
        \end{pmatrix}.
    \end{align*}
    For $k\geq 0$, denote
    \begin{align*}
        \mathsf{T}_{\gamma,\varphi}&:= \begin{pmatrix}
                \mathsf{L}_{\gamma,\varphi} & - \mathsf{L}_{\gamma,\varphi} e e_1^{\T}\\
                \mathbf{0} & \mathsf{L}_{\gamma,\varphi}
                \end{pmatrix},\quad 
                &\mathsf{h}_{\gamma,\varphi}:= \begin{pmatrix}
                \mathsf{L}_{\gamma,\varphi}e \\
                -(\mathsf{L}_{\gamma,\varphi} - I_2)e 
            \end{pmatrix},\\
            \mathsf{r}&:= (1,0,0,0) = (e_1^{\T}, \mathbf{0}),\quad &\psi_{\gamma,\varphi,k}:= \mathsf{r}\mathsf{T}_{\gamma,\varphi}^k\mathsf{h}_{\gamma,\varphi}.
    \end{align*}
    Then, \eqref{eq:lca_m}, \eqref{eq:lca_txu}, and \eqref{eq:lca_hdu} can be rewritten as
    \begin{subequations}
        \label{eq:lca_u}
        \begin{align}
            z_{t + 1}(w) &= \beta z_t(w) + (1 - \beta)\hat{\xi}_t(w),\\
                \tilde{s}_{t + 1}(w) &= \mathsf{T}_{\gamma,\varphi}\tilde{s}_t(w) - \alpha \mathsf{h}_{\gamma,\varphi}  z_{t + 1}(w),\\
            \bs{u}^{\T}\x_t &= e_1^{\T}\tilde{\bs{u}}^{\T}\tx_t = \mathsf{r}\tilde{s}_t.
        \end{align}
    \end{subequations}
    The dynamics in \eqref{eq:lca_u} is precisely of the form in \eqref{eq:sparkle_plus_master_realization}. It remains to show that 
    \begin{align*}
        \psi_{\gamma,\varphi,k}\geq \underline{\psi}>0,
    \end{align*}
    for any eigenpair $(\bs{u}, w)$ of $W_{n,\theta}$ satisfying $\bs{u}\in\mathrm{Range}(\Pi)$ and $1-w\leq \theta/\gamma_*$, and $k=0,1,\ldots, K-1$ for some $K\geq 1$.

    \textbf{Lower Bounding $\psi_{\gamma,\varphi,k}$.} The upper triangular structure of $\mathsf{T}_{\gamma,\varphi}$ implies that
    \begin{align*}
        \mathsf{T}_{\gamma,\varphi}^k = \begin{pmatrix}
            \mathsf{L}_{\gamma,\varphi}^k & -\sum_{j=0}^{k-1}\mathsf{L}_{\gamma,\varphi}^{k-j} e e_1^{\T}\mathsf{L}_{\gamma,\varphi}^j\\
            \mathbf{0} & \mathsf{L}_{\gamma,\varphi}^k
        \end{pmatrix}.
    \end{align*}
    Hence, we have 
    \begin{equation}
        \label{eq:lca_psi}
        \begin{aligned}
            \psi_{\gamma,\varphi,k} &= e_1^{\T}\mathsf{L}_{\gamma,\varphi}^{k+1}e +\sum_{j=0}^{k-1}e_1^{\T}\mathsf{L}_{\gamma,\varphi}^{k-j} e e_1^{\T}\mathsf{L}_{\gamma,\varphi}^j (\mathsf{L}_{\gamma,\varphi}-I_2)e.
        \end{aligned}
    \end{equation}
    Noting that $(\mathsf{L}_{\gamma,\varphi}-I_2)e = -(1 + \gamma)(1-\varphi(w))e_1=: -(1+\gamma)\theta' e_1$ and $\mathsf{L}_{\gamma,\varphi}^{r} - I_2 = \sum_{\ell=0}^{r-1}\prt{\mathsf{L}_{\gamma,\varphi}^{\ell + 1} -  \mathsf{L}_{\gamma,\varphi}^{\ell} } = \sum_{\ell=0}^{r-1}\mathsf{L}_{\gamma,\varphi}^{\ell}\prt{ \mathsf{L}_{\gamma,\varphi} -  I_2}$,
    we further decompose \eqref{eq:lca_psi} as
    \begin{equation}
        \label{eq:lca_psi_decomp}
        \begin{aligned}
            \psi_{\gamma,\varphi,k} - 1
            &= e_1^{\T}\prt{\mathsf{L}_{\gamma,\varphi}^{k+1} - I_2} e + \sum_{j=0}^{k-1}\brk{1 + e_1^{\T}\prt{\mathsf{L}_{\gamma,\varphi}^{k-j} - I_2} e}\brk{ e_1^{\T}\mathsf{L}_{\gamma,\varphi}^j (\mathsf{L}_{\gamma,\varphi}-I_2)e},
        \end{aligned}
    \end{equation}
    where we used $e_1^{\T} e = 1$.
    It then suffices to upper bound $\normi{\mathsf{L}_{\gamma,\varphi}^{k + 1}}$ in order to bound $| \psi_{\gamma,\varphi,k} - 1 |$.

    The characteristic polynomial of $\mathsf{L}_{\gamma,\varphi}$ is given by $p_{\mathsf{L}_{\gamma,\varphi}}(\lambda) = \lambda^2 -(1 + \gamma)\varphi(w) \lambda + \gamma,\;\varphi(w)\in[0,1].$
    When $\varphi(w) \in[0,1)$, by the Jury criterion \cite{choo2011elementary}, $\mathsf{L}_{\gamma,\varphi}$ is Schur stable due to
    \begin{equation}
        \label{eq:jury_phi}
        \begin{aligned}
            &1-(1 + \gamma)\varphi(w) + \gamma = (1 - \varphi(w))(1 + \gamma)>0,\\
            &1+(1 + \gamma)\varphi(w) + \gamma = (1 + \varphi(w))(1 + \gamma)>0,\quad
            0<\gamma<1.
        \end{aligned}
    \end{equation}
    Consequently, both roots of $p_{\mathsf{L}_{\gamma,\varphi}}(\lambda)$ are strictly less than one in magnitude for $\varphi(w)\in[0,1)$. 
    When $\varphi(w) = 1$, the eigenvalues of $\mathsf{L}_{\gamma,\varphi}$ are $1$ and $\gamma$. In summary, denote both roots of $p_{\mathsf{L}_{\gamma,\varphi}}(\lambda)$ as $\lambda_1$ and $\lambda_2$. Then, $|\lambda_1|,|\lambda_2|\leq 1$. By Cayley-Hamilton theorem, we have $\mathsf{L}_{\gamma,\varphi}^2 = (1 + \gamma)\varphi(w)\mathsf{L}_{\gamma,\varphi} - \gamma I_2$.
    Consequently, we have
    \begin{align*}
        \mathsf{L}_{\gamma,\varphi}^k = (1 + \gamma)\varphi(w)\mathsf{L}_{\gamma,\varphi}^{k-1} - \gamma \mathsf{L}_{\gamma,\varphi}^{k-2},\quad k\geq 2.
    \end{align*}
    Consider the recursion of $\crki{a_k}$: 
    \begin{align}\label{eq:ak}
        a_{k + 1} = (1+ \gamma)\varphi(w) a_k - \gamma a_{k - 1},\quad k\geq 0,\quad a_{-1}=0,\quad a_0 = 1.
    \end{align}
    It can be verified that $a_k = \sum_{i=0}^k\lambda_1^{k-i}\lambda_2^i$, $|a_k|\leq (k + 1)$, and
    \begin{align}
        \label{eq:sfL}
        \mathsf{L}_{\gamma,\varphi}^k = a_{k-1}\mathsf{L}_{\gamma,\varphi} - \gamma a_{k-2}I_2,\quad k\geq 1,\quad \mathsf{L}_{\gamma,\varphi}^0 = I_2.
    \end{align}
    Since $\normi{\mathsf{L}_{\gamma,\varphi}}_2\leq \normi{\mathsf{L}_{\gamma,\varphi}}_F = \sqrt{(1 + \gamma)^2\varphi(w)^2 + 1 + \gamma^2}\leq \sqrt{6},$
    we have from \eqref{eq:sfL} that
    \begin{align}
        \label{eq:L_ub}
        \norm{\mathsf{L}_{\gamma,\varphi}^k}_2\leq |a_{k-1}|\norm{\mathsf{L}_{\gamma,\varphi}}_2 + \gamma|a_{k-2}|\leq 4(k + 1),\quad k\geq 1,
    \end{align}
    and trivially for $k=0$.

    In light of \eqref{eq:L_ub}, we obtain from \eqref{eq:lca_psi_decomp} that
    \begin{align*}
        &|\psi_{\gamma,\varphi,k} - 1| 
        \leq 4(1 + \gamma)\theta'(k + 1)^2\\
        &\quad + \sum_{j=0}^{k-1}\brk{1 + (1+\gamma)\theta'\sum_{\ell=0}^{k-j-1}4(\ell + 1)}\cdot\brk{4(j + 1)(1+\gamma)\theta'}\\
        &\leq 4(1 + \gamma)\theta'(k + 1)^2 + 2(1+\gamma)\theta'\brk{1 + 2(1 + \gamma)\theta'(k+1)^2}(k + 1)^2.
    \end{align*}
    Noting that $\gamma<1$, we obtain for $(k + 1)\leq 1/(16\sqrt{\theta'})$ that 
    \begin{align*}
        |\psi_{\gamma,\varphi,k} - 1| \leq \frac{1}{2}.
    \end{align*}

    We now show that $\theta' =\orderi{\theta}$. Consider the $(D+1)$ distinct eigenvalues of $W_{n,\theta}$ (Lemma~\ref{lem:mixing_matrix}), denoted as $\tau_1=1>\tau_2>\cdots>\tau_{D+1}>0$. Then, $\varphi(\tau_s)\in[0,1]$, $s=1,2,\ldots, D+1$. Let $\ell_1,\ldots,\ell_{D+1}$ denote the Lagrange basis polynomials
    associated with the interpolation nodes $\tau_1,\ldots,\tau_{D+1}$. Specifically, 
    \begin{align*}
        \ell_s(u):= \prod_{\substack{t=1\\t\neq s}}^{D+1} \frac{u-\tau_t}{\tau_s-\tau_t},\quad
        \ell_s(\tau_t)=
        \begin{cases}
            1,&t=s,\\
            0,&t\neq s,
        \end{cases}
        \qquad s=1,\ldots,D+1.
    \end{align*} 
    Consequently, $\mathrm{deg}(\varphi)\leq D$ yields the exact representation of $\varphi$ (see, e.g., \cite{berrut2004barycentric}): $\varphi(u) = \sum_{s=1}^{D+1} \varphi(\tau_s)\ell_s(u)$ and $\varphi'(u) = \sum_{s=1}^{D+1} \varphi(\tau_s)\ell_s'(u).$
    Denote $\Lambda_D:=\sup_{u\in[0,1]} \sum_{s=1}^{D+1}|\ell_s'(u)|$. Then,
    \begin{equation}
        \label{eq:Lambda}
        \begin{aligned}
            &\sup_{u\in[0,1]} |\varphi'(u)| \leq \sup_{u\in[0,1]} \sum_{s=1}^{D+1}|\varphi(\tau_s)|\cdot |\ell_s'(u)|\leq \Lambda_D \prt{\max_{s=1,\ldots,D+1}|\varphi(\tau_s)|}\leq  \Lambda_D.
        \end{aligned}
    \end{equation}
    By the mean value theorem, we have 
    \begin{equation}
        \label{eq:poly_phi}
        \begin{aligned}
            1-\varphi(w) &= \varphi(1) - \varphi(w) \leq \sup_{u\in[w,1]}|\varphi'(u)|(1-w) \leq \Lambda_D (1-w).
        \end{aligned}
    \end{equation}
    Noting that $1-w\leq \theta/\gamma_*$ (Lemma~\ref{lem:mixing_matrix}), we have $\theta' \leq \Lambda_D\theta/\gamma_*=: C_\text{LCA} \theta$. To ensure $1/(16\sqrt{\theta'})\geq 1$, we set $\theta\leq \theta_D^{\text{LCA}}:=\min\crki{\theta_D, 1/(256C_{\text{LCA}})}$.
    
    Applying Lemma \ref{lem:sparkle_plus_master} with $\underline{\psi} = 1/2$, $M = [n/(D+1)-1]$, $K=\floori{1/(16\sqrt{C_\text{LCA}\theta})}$, $T_D^{\text{LCA}} = 1/(16\sqrt{C_\text{LCA}})$, $T\geq T_D^{\text{LCA}}/\sqrt{\theta}$, and $H=\min\crki{T, K}\geq 1/(32\sqrt{C_{\text{LCA}}\theta})$ leads to 
    \begin{equation}
        \label{eq:lca_cons0}
        \begin{aligned}
            \E\brk{\norm{\Pi\x_T}^2}&\geq \frac{\alpha^2\sigma^2 n H}{256(D+1)}\min\crk{1,[(1-\beta)H]^2}\\
            &\geq \frac{\alpha^2\sigma^2 n }{8192(D+1)\sqrt{C_\text{LCA}\theta}}\min\crk{1,\frac{(1-\beta)^2}{1024C_\text{LCA}\theta}},
        \end{aligned}
    \end{equation}
    where we used the fact that $u\mapsto u\min\crk{1,[(1-\beta)u]^2}$ is nondecreasing for $u\geq 0$. This completes the proof.

}

{
\subsection{Proof of Theorem \ref{thm:sparkle_lb}}
\label{app:sparkle_lb}
    The procedures are similar to those in the proof of Theorem \ref{thm:poly_lca_lb}. We first show that the SPARKLE family admits a modal realization of the form \eqref{eq:sparkle_plus_master_realization} satisfying the conditions of Lemma~\ref{lem:sparkle_plus_master}. Then, we apply Lemma~\ref{lem:sparkle_plus_master} to obtain the lower bound. 

    \textbf{Obtaining the SPARKLE dynamics under the given instance.}
    We choose $W_{n,\theta}$ from Lemma~\ref{lem:mixing_matrix}. Let the sparkle dynamics evolve over $W_{n,\theta}$ with the same instance as in \eqref{eq:fi0}. Let
    $\bs{\xi}_t=(\xi_{1,t},\xi_{2,t},\ldots,\xi_{n,t})^{\T}\in\R^n$.
    Then $\g_t=\bs{\xi}_t$ with $\E[\bs{\xi}_t]=0$ and $\E[\bs{\xi}_t\bs{\xi}_t^{\T}]=\sigma^2 I_n.$

    Let $v$ be a unit eigenvector of $W_{n,\theta}$ associated with an eigenvalue $w$ satisfying $v\in\mathrm{Range}(\Pi)$ and $1-w\leq \theta/\gamma_*$. Denote
    \(
        x_t(w):=v^\T\x_t,\; d_t(w):=v^\T\d_t,\; z_t(w):=v^\T\z_t,\;\text{and } \hat{\xi}_t(w):= v^{\T}\bs{\xi}_t.
    \)
    Since $Cv=c(w)v,\; Bv=b(w)v,\;Av=a(w)v,$ and $A$, $B$, $C$ are symmetric, multiplying both sides of \eqref{eq:sparkle} by $v^{\T}$ gives
    \begin{subequations}
    \label{eq:sparkle_projected_primal_dual_recursion}
    \begin{align}
        z_{t + 1}(w) &= \beta z_t(w) + (1-\beta) \hat{\xi}_t(w),\label{eq:sparkle_projected_momentum_recursion}\\
        x_{t+1}(w) &= c(w)x_t(w) - \alpha a(w)z_{t+1}(w) - b(w)d_t(w),\label{eq:sparkle_projected_primal_recursion}\\
        d_{t+1}(w) &= d_t(w)+b(w)x_{t+1}(w)\nonumber\\
        &= \brk{1 - b^2(w)}d_t(w) + b(w)c(w)x_t(w) - \alpha b(w)a(w)z_{t+1}(w).
        \label{eq:sparkle_projected_dual_recursion}
    \end{align}
    \end{subequations}
    Defining $(y_t(w))^{\T} := (x_t(w), d_t(w))$,
    we obtain the recursion
    \begin{equation}
    \label{eq:sparkle_eigen_coordinate_state_recursion}
        y_{t+1}(w) = \mathsf T(w)y_t(w) - \alpha \mathsf h(w)z_{t+1}(w),\quad v^{\T} \x_t = e_1^{\T}y_t(w),
    \end{equation}
    where
    \begin{equation}
    \label{eq:sparkle_eigen_coordinate_T_h}
        \mathsf T(w)
        :=
        \begin{pmatrix}
            c(w) & -b(w)\\
            b(w)c(w) & 1-b^2(w)
        \end{pmatrix},
        \qquad
        \mathsf h(w)
        :=
        \begin{pmatrix}
            a(w)\\
            b(w)a(w)
        \end{pmatrix}.
    \end{equation}
    Consequently, the dynamics of the SPARKLE family under the given instance is of the form in \eqref{eq:sparkle_plus_master_realization} with $\psi_{j,k} = \psi_{w,k} := e_1^{\T} \mathsf{T}^k(w) \mathsf{h}(w)$.

    \textbf{Bounding $|\psi_{w,k} - 1|$.} We have 
    \begin{equation}
        \label{eq:Rq_lb_s0}
        \begin{aligned}
            \left|\psi_{w,k} -1 \right| &= \left|e_1^{\T} \mathsf{T}^{k}(w) \mathsf{h}(w) - e_1^{\T}\mathsf{T}^k(w) e_1 +  e_1^{\T}\brk{\mathsf{T}^k(w) - I_2} e_1\right|\\
            &\leq \norm{\mathsf{T}^k(w)}_2\cdot \norm{\mathsf{h}(w) - e_1}_2 + \sum_{j=0}^{k-1}\norm{\mathsf{T}^j(w)}_2\cdot \norm{\mathsf{T}(w) - I_2}_2.
        \end{aligned}
    \end{equation}
    Then, it suffices to bound $ \norm{\mathsf{T}^k(w)}_2$,  $\norm{\mathsf{h}(w) - e_1}_2$, and $\norm{\mathsf{T}(w) - I_2}_2$, respectively. 
    \begin{itemize}
        \item Bounding $\normi{\mathsf T(w) - I_2}_2$. From \eqref{eq:sparkle_eigen_coordinate_T_h},
    \[ 
        \mathsf T(w)-I_2
        =
        \begin{pmatrix}
            c(w)-1 & -b(w)\\
            b(w)c(w) & -b(w)^2
        \end{pmatrix} = \begin{pmatrix}
            c(w)-1 & 0\\
            0 & -b(w)^2
        \end{pmatrix}
        +
        b(w)
        \begin{pmatrix}
            0 & -1\\
            c(w) & 0
        \end{pmatrix}.
    \]
    Because $C$ is symmetric, nonnegative, and doubly stochastic, every eigenvalue of $C$ belongs to $[-1,1]$, and hence $|c(w)|\leq1$.  Therefore,
    \[
        \norm{
            \begin{pmatrix}
                0 & -1\\
                c(w) & 0
            \end{pmatrix}
        }_2 =\max\{1,|c(w)|\}=1.
    \]
    The triangle inequality gives
    \begin{align}
        \|\mathsf T(w)-I_2\|_2
        &\leq
        \max\{|c(w)-1|,b(w)^2\}+|b(w)|\leq
        |c(w)-1|+b(w)^2+|b(w)|.
        \label{eq:sparkle_T_close_identity_delta}
    \end{align}
    \item Bounding $\normi{\mathsf{h}(w) - e_1}_2$. Since $A$ is symmetric and doubly stochastic, we have $|a(w)|\leq1$. Thus,
    \begin{align}
        \norm{\mathsf{h}(w) - e_1}_2
        &=
        \sqrt{(a(w)-1)^2+b(w)^2a(w)^2}
        \leq
        |a(w)-1|+|b(w)|.
        \label{eq:sparkle_h_close_e1_delta}
    \end{align}
    \item Bounding $\normi{\mathsf{T}^k(w)}_2$. We have 
    \begin{align}
    \label{eq:Tk}
    \norm{\mathsf{T}^k(w)}_2 \leq \norm{\mathsf{T}(w)}_2^k \leq \prt{\norm{\mathsf{T}(w) - I_2}_2 + \norm{I_2}_2}^k\leq \exp\prt{k\norm{\mathsf{T}(w) - I_2}_2 } .
    \end{align}
    \end{itemize}

It then suffices to bound $a(w)$, $c(w)$, and $s(w)$. The main tool is the Lagrange basis polynomials. 
Denote $C_1:=\Lambda_D/\gamma_*$. Following similar derivations as in \eqref{eq:poly_phi}, we have 
$$|a(w) - 1|  \leq C_1 \theta \quad\text{and}\quad |c(w) - 1|  \leq C_1 \theta.$$ 
To bound $|s(w)|$, let $v_1,\ldots,v_n$ be an orthonormal eigenbasis of
$W_{n,\theta}$ such that
\(
    W_{n,\theta}v_j=w_jv_j,\;
    v_1=\1/\sqrt{n},\;
    w_1=1,
\)
and $b(w_j)$ denote the
eigenvalue of $B$ corresponding to $v_j$. Then
\(
    Bv_j=b(w_j)v_j,\;
    b(w_j)^2=s(w_j).
\)
Moreover, 
\(
    \Pi v_j=\pi_jv_j,
    \;
    \pi_1=0,
    \;
    \pi_j=1,\; j=2,\ldots,n.
\)
Thus, after an orthogonal change of basis followed by a
permutation (see, e.g., \cite{huang2023abcrr,alghunaim2021unified,zhu2024sparkle}), the augmented matrix $\mathcal M$ can be written in block diagonal form
\begin{equation}
\label{eq:sparkle_augmented_eigenbasis_block_decomposition}
    \mathrm{diag}(\mathcal M_1(1),\mathcal M_2(w_2),\ldots,\mathcal M_2(w_n)),\quad \mathcal{M}_2(w_j) = \begin{pmatrix}
        \pi_jc(w_j) - s(w_j) & b(w_j)\\
        -b(w_j) & \pi_j
    \end{pmatrix}.
\end{equation}
Regarding $\mathcal{M}_1(1)$, $\mathrm{Null}(B)=\mathrm{Span}(\1)$ implies $b(1)=0$ and $s(1)=b(1)^2=0.$
Since $\pi_1=0$, we have $\mathcal M_1(1)=\mathbf{0}_{2\times2}.$
For any $w=w_j$, $j\geq2$, we have $\pi_j=1$, and hence
\begin{equation}
\label{eq:sparkle_disagreement_eigenvalue_block}
    \mathcal M_2(w)
    :=
    \begin{pmatrix}
        c(w)-s(w) & b(w)\\
        -b(w) & 1
    \end{pmatrix},
\end{equation}
whose characteristic polynomial is $\zeta^2-\prt{1+c(w)-s(w)}\zeta+ c(w)$. The assumption $\rho(\mathcal M)<1$ and the Jury criterion \cite{choo2011elementary} therefore imply $|c(w)|<1$ and $0<s(w)<2(1+c(w))<4.$
Similar derivations as in \eqref{eq:poly_phi} yield
\begin{equation}
    \label{eq:poly_s}
    |s(w)| = |s(w) - s(1)| \leq 4 C_1 \theta.
\end{equation}
Finally, using $b(w)^2=s(w)$ leads to 
\begin{equation}
\label{eq:poly_b}
    |b(w)| = \sqrt{|s(w)|}\leq 2\sqrt{C_1 \theta}.
\end{equation}

Noting $5C_1\sqrt{\theta} + 2\sqrt{C_1}\leq C_2:= 5C_1 + 2\sqrt{C_1}<\infty$, substituting $|a(w)-1|\leq C_1\theta$, $|c(w)-1|\leq C_1\theta$, and $|b(w)|\leq 2\sqrt{C_1 \theta}$ into \eqref{eq:sparkle_T_close_identity_delta}--\eqref{eq:Tk} yields
\begin{align*}
    |\psi_{w,k} - 1| &\leq \sqrt{\theta}\exp\prt{k\sqrt{\theta}C_2}\prt{C_1\sqrt{\theta} + 2\sqrt{C_1}} + \sqrt{\theta}C_2 \sum_{j=0}^{k-1}\prt{1 + \sqrt{\theta} C_2}^j\\
    &\leq \sqrt{\theta}C_2 \exp\prt{k\sqrt{\theta}C_2} + (1 + \sqrt{\theta}C_2)^k - 1\\
    &\leq \prt{\sqrt{\theta} C_2 + 1}\exp\prt{k\sqrt{\theta} C_2} - 1.
\end{align*}
Let 
\begin{align*}
    \theta_D^{\text{SP}}:= \min\crk{\theta_D, \frac{1}{64 C_2^2}},\quad T_D^{\text{SP}}:= \frac{\log(4/3)}{C_2},\quad K:= \floor{\frac{T_D^{\text{SP}}}{\sqrt{\theta}}} + 1.
\end{align*}
Hence, when $0<\theta\leq \theta_{D}^{\text{SP}}$ and $0\leq k\leq K-1$, we have 
\begin{align*}
    |\psi_{w,k} - 1|\leq \frac{1}{2},\quad k=0,1,\ldots, K-1.
\end{align*}

Applying Lemma~\ref{lem:sparkle_plus_master} with $\underline{\psi} = 1/2$, $M = n/(D+1) - 1$, $H\geq\floori{T_D^{\text{SP}}/\sqrt{\theta}}$ and noting that $u\mapsto u\min\crk{1,[(1-\beta)u]^2}$ is nondecreasing for $u\geq 0$ leads to the desired result.

}

{
\subsection{Proof of Theorem \ref{thm:sparkle_plus_ncvx_lb}}
\label{app:sparkle_plus_ncvx_lb}

    We first present a lemma that helps lower bound specific terms in the subsequent proof. 

    \begin{lemma}
    \label{lem:hk_pk_sum}
    Let $\beta\in[0,1)$ and $s>0$. Let $\crki{h_k}_{k\geq -1}$ satisfy $h_{-1}=0$, $h_0=1$, and $h_{k} = \brk{1+\beta-(1-\beta)s} h_{k-1} - \beta h_{k-2}$ for $k\geq 1$, and let $\{p_k\}_{k\geq -1}$ satisfy
    \begin{equation}
    \label{eq:pk_re}
        p_{-1}=p_0=1,\quad p_{k+1}= \brk{1+\beta-(1-\beta)s}p_k-\beta p_{k-1},\qquad k\geq 0.
    \end{equation}
    Then, we have 
    \begin{equation}
        \label{eq:hk_sum}
        \sum_{k=0}^{t-1} h_k^2 \geq C_h \min\crk{\frac{1}{(1-\beta)^3}, \frac{1}{(1-\beta)^2s}},\quad C_h:=\frac{\exp(-2)}{162\pi^2},\qquad t\geq \frac{1}{1-\beta}
    \end{equation}
    and 
    \begin{equation}
    \label{eq:pk_sum}
    \sum_{k=0}^{t-1}p_k^2\geq C_p\min\crk{t,\frac{1}{1-\beta}}, \quad C_p=\frac{\exp(-2)}{8},\qquad t\geq 1.
    \end{equation}
    Moreover, if $s\leq (1-\beta)/8$, then
    \begin{equation}
        \label{eq:pk_last}
        p_t\geq (1-2s)^t,\quad t\geq 0.
    \end{equation}
\end{lemma}
    \begin{proof}
        See Appendix~\ref{app:hk_pk_sum}.
    \end{proof}

    We establish the terms depending explicitly on $\alpha$ and $(1-\beta)$ on one structured hard instance using the matrix $W_{n,\theta}$ from Corollary~\ref{cor:sparkle_plus_lb} and a four-dimensional decision variable $x=(u,v,w,r)$. Define $\mu_\alpha:= \min\crki{L, 1/(4\alpha T)}$, $w^*:=\sqrt{\Delta_f/\mu_\alpha}$, $r^*:= \sqrt{\Delta_f/L}$ and 
    \begin{equation}
    \label{eq:sparkle_plus_ncvx_product_instance}
        \begin{aligned}
        f_i(u,v,w, r) &:= \frac{L}{2}u^2 +\frac{\mu_\alpha}{2}(w-w^*)^2 + \frac{L}{2}(r - r^*)^2,\\
        g_{i,t}(u,v,w, r)
        &:=
        \prt{
        Lu+\xi^u_{i,t},\;
        \xi^v_{i,t},\;
        \mu_\alpha(w-w^*),\;
        L(r - r^*)
        }^{\T},
        \end{aligned}
    \end{equation}
    where all $\xi^u_{i,t}$ and $\xi^v_{i,t}$ are mutually independent centered
    Gaussian random variables satisfying
    \(
        \E\brk{|\xi^u_{i,t}|^2}=\E\brk{|\xi^v_{i,t}|^2}=\sigma^2/2.
    \)
    This instance is $L$-smooth, has initial gap exactly $\Delta_f$ under zero
    initialization, and satisfies
    \(
    \E\brki{\normi{g_{i,t}(u,v,w,r)-\nabla f_i(u,v,w,r)}^2}=\sigma^2.
    \)

    Write the stacked iterate coordinatewise as $\x_t = (\mathbf u_t, \mathbf v_t, \mathbf w_t,\mathbf{r}_t)$ for $\mathbf u_t,\mathbf v_t,\mathbf w_t,\mathbf{r}_t\in\mathbb R^n$.
    Then every network operator $K\in\mathbb R^{n\times n}$ acts independently on the four coordinates $K\x_t = (K \mathbf{u}_t, K \mathbf{v}_t, K \mathbf{w}_t, K \mathbf{r}_t)$.
    Thus, network mixing does not couple the four decision coordinates.
    In addition, the constructed stochastic gradient has the separable form $(L\mathbf{u}_t +\bs{\xi}_t^u, \bs{\xi}_t^v, \mu_\alpha(\mathbf{w}_t - w^*\mathbf{1}), L(\mathbf{r}_t - r^*\mathbf{1}))$.
    Hence,
    the $u$-, $v$-, $w$, and $r$-recursions evolve independently.

\textbf{Analysis of the $u$ coordinate.} Both families have the same average recursion $\bar z_{t+1} = \beta\bar z_t + (1-\beta)\prti{L\bar u_t+\bar\xi^u_t}$ with $\E\brki{|\bar\xi^u_t|^2}=\sigma^2/(2n)$ and $\bar u_{t+1}= \bar u_t-\alpha\bar z_{t+1}$. 
Detailed proof can be seen in \cite[Lemma 2
]{huang2025accelerated}.
Eliminating $\bar z_t = (\bar{u}_{t-1} - \bar{u}_t)/\alpha$ leads to 
\begin{equation}
    \label{eq:baru}
    \bar{u}_{t + 1} = \brk{1 + \beta - (1-\beta)\alpha L}\bar{u}_t - \beta\bar{u}_{t-1} - \alpha(1-\beta)\bar\xi^u_t,\quad \bar{u}_0=0,\quad \bar{u}_{-1}:=0.
\end{equation} 

Let 
\begin{align*}
    \tilde{u}_t:= \begin{pmatrix}
        \bar{u}_t\\
        \bar{u}_{t-1}
    \end{pmatrix},\quad M:= \begin{pmatrix}
        1 + \beta - (1-\beta)\alpha L & -\beta \\
        1 & 0
    \end{pmatrix}.
\end{align*}
Then, \eqref{eq:baru} can be written as 
\begin{equation}
    \label{eq:tu}
    \begin{aligned}
        \tilde{u}_{t+1} &= M \tilde{u}_t - \alpha(1-\beta)e_1 \bar\xi^u_t
        = - \alpha(1-\beta)\sum_{j=0}^{t}M^{t-j}e_1\bar\xi^u_j.
    \end{aligned}
\end{equation}
Multiplying both sides of \eqref{eq:tu} by $e_1^{\T} = (1,0)$ yields $\bar{u}_{t} = - \alpha(1-\beta)\sum_{j=0}^{t-1}e_1^{\T}M^{t-1-j}e_1\bar\xi^u_j$.
By the independence of $\crki{\bar{\xi}_t^u}$, we have 
\begin{equation}
    \label{eq:baru_sq}
    \E\brk{|\bar{u}_t|^2} = \frac{\alpha^2(1-\beta)^2\sigma^2}{2n}\sum_{j=0}^{t-1}\prt{e_1^{\T}M^{t-1-j}e_1}^2 = \frac{\alpha^2(1-\beta)^2\sigma^2}{2n}\sum_{r=0}^{t-1} \prt{e_1^{\T}M^{r}e_1}^2.
\end{equation}

Denote $h_k:= e_1^{\T}M^k e_1$. By Cayley-Hamilton theorem, we can write down the recursion of $h_k$ as $ h_{k} = \brk{1 + \beta - (1-\beta)\alpha L} h_{k-1} - \beta h_{k-2}$ with $h_{-1}=0$, $h_0=1$, $k\geq 1.$
Applying \eqref{eq:hk_sum} with $\tau_1 = 1 + \beta - (1-\beta)\alpha L$ and $\tau_2 = \beta$ in Lemma~\ref{lem:hk_pk_sum} to \eqref{eq:baru_sq} leads to 
\begin{equation}
    \label{eq:baru_sq_lb}
    \begin{aligned}
        \E\brk{|L\bar{u}_t|^2} &\geq \frac{\alpha^2 L^2 (1-\beta)^2 \sigma^2}{2n}\sum_{k=0}^{t-1}h_k^2\geq \frac{C_h \alpha L \sigma^2}{2n}\min\crk{\frac{\alpha L}{1-\beta}, 1}, \quad t\geq \frac{1}{1-\beta}.
    \end{aligned}
\end{equation}

\textbf{Analysis of the $v$ coordinate.}
The $v$ coordinate is precisely the scalar instance in Corollary~\ref{cor:sparkle_plus_lb}, with oracle variance $\sigma^2/2$.
Applying Corollary~\ref{cor:sparkle_plus_lb} yields
\begin{equation}
\label{eq:sparkle_plus_product_disagreement_lb}
    \frac{L^2}{n} \E\brk{\norm{\Pi\x_t}^2} \geq \frac{L^2}{n}\E\brk{\norm{\Pi\mathbf{v}_t}^2}\geq
    \frac{C_v L^2\alpha^2\sigma^2}
        {(D+1)\sqrt{\theta}}
    \min\crk{1,\frac{(1-\beta)^2}{\theta}},\quad t\geq \frac{T_D^+}{\sqrt{\theta}},
\end{equation}
where $T_D^+ = \max\crki{T_D^{\text{LCA}},T_D^{\text{SP}}}$ and $C_v = \min\crki{C_D^{\text{LCA}}, C_D^{\text{SP}}}/2$.

\textbf{Analysis of the $w$ coordinate.}
The normalized error $p_t := (w^*-\bar w_t)/w^*$ satisfies
\eqref{eq:pk_re} with $s=\alpha\mu_\alpha$.
Since $\alpha\mu_\alpha\leq1/(4T)\leq(1-\beta)/8$ for $T\geq 2\ceili{\max\crki{T_D^+/\sqrt{\theta},1/(1-\beta)}}$, \eqref{eq:pk_last} gives
\begin{align*}
    \frac{w^*-\bar w_t}{w^*} \geq(1-2\alpha\mu_\alpha)^t\geq 1 - 2t\alpha\mu_\alpha \geq\frac12,\qquad 0\leq t\leq T.
\end{align*}
Consequently,
\begin{equation}\label{eq:sparkle_plus_stepsize_bias_lb}
    \E\brk{|\mu_\alpha(\bar{w}_t - w^*)|^2} \geq \frac{\mu_\alpha^2(w^*)^2}{4}\geq C_w \min\crk{L\Delta_f,\frac{\Delta_f}{\alpha T}},\quad C_w:= \frac{1}{16}.
\end{equation}

\textbf{Analysis of the $r$ coordinate.} Similar derivations as in \eqref{eq:baru} yields $\bar r_{t+1}= \brki{1+\beta-(1-\beta)\alpha L}\bar r_t -\beta\bar r_{t-1} +(1-\beta)\alpha Lr^\star.$
Define $\bar{r}_{-1}:= \bar{r}_0 = 0$ and let $p_t:= (r^* - \bar{r}_t)/r^*$. We obtain
\begin{align}
    \label{eq:barr_pk}
    p_{-1} = p_0 = 1,\quad  p_{t + 1} = \brk{1+\beta-(1-\beta)\alpha L}p_t -\beta p_{t-1},\qquad t\geq0.
\end{align}
Applying \eqref{eq:pk_sum} in Lemma~\ref{lem:hk_pk_sum} with $s = \alpha L$ to \eqref{eq:barr_pk} yields
\begin{equation}
    \label{eq:barr_sq_lb}
    \frac{1}{T}\sum_{t=0}^{T-1}\E\brk{\abs{L(\bar{r}_t - r^*)}^2} = \frac{L^2 (r^*)^2}{T}\sum_{t=0}^{T-1}p_t^2 \geq C_p \min\crk{L\Delta_f, \frac{L\Delta_f}{(1-\beta)T} }.
\end{equation}

Let $t_0 := \ceili{\max\crk{1/(1-\beta), T_D^+/\sqrt{\theta}}}$ and $T\geq 2t_0$.
In light of \eqref{eq:baru_sq_lb},
\eqref{eq:sparkle_plus_product_disagreement_lb}, 
\eqref{eq:sparkle_plus_stepsize_bias_lb}, and \eqref{eq:barr_sq_lb}, we conclude that 
\begin{equation}
    \label{eq:four_plus}
    \begin{aligned}
        &\mathfrak{E}^{(\text{NCVX})}\prt{\x_{0:(T-1)}}
        \geq \frac{C_h \alpha L\sigma^2}{4n}\min\crk{\frac{\alpha L}{1-\beta},1} + \frac{C_w}{2}\min\crk{L\Delta_f,\frac{\Delta_f}{\alpha T}}\\
        &\quad + \frac{C_vL^2\alpha^2\sigma^2}{2(D+1)\sqrt{\theta}}\min\crk{1, \frac{(1-\beta)^2}{\theta}} + C_p\min\crk{L\Delta_f, \frac{L\Delta_f}{(1-\beta)T}}.
    \end{aligned}
\end{equation}

\textbf{The overall lower bound.}
Fix $\mathsf A\in\mathfrak A_{D,n,\theta}^{+}(\alpha,\beta)$. Under zero initialization, every iteration of either comparison family can be implemented using at most $\orderi{D}$ primitive multiplications by $W$ and local linear combinations. These implementations satisfy the partial-averaging and zero-respecting policies of \cite{yuan2022revisiting}. Consequently, the lower bound in \cite{yuan2022revisiting} applies to the two comparison families of algorithms.

The proof of Theorem~2 in \cite{yuan2022revisiting} establishes its stochastic and network-dependent terms through separate hard instances. Both fall into the function class $\cF_{L}$, the oracle class $\mathcal{O}_{\sigma^2}$, and the network class $\mathcal{W}_{n,\theta}$. The stochastic construction depends only on the number of oracle queries and gives $ \Omega\prti{\sqrt{\nicefrac{\sigma^2 L\Delta_f}{nT}}}$.
The deterministic communication construction in \cite{yuan2022revisiting} gives $\Omega\prti{\nicefrac{L\Delta_f}{\sqrt{\theta}T}}.$
Recalling the $D$-degree polynomial network operator, the deterministic communication is therefore scaled by $\orderi{1/D}$. Hence, combining the two hard instances through the supremum and using \(\max\{a,b\}\geq(a+b)/2\), taking the output uniformly at random from $\crki{\bar{x}_t}_{t=0}^{T-1}$, the lower bound in \cite{yuan2022revisiting} becomes 
\begin{equation}
    \label{eq:yuan}
    \sup_{W\in\mathcal{W}_{n,\theta}}\sup_{\crki{f_i}\subset\cF_{L}}\sup_{\crki{g_i}\subset\mathcal{O}_{\sigma^2}} \frac{1}{T}\sum_{t=0}^{T-1}\E\brk{\norm{\nabla f(\bar{x}_t)}^2} \geq \Omega\prt{\sqrt{\frac{\sigma^2 L\Delta_f}{nT}} + \frac{L\Delta_f}{\sqrt{\theta} DT} },
\end{equation}
where $1-\theta\leq \cos\frac{\pi}{n}$.

The hard instances in \cite{yuan2022revisiting} and the construction in \eqref{eq:sparkle_plus_ncvx_product_instance} use different networks and objective functions. Since both belong to the domain of the same supremum, the supremum defining the minimax risk is at least the maximum of their lower bounds. Applying $\max\{a,b\}\geq(a+b)/2$, and then taking the infimum over
$\mathsf A$, proves the desired result.
}

{
\subsection{Proof of Theorem \ref{thm:sparkle_plus_pl_lb}}
\label{app:sparkle_plus_pl_lb}
    Let $w^* = \sqrt{\Delta_f/\rho}$, $\rho\in[\mu, L]$, and $r^* = \sqrt{\Delta_f/\mu}$. Consider the following four-dimensional instance
\begin{equation}
\label{eq:sparkle_plus_pl_quadratic_instance}
    \begin{aligned}
        f_i(u,v, w,r) &= \frac{L}{2}u^2  + \frac{\rho}{2}(w-w^*)^2 + \frac{\mu}{2}\prt{r - r^*}^2,\\
        g_{i,t}(u,v,w,r) &= \prt{
            L u+\xi_{i,t}^u,\;
            \xi_{i,t}^v,\;
            \rho(w-w^*),\;
            \mu(r-r^*)
        }^{\T},
    \end{aligned}
\end{equation}
where all $\xi_{i,t}^u$ and $\xi_{i,t}^v$ are independent centered Gaussian random variables with variance $\sigma^2/2$. Note that $f(u,v,w,r) = \sumn f_i(u,v,w,r)/n$ and
\begin{align*}
    \norm{\nabla f(u,v,w,r)}^2 &= L^2|u|^2 + \rho^2|w-w^*|^2 + \mu^2\abs{r - r^*}^2,\quad f^* = 0,\\
    2\mu\brk{f(u,v,w,r)-f^*} &= \mu L|u|^2 + \mu \rho|w-w^*|^2 + \mu^2\abs{r - r^*}^2\leq \norm{\nabla f(u,v,w,r)}^2.
\end{align*}
Hence, the instance is $L$-smooth and satisfies the $\mu$-PL condition for any $\rho\in[\mu, L]$.

\textbf{Analysis of the $u$-coordinate.} The $u$-coordinate has dynamics of $\bar{z}_{t+1}^u = \beta \bar{z}_t^u + (1-\beta)\prt{L\bar{u}_t+\bar{\xi}_t^u}$ and $\bar{u}_{t + 1} = \bar{u}_t - \alpha \bar{z}_{t+1}^u$,
which is the same as the $u$-coordinate dynamics in the proof of Theorem~\ref{thm:sparkle_plus_ncvx_lb}. Following the same derivations in \eqref{eq:baru_sq_lb}, we have 
\begin{equation}
    \label{eq:pl_Eu}
    \begin{aligned}
        \frac{L}{2}\E\brk{|\bar{u}_t|^2} \geq \frac{C_h \alpha \sigma^2}{4n}\min\crk{\frac{\alpha L}{1-\beta}, 1},\quad t\geq \frac{1}{1-\beta}.
    \end{aligned}
\end{equation}

\textbf{Analysis of the $v$-coordinate.} We have from Corollary~\ref{cor:sparkle_plus_lb} that
\begin{equation}
    \label{eq:pl_Ev}
    \E\brk{\norm{\Pi\x_t}^2} \geq
    \frac{C_v n\alpha^2\sigma^2}
        {(D+1)\sqrt{\theta}}
    \min\crk{1,\frac{(1-\beta)^2}{\theta}},\quad t\geq \frac{T_D^+}{\sqrt{\theta}}.
\end{equation}

\textbf{Analysis of the $w$-coordinate.} 
Let $\bar{w}_t'(\rho) = (w^* - \bar{w}_t)/w^*$. Following a similar derivation as in \eqref{eq:baru} to eliminate $\bar{z}_t^w$ leads to
\begin{equation}
    \label{eq:pl_wp}
    \bar{w}_{t+1}'(\rho) = \brk{1 + \beta - (1-\beta)\alpha\rho} \bar{w}_t'(\rho) - \beta\bar{w}_{t-1}'(\rho),\quad \bar{w}_{-1}'(\rho) = \bar{w}_0'(\rho) = 1.
\end{equation}
Moreover, $\bar{w}_t'(0) = 1$ and 
\begin{equation}
    \label{eq:w_cheb}
    \frac{\rho}{2}|\bar{w}_t - w^*|^2 = \frac{\Delta_f}{2} |\bar{w}_t'(\rho)|^2.
\end{equation}

Since we take the supremum over $\cF_{\mu,L}$, we define 
$\mathcal{P}(\alpha,\beta):= \Delta_f \sup_{\rho\in[\mu,L]} |\bar{w}_t'(\rho)|^2.$
For any fixed $\alpha$ and $\beta$, $\bar{w}_t'(\rho)$ is a polynomial of $\rho$ with degree at most $t$ and $\bar{w}_{-1}'(\rho) = \bar{w}_0'(\rho) = 1$. By the Chebyshev polynomial property \cite[Section 2.3]{d2021acceleration}, we have 
\begin{align}\label{eq:w_lb}
    \mathcal{P}(\alpha,\beta) \geq \Delta_f \prt{\frac{\sqrt{\kappa} - 1}{\sqrt{\kappa} + 1}}^{2t},\quad \kappa = \frac{L}{\mu}.
\end{align}

\textbf{Analysis of the $r$-coordinate.} Defining $p_t:= (r^* - \bar{r}_t)/r^*$ leads to the recursion in \eqref{eq:pk_re} with $s = \alpha \mu$. When $\alpha\mu\leq (1-\beta)/8$, we have from \eqref{eq:pk_last} that 
\begin{align*}
    p_t^2 \geq (1 - 2\alpha\mu)^{2t}\geq \exp\prt{-8\alpha\mu t},\quad t\geq 0.
\end{align*}
Consequently, we have
\begin{equation}
    \label{eq:pl_r_lb}
    \frac{\mu}{2}\abs{\bar{r}_T - r^*}^2 = \frac{\mu (r^*)^2}{2}p_T^2\geq \frac{\Delta_f}{2} \exp\prt{-8\alpha\mu T},\quad \alpha\leq \frac{1-\beta}{8\mu}.
\end{equation}

\textbf{The overall lower bound.} 
Following similar arguments as in the proof of Theorem~\ref{thm:sparkle_plus_ncvx_lb}, we can rewrite Theorem~3 in \cite{yuan2022revisiting} as 
\begin{equation}
\label{eq:sparkle_plus_yuan_pl_lb}
    \Omega\left(
        \frac{\sigma^2}{\mu nT}
        +
        \frac{\Delta_f \mu}{L}
        \exp\left[
            -C_D T\sqrt{\frac{\mu\theta}{L}}
        \right]
    \right), \quad \theta \geq 1- \cos\frac{\pi}{n},
\end{equation}
for some constant $C_D>0$ depending only on $D$ with $C_D=\orderi{D}$.

Let $t_0 := \ceili{\max\crk{1/(1-\beta), T_D^+/\sqrt{\theta}}}$ and $T\geq 2t_0$. In light of \eqref{eq:pl_Eu}, \eqref{eq:pl_Ev}, \eqref{eq:w_lb}, \eqref{eq:pl_r_lb}, and \eqref{eq:sparkle_plus_yuan_pl_lb}, we obtain the desired result.

}

{
\section{A Technical Lemma}
\subsection{Proof of Lemma \ref{lem:hk_pk_sum}}
\label{app:hk_pk_sum}
Denote $\tau_1:=1+\beta-(1-\beta)s$. Let $\lambda_1,\lambda_2\in\mathbb{C}$ satisfying $|\lambda_1|\geq|\lambda_2|$ be the roots of $P(\lambda):=\lambda^2-\tau_1\lambda+\beta$. Since $\lambda_1+\lambda_2=\tau_1$ and $\lambda_1\lambda_2=\beta$, it follows from induction that $h_k=\sum_{i=0}^k\lambda_1^{k-i}\lambda_2^i$, and hence $h_k=\lambda_1^k+\lambda_2h_{k-1}$, for $k\geq0$. We also have $h_{k + 1} = \tau_1h_k - \beta h_{k-1}$.
Indeed, the formula for $h_k$ holds for $k=-1$ (empty sum) and $k=0$. If it holds for $k-1$ and $k$ with some $k\geq0$, then
\begin{align*}
    h_{k+1} &= \tau_1h_k-\beta h_{k-1} = (\lambda_1+\lambda_2)\sum_{i=0}^{k}\lambda_1^{k-i}\lambda_2^i-\lambda_1\lambda_2\sum_{i=0}^{k-1}\lambda_1^{k-1-i}\lambda_2^i\\
    &= \sum_{i=0}^{k}\lambda_1^{k+1-i}\lambda_2^i+\sum_{i=0}^{k}\lambda_1^{k-i}\lambda_2^{i+1}-\sum_{i=0}^{k-1}\lambda_1^{k-i}\lambda_2^{i+1}\\
    &= \sum_{i=0}^{k}\lambda_1^{k+1-i}\lambda_2^i+\lambda_2^{k+1} = \sum_{i=0}^{k+1}\lambda_1^{k+1-i}\lambda_2^i.
\end{align*}

    Define $\tilde{p}_k:= h_{k} - \beta h_{k-1}$ for $k\geq 0$ and $\tilde{p}_{-1}  = 1$. Noting that
    \begin{align*}
        &\tau_1 \tilde{p}_k -\beta\tilde{p}_{k-1} 
        =\tau_1 h_{k} - \beta h_{k-1} -\beta \brk{\tau_1 h_{k-1} - \beta h_{k-2}} = h_{k + 1} - \beta h_{k} = \tilde{p}_{k+1}, \quad k\geq 1,\\
        &\tilde{p}_0 = h_0 - \beta h_{-1} = 1 - 0 = 1,\quad \tilde{p}_1 = h_1 - \beta h_0 = \tau_1 -\beta = p_1,
    \end{align*}
    we obtain that $\crki{\tilde{p}_k}_{k\geq -1}$ satisfies the same recursion as $\crki{p_k}_{k\geq -1}$ in \eqref{eq:pk_re}.
    Hence, $\crki{\tilde{p}_k}_{k\geq -1}$ has the same recursion and initial conditions as $\crki{p_k}_{k\geq -1}$.
   Therefore, it follows from $h_k=\lambda_1^k+\lambda_2h_{k-1}$ and $\beta=\lambda_1\lambda_2$ that
\begin{equation}
    \label{eq:pk_h}
    p_k=h_k-\beta h_{k-1}=\lambda_1^k+\lambda_2(1-\lambda_1)h_{k-1} ,\qquad k\geq0.
\end{equation}

    For $\beta\in[0,1/2)$, \eqref{eq:hk_sum} and \eqref{eq:pk_sum} follow from
    \begin{align*}
        \sum_{k=0}^{t-1}h_k^2&\geq h_0^2=1\geq 8C_h\geq\frac{C_h}{(1-\beta)^3}\geq C_h\min\crk{\frac{1}{(1-\beta)^3},\frac{1}{(1-\beta)^2s}},\\
        \sum_{k=0}^{t-1}p_k^2&\geq p_0^2=1\geq 2C_p\geq C_p\min\crk{t,\frac{1}{1-\beta}}.
    \end{align*}
    Moreover, \eqref{eq:pk_sum} also holds for $t=1$ and any $\beta\in[0,1)$. This is because $\sum_{k=0}^{0}p_k^2=p_0^2=1\geq C_p=C_p\min\{1,1/(1-\beta)\}$. Hence, assume $\beta\geq1/2$ and $t\geq2$, and let
    \begin{align*}
        m:=\min\crk{t, \floor{\frac{1}{1-\beta}}}\geq 2.
    \end{align*}
    Then
    \begin{align}
        m&\geq\frac12\min\left\{t,\frac{1}{1-\beta}\right\},
        \label{eq:m_large_beta}\\
        \beta^k &= \exp(k\log \beta) \geq \exp(-2k(1-\beta)) \geq \exp(-2),\quad k=0,1,\ldots,m-1.
        \label{eq:beta_geqh}
    \end{align}

    Note that $\lambda_1$ and $\lambda_2$ are the roots of a quadratic polynomial with positive product $\beta$. We consider two cases separately.

    \textbf{Case (1): $\lambda_1,\lambda_2\in\R$.} Since $\lambda_1\lambda_2=\beta>0$, the roots have the same sign, and the AM--GM inequality gives
    \begin{align*}
        |h_k| &= \sum_{j=0}^k |\lambda_1|^{k-j} |\lambda_2|^j \geq (k + 1) \prt{\prod_{j=0}^k |\lambda_1|^{k-j} |\lambda_2|^j}^{\frac{1}{k+1}}\\
        &= (k+1)\prt{|\lambda_1|^{\frac{k(k+1)}{2}}|\lambda_2|^{\frac{k(k+1)}{2}}}^{\frac{1}{k+1}}= (k + 1) \beta^{\frac{k}{2}},
    \end{align*}
    where the first equality holds since all terms $\lambda_1^{k-j}\lambda_2^j$ have the same sign as $\lambda_1^k$, and we used $|\lambda_1\lambda_2|=\beta$.
Hence, for $t\geq 1/(1-\beta)$, we have $m=\floori{1/(1-\beta)}\geq1/[2(1-\beta)]$ and
\begin{align}\label{eq:hk_case1}
    \sum_{k=0}^{t-1}h_k^2&\geq\sum_{k=0}^{m-1}(k+1)^2\beta^k\geq\exp(-2)\sum_{k=1}^{m}k^2\geq\frac{\exp(-2)m^3}{3}\geq\frac{\exp(-2)}{24(1-\beta)^3}.
\end{align}

Moreover, $|p_k|\geq|\lambda_1|^k$ for $k\geq0$. To see this, let $\mu_j:=|\lambda_j|$, $j=1,2$, so that $\mu_1\geq\mu_2>0$ and, since the two roots have the same sign, $\lambda_j=\mathrm{sgn}(\lambda_1)\mu_j$. Then
\begin{align*}
    h_{k-1}=\sum_{i=0}^{k-1}\lambda_1^{k-1-i}\lambda_2^i=\mathrm{sgn}(\lambda_1)^{k-1} \sum_{i=0}^{k-1}\mu_1^{k-1-i}\mu_2^i,
\end{align*}
so that \eqref{eq:pk_h} becomes
\begin{align*}
    p_k&=\mathrm{sgn}(\lambda_1)^k\mu_1^k+\mathrm{sgn}(\lambda_1)\mu_2(1-\lambda_1)\cdot\mathrm{sgn}(\lambda_1)^{k-1}\sum_{i=0}^{k-1}\mu_1^{k-1-i}\mu_2^i\\
    &=\mathrm{sgn}(\lambda_1)^k\brk{\mu_1^k+\mu_2(1-\lambda_1)\sum_{i=0}^{k-1}\mu_1^{k-1-i}\mu_2^i}.
\end{align*}
It remains to show $1-\lambda_1>0$, which is trivial if $\lambda_1<0$. If $\lambda_1>0$, then $P(1)=1-\tau_1+\beta=(1-\beta)s>0$ implies that $1$ does not lie between the two roots, and the axis of symmetry $\lambda=\tau_1/2$ of the parabola $y=P(\lambda)$ satisfies $\tau_1/2=[1+\beta-(1-\beta)s]/2<1$. Hence both roots are smaller than $1$. Consequently, $|p_k|\geq\mu_1^k=|\lambda_1|^k$. Since $|\lambda_1|\geq|\lambda_2|$ gives $|\lambda_1|^2\geq|\lambda_1||\lambda_2|=\lambda_1\lambda_2=\beta$, we obtain $|p_k|\geq\beta^{k/2}$, and therefore
\begin{align}\label{eq:pk_case1}
    \sum_{k=0}^{t-1}p_k^2&\geq\sum_{k=0}^{m-1}p_k^2\geq\sum_{k=0}^{m-1}\beta^k\geq\exp(-2)m\geq\frac{\exp(-2)}{2}\min\crk{t,\frac{1}{1-\beta}},
\end{align}
where we used $m\leq t$, \eqref{eq:beta_geqh}, and \eqref{eq:m_large_beta}, respectively.

\textbf{Case (2): $\lambda_1,\lambda_2\in\mathbb{C}\setminus\mathbb{R}$.} Since $\lambda_1$ and $\lambda_2$ are complex conjugates with $\lambda_1\lambda_2=\beta$, we have $|\lambda_1|^2=|\lambda_2|^2=\lambda_1\bar{\lambda}_1=\lambda_1\lambda_2=\beta$, and thus we can write
\begin{align*}
    \lambda_{1,2}=\sqrt{\beta}\exp\prti{\pm\mathrm{i}\vartheta} = \sqrt{\beta}\brk{\cos\vartheta \pm \mathrm{i}\sin\vartheta},\qquad \vartheta\in(0,\pi),
\end{align*}
where $\mathrm{i}^2=-1$. We have $\cos\vartheta=(\lambda_1+\lambda_2)/(2\sqrt{\beta}) = \tau_1/(2\sqrt{\beta})$, since $\lambda_1+\lambda_2=2\sqrt{\beta}\cos\vartheta$. Let $s_k:=\sin(k\vartheta)/\sin\vartheta$, $k\geq0$. The identity $\sin((k+1)\vartheta)=\sin(k\vartheta)\cos\vartheta+\cos(k\vartheta)\sin\vartheta$ gives
\begin{equation}
    \label{eq:sk_rec}
    s_{k+1}=s_k\cos\vartheta+\cos(k\vartheta),\quad s_0=0,\quad\text{hence}\quad s_m=\sum_{k=0}^{m-1}\cos^{m-1-k}\vartheta\,\cos(k\vartheta).
\end{equation}

Then
\begin{align*}
    h_k&=\frac{\lambda_1^{k+1}-\lambda_2^{k+1}}{\lambda_1-\lambda_2}=\frac{2\mathrm{i}\beta^{\frac{k+1}{2}}\sin((k+1)\vartheta)}{2\mathrm{i}\beta^{\frac{1}{2}}\sin\vartheta}=\beta^{\frac{k}{2}}s_{k+1},
\end{align*}
where we used  $h_k=\sum_{i=0}^k\lambda_1^{k-i}\lambda_2^i$ with $\lambda_1\neq\lambda_2$ and $\lambda_1^{j}-\lambda_2^{j}=2\mathrm{i}\beta^{\frac{j}{2}}\sin(j\vartheta)$ for $j\geq0$. 
Recall that $p_k=h_k-\beta h_{k-1}$. We obtain
\begin{align*}
    p_k&=h_k-\beta h_{k-1}=\beta^{\frac{k}{2}}\prt{s_{k+1}-\sqrt{\beta}s_k}=\beta^{\frac{k}{2}}\brk{\cos(k\vartheta)-\prt{\sqrt{\beta}-\cos\vartheta}s_k}\\
    &=\beta^{\frac{k}{2}}\brk{\cos(k\vartheta)-B\sin(k\vartheta)},\qquad B:=\frac{\sqrt{\beta}-\cos\vartheta}{\sin\vartheta}.
\end{align*}

\emph{Bounding $\sum_{k=0}^{t-1} h_k^2$.} Let $t\geq1/(1-\beta)$, so $m=\floori{1/(1-\beta)}\geq 1/[2(1-\beta)]$. We have
\begin{align*}
    \sum_{k=0}^{t-1}h_k^2\geq\sum_{k=0}^{m-1}h_k^2=\sum_{k=0}^{m-1}\beta^k s_{k + 1}^2\geq \exp(-2)\sum_{k=0}^{m-1}s_{k+1}^2=\exp(-2)G_m,
\end{align*}
where $G_m:=\sum_{k=1}^{m}s_k^2$. Since $\sin(k(\pi-\vartheta))=\sin(k\pi)\cos(k\vartheta)-\cos(k\pi)\sin(k\vartheta)=(-1)^{k+1}\sin(k\vartheta)$, $G_m$ and $\sin^2\vartheta$ are unchanged when $\vartheta$ is replaced by $\pi-\vartheta$. Hence, we assume $\vartheta\in(0,\pi/2]$ without loss of generality, which allows us to use $\sin x\geq 2x/\pi$ for $x\in[0,\pi/2]$.

If $\vartheta\leq\pi/m$, then $k\vartheta\leq m\vartheta/2\leq\pi/2$ for $k\leq\floori{m/2}$. We have $\sin(k\vartheta)\geq2k\vartheta/\pi$.  Noting $\sin x\leq x$, we obtain $s_k\geq 2k/\pi$ for $k\leq\floori{m/2}$. Hence,
\begin{align*}
    G_m\geq \frac{4}{\pi^2}\sum_{k=1}^{\lfloor m/2\rfloor}k^2\geq\frac{4}{3\pi^2}\floor{\frac{m}{2}}^3\geq\frac{4m^3}{81\pi^2}\geq\frac{1}{162\pi^2(1-\beta)^3},
\end{align*}
where we used $\sum_{k=1}^{n}k^2\geq n^3/3$, $\floori{m/2}\geq m/3$ for $m\geq2$, and $m\geq1/[2(1-\beta)]$.

If $\vartheta>\pi/m$, then $\sin\vartheta\geq 2\vartheta/\pi>2/m$. By the double-angle formula $\sin^2x=(1-\cos 2x)/2$, we have
\begin{align*}
    \sum_{k=1}^m\sin^2(k\vartheta)=\sum_{k=1}^m\frac{1-\cos(2k\vartheta)}{2}=\frac{m}{2}-\frac{1}{2}\operatorname{Re}\prt{\sum_{k=1}^m\exp(2\mathrm{i}k\vartheta)},
\end{align*}
where $\operatorname{Re}(\cdot)$ denotes the real part. Using $1-\exp(2\mathrm{i}x) = \exp(\mathrm{i}x)[\exp(-\mathrm{i}x)-\exp(\mathrm{i}x)]=-2\mathrm{i}\exp(\mathrm{i}x)\sin x$, we obtain
\begin{align}\label{eq:sum_exp}
    \sum_{k=1}^m\exp(2\mathrm{i}k\vartheta)&=\frac{\exp(2\mathrm{i}\vartheta)\,[1-\exp(2\mathrm{i}m\vartheta)]}{1-\exp(2\mathrm{i}\vartheta)}=\exp(\mathrm{i}(m+1)\vartheta)\,s_m,
\end{align}
whose real part is at most $|s_m|\leq 1/\sin\vartheta$ as $|\exp(\mathrm{i}(m+1)\vartheta)|=1$ and $|\sin(m\vartheta)|\leq1$. Hence,
\begin{align*}
    \sum_{k=1}^m\sin^2(k\vartheta)\geq\frac{m}{2}-\frac{1}{2\sin\vartheta}\geq \frac{m}{4}.
\end{align*}
Moreover,
\begin{align*}
    \sin^2\vartheta=1-\cos^2\vartheta=1-\frac{\tau_1^2}{4\beta}=\frac{(2\sqrt{\beta}-\tau_1)(2\sqrt{\beta}+\tau_1)}{4\beta},
\end{align*}
where $\tau_1^2<4\beta$ in Case (2) gives $0<2\sqrt{\beta}+\tau_1\leq4$ and
\begin{align*}
    0<2\sqrt{\beta}-\tau_1=2\sqrt{\beta}-1-\beta+(1-\beta)s=(1-\beta)s-(1-\sqrt{\beta})^2\leq(1-\beta)s.
\end{align*}
Thus, $\sin^2\vartheta\leq4(1-\beta)s/(4\beta)=(1-\beta)s/\beta\leq2(1-\beta)s$ for $\beta\geq1/2$, and therefore
\begin{align*}
    G_m = \sum_{k=1}^m s_k^2 =\frac{1}{\sin^2\vartheta}\sum_{k=1}^m\sin^2(k\vartheta)\geq \frac{m}{4\sin^2\vartheta}\geq\frac{m}{8(1-\beta)s}\geq \frac{1}{16(1-\beta)^2s}.
\end{align*}
Combining both subcases with \eqref{eq:hk_case1} yields \eqref{eq:hk_sum}.

\emph{Bounding $\sum_{k=0}^{t-1} p_k^2$.} Let
\begin{align*}
    \Gamma_{\mathrm{c}}:=\sum_{k=0}^{m-1}\cos^2(k\vartheta),\quad \Gamma_{\mathrm{s}}:=\sum_{k=0}^{m-1}\sin^2(k\vartheta)>0,\quad \Gamma_{\mathrm{cs}}:=\sum_{k=0}^{m-1}\cos(k\vartheta)\sin(k\vartheta),
\end{align*}
where $\Gamma_{\mathrm{s}}\geq\sin^2\vartheta>0$ since $m\geq2$ and $\vartheta\in(0,\pi)$. Then, $\Gamma_{\mathrm{c}} + \Gamma_{\mathrm{s}} = m$.
Note that
\begin{align*}
    \sum_{k=0}^{m-1}p_k^2 &= \sum_{k=0}^{m-1}\beta^k\brk{\cos(k\vartheta) - B\sin(k\vartheta)}^2\geq\exp(-2)\sum_{k=0}^{m-1}\brk{\cos(k\vartheta) - B\sin(k\vartheta)}^2\\
    &=\exp(-2)\prt{\Gamma_{\mathrm{c}}-2B\Gamma_{\mathrm{cs}}+B^2\Gamma_{\mathrm{s}}}=\exp(-2)\brk{\Gamma_{\mathrm{s}}\prt{B-\frac{\Gamma_{\mathrm{cs}}}{\Gamma_{\mathrm{s}}}}^2+\frac{\Gamma_{\mathrm{c}}\Gamma_{\mathrm{s}} - \Gamma_{\mathrm{cs}}^2}{\Gamma_{\mathrm{s}}}}\\
    &\geq \exp(-2)\frac{\Gamma_{\mathrm{c}}\Gamma_{\mathrm{s}} - \Gamma_{\mathrm{cs}}^2}{\Gamma_{\mathrm{s}}}.
\end{align*}
Here, we invoked $\beta^k\geq \exp(-2)$ for $k\leq m-1$.

To evaluate $\Gamma_{\mathrm{c}}\Gamma_{\mathrm{s}}-\Gamma_{\mathrm{cs}}^2$, let $a_m:=\sum_{k=0}^{m-1}\cos(2k\vartheta)$ and $b_m:=\sum_{k=0}^{m-1}\sin(2k\vartheta)$. The double-angle formulas $\cos^2x=(1+\cos2x)/2$, $\sin^2x=(1-\cos2x)/2$, and $\sin x\cos x=\sin(2x)/2$ give
\begin{align*}
    \Gamma_{\mathrm{c}}=\frac{m+a_m}{2},\quad \Gamma_{\mathrm{s}}=\frac{m-a_m}{2},\quad \Gamma_{\mathrm{cs}}=\frac{b_m}{2},
\end{align*}
so that $4(\Gamma_{\mathrm{c}}\Gamma_{\mathrm{s}}-\Gamma_{\mathrm{cs}}^2)=(m+a_m)(m-a_m)-b_m^2=m^2-(a_m^2+b_m^2)$. Moreover, as in \eqref{eq:sum_exp}
\begin{align*}
    a_m+\mathrm{i}b_m&=\sum_{k=0}^{m-1}\exp(2\mathrm{i}k\vartheta)=\exp(\mathrm{i}(m-1)\vartheta)s_m,
\end{align*}
and hence $a_m^2+b_m^2=s_m^2$. Therefore, $4(\Gamma_{\mathrm{c}}\Gamma_{\mathrm{s}}-\Gamma_{\mathrm{cs}}^2)=m^2-s_m^2$.

By \eqref{eq:sk_rec} and the Cauchy--Schwarz inequality, we obtain
\begin{align*}
    s_m^2=\prt{\sum_{k=0}^{m-1}\cos^{m-1-k}\vartheta\,\cos(k\vartheta)}^2\leq\prt{\sum_{k=0}^{m-1}\cos^{2(m-1-k)}\vartheta}\prt{\sum_{k=0}^{m-1}\cos^2(k\vartheta)}\leq m\Gamma_{\mathrm{c}}.
\end{align*}
Hence $s_m^2\leq m\Gamma_{\mathrm{c}}$.
Consequently,
\begin{align*}
    \exp(2)\sum_{k=0}^{m-1}p_k^2&\geq \frac{\Gamma_{\mathrm{c}}\Gamma_{\mathrm{s}} - \Gamma_{\mathrm{cs}}^2}{\Gamma_{\mathrm{s}}} = \frac{m^2 - s_m^2}{4(m-\Gamma_{\mathrm{c}})}\geq\frac{m^2-m\Gamma_{\mathrm{c}}}{4(m-\Gamma_{\mathrm{c}})}= \frac{m}{4}\geq\frac{1}{8}\min\crk{t,\frac{1}{1-\beta}}.
\end{align*}
Combining the above estimate with \eqref{eq:pk_case1} proves \eqref{eq:pk_sum}.

\noindent\textbf{Proof of \eqref{eq:pk_last}.} Let $\beta\in[0,1)$ and $0<s\leq(1-\beta)/8$. Since $P(0)=\beta\geq0$,
\begin{align*}
    P(1-2s) &=(1-2s)^2-\brk{1+\beta-(1-\beta)s}(1-2s)+\beta\\
    &=1-4s+4s^2-(1+\beta)+2(1+\beta)s+(1-\beta)s-2(1-\beta)s^2+\beta\\
    &=s[2(1+\beta)s-(1-\beta)]\leq s\brk{\frac{1-\beta}{2}-(1-\beta)}<0,
\end{align*}
and $P(1)=(1-\beta)s>0$, the roots are real and satisfy $0\leq\lambda_2<1-2s<\lambda_1<1$. Hence, $h_{k-1}\geq0$ as $h_{k-1}=\sum_{i=0}^{k-1}\lambda_1^{k-1-i}\lambda_2^i$ with $\lambda_1,\lambda_2\geq0$, and \eqref{eq:pk_h} gives $p_k\geq\lambda_1^k\geq(1-2s)^k$ since $\lambda_2(1-\lambda_1)\geq0$.

}

\end{appendices}


\bibliography{references_all}

@article{li2019decentralized,
    title = {A decentralized proximal-gradient method with network independent step-sizes and separated convergence rates},
    author = {Li, Zhi and Shi, Wei and Yan, Ming},
    journal = {IEEE Transactions on Signal Processing},
    volume = {67},
    number = {17},
    pages = {4494--4506},
    year = {2019},
    publisher = {IEEE},
}

@inproceedings{tang2018d,
    title = {D2: Decentralized Training over Decentralized Data},
    author = {Tang, Hanlin and Lian, Xiangru and Yan, Ming and Zhang, Ce and Liu, Ji},
    booktitle = {International Conference on Machine Learning},
    pages = {4848--4856},
    year = {2018},
}

@article{koloskova2020unified,
    title={A unified theory of decentralized SGD with changing topology and local updates},
  author={Koloskova, Anastasia and Loizou, Nicolas and Boreiri, Sadra and Jaggi, Martin and Stich, Sebastian},
  journal={International conference on machine learning},
  pages={5381--5393},
  year={2020},
  organization={PMLR}
}

@article{di2016next,
    title = {Next: In-network nonconvex optimization},
    author = {Di Lorenzo, Paolo and Scutari, Gesualdo},
    journal = {IEEE Transactions on Signal and Information Processing over Networks},
    volume = {2},
    number = {2},
    pages = {120--136},
    year = {2016},
    publisher = {IEEE},
}

@article{nedic2017achieving,
    title = {Achieving geometric convergence for distributed optimization over time-varying graphs},
    author = {Nedi{\'c}, Angelia and Olshevsky, Alex and Shi, Wei},
    journal = {SIAM Journal on Optimization},
    volume = {27},
    number = {4},
    pages = {2597--2633},
    year = {2017},
    publisher = {SIAM},
}

@article{khaled2020better,
    title={Better Theory for {SGD} in the  Nonconvex World},
    author={Ahmed Khaled and Peter Richt{\'a}rik},
    journal={Transactions on Machine Learning Research},
    issn={2835-8856},
    year={2023},
    url={https://openreview.net/forum?id=AU4qHN2VkS},
    note={Survey Certification}
}

@article{shi2015extra,
    title = {Extra: An exact first-order algorithm for decentralized consensus optimization},
    author = {Shi, Wei and Ling, Qing and Wu, Gang and Yin, Wotao},
    journal = {SIAM Journal on Optimization},
    volume = {25},
    number = {2},
    pages = {944--966},
    year = {2015},
    publisher = {SIAM},
}

@article{yuan2018exact,
    title = {Exact diffusion for distributed optimization and learning—Part I: Algorithm development},
    author = {Yuan, Kun and Ying, Bicheng and Zhao, Xiaochuan and Sayed, Ali H},
    journal = {IEEE Transactions on Signal Processing},
    volume = {67},
    number = {3},
    pages = {708--723},
    year = {2018},
    publisher = {IEEE},
}

@article{huang2021improving,
    title={Improving the transient times for distributed stochastic gradient methods},
  author={Huang, Kun and Pu, Shi},
  journal={IEEE Transactions on Automatic Control},
  volume={68},
  number={7},
  pages={4127--4142},
  year={2022},
  publisher={IEEE}
}

@article{pu2021sharp,
  title={A sharp estimate on the transient time of distributed stochastic gradient descent},
  author={Pu, Shi and Olshevsky, Alex and Paschalidis, Ioannis Ch},
  journal={IEEE Transactions on Automatic Control},
  volume={67},
  number={11},
  pages={5900--5915},
  year={2021},
  publisher={IEEE}
}

@inproceedings{lu2021optimal,
  title={Optimal complexity in decentralized training},
  author={Lu, Yucheng and De Sa, Christopher},
  booktitle={International Conference on Machine Learning},
  pages={7111--7123},
  year={2021},
  organization={PMLR}
}

@article{liu2011accelerated,
    title = {Accelerated linear iterations for distributed averaging},
    author = {Liu, Ji and Morse, A Stephen},
    journal = {Annual Reviews in Control},
    volume = {35},
    number = {2},
    pages = {160--165},
    year = {2011},
    publisher = {Elsevier},
}

@article{pu2021distributed,
    title = {Distributed stochastic gradient tracking methods},
    author = {Pu, Shi and Nedi{\'c}, Angelia},
    journal = {Mathematical Programming},
    volume = {187},
    number = {1},
    pages = {409--457},
    year = {2021},
    publisher = {Springer},
}

@article{nedic2018network,
    title = {Network topology and communication-computation tradeoffs in decentralized optimization},
    author = {Nedi{\'c}, Angelia and Olshevsky, Alex and Rabbat, Michael G},
    journal = {Proceedings of the IEEE},
    volume = {106},
    number = {5},
    pages = {953--976},
    year = {2018},
    publisher = {IEEE},
}

@inproceedings{lian2017can,
    author = {Xiangru Lian and Ce Zhang and Huan Zhang and Cho-Jui Hsieh and Wei Zhang and Ji Liu},
    title = {Can Decentralized Algorithms Outperform Centralized Algorithms? A Case Study for Decentralized Parallel Stochastic Gradient Descent},
    year = {2017},
    cdate = {1483228800000},
    pages = {5336-5346},
    booktitle = {NIPS},
}

@article{pu2020asymptotic,
    title = {Asymptotic network independence in distributed stochastic optimization for machine learning: Examining distributed and centralized stochastic gradient descent},
    author = {Pu, Shi and Olshevsky, Alex and Paschalidis, Ioannis Ch},
    journal = {IEEE signal processing magazine},
    volume = {37},
    number = {3},
    pages = {114--122},
    year = {2020},
    publisher = {IEEE},
}

@article{xin2021improved,
    author = {Xin, Ran and Khan, Usman A. and Kar, Soummya},
    journal = {IEEE Transactions on Signal Processing},
    title = {An Improved Convergence Analysis for Decentralized Online Stochastic Non-Convex Optimization},
    year = {2021},
    volume = {69},
    number = {},
    pages = {1842-1858},
    doi = {10.1109/TSP.2021.3062553},
}

@inproceedings{xu2015augmented,
    title = {Augmented distributed gradient methods for multi-agent optimization under uncoordinated constant stepsizes},
    author = {Xu, Jinming and Zhu, Shanying and Soh, Yeng Chai and Xie, Lihua},
    booktitle = {2015 54th IEEE Conference on Decision and Control (CDC)},
    pages = {2055--2060},
    year = {2015},
    organization = {IEEE},
}

@article{alghunaim2021unified,
    title={A unified and refined convergence analysis for non-convex decentralized learning},
  author={Alghunaim, Sulaiman A and Yuan, Kun},
  journal={IEEE Transactions on Signal Processing},
  volume={70},
  pages={3264--3279},
  year={2022},
  publisher={IEEE}
}

@article{song2021optimal,
  title={Optimal gradient tracking for decentralized optimization},
  author={Song, Zhuoqing and Shi, Lei and Pu, Shi and Yan, Ming},
  journal={Mathematical Programming},
  pages={1--53},
  year={2023},
  publisher={Springer}
}

@article{krizhevsky2009learning,
  title={Learning multiple layers of features from tiny images},
  author={Krizhevsky, Alex and Hinton, Geoffrey and others},
  year={2009},
  publisher={Citeseer}
}

@article{zhang2019decentralized,
  title={Decentralized stochastic gradient tracking for non-convex empirical risk minimization},
  author={Zhang, Jiaqi and You, Keyou},
  journal={arXiv preprint arXiv:1909.02712},
  year={2019}
}

@article{huang2023distributed,
  author={Huang, Kun and Li, Xiao and Pu, Shi},
  journal={IEEE Transactions on Automatic Control}, 
  title={Distributed Stochastic Optimization Under a General Variance Condition}, 
  year={2024},
  volume={69},
  number={9},
  pages={6105-6120},
}

@article{huang2023cedas,
     author={Huang, Kun and Pu, Shi},
  journal={IEEE Transactions on Automatic Control}, 
  title={CEDAS: A Compressed Decentralized Stochastic Gradient Method With Improved Convergence}, 
  year={2025},
  volume={70},
  number={4},
  pages={2242-2257},
}

@inproceedings{karimi2016linear,
  title={Linear convergence of gradient and proximal-gradient methods under the polyak-{\l}ojasiewicz condition},
  author={Karimi, Hamed and Nutini, Julie and Schmidt, Mark},
  booktitle={Joint European conference on machine learning and knowledge discovery in databases},
  pages={795--811},
  year={2016},
  organization={Springer}
}

@article{liu2020improved,
  title={An improved analysis of stochastic gradient descent with momentum},
  author={Liu, Yanli and Gao, Yuan and Yin, Wotao},
  journal={Advances in Neural Information Processing Systems},
  volume={33},
  pages={18261--18271},
  year={2020}
}

@article{yuan2022revisiting,
  title={Revisiting optimal convergence rate for smooth and non-convex stochastic decentralized optimization},
  author={Yuan, Kun and Huang, Xinmeng and Chen, Yiming and Zhang, Xiaohan and Zhang, Yingya and Pan, Pan},
  journal={Advances in Neural Information Processing Systems},
  volume={35},
  pages={36382--36395},
  year={2022}
}

@article{polyak1964some,
  title={Some methods of speeding up the convergence of iteration methods},
  author={Polyak, Boris T},
  journal={Ussr computational mathematics and mathematical physics},
  volume={4},
  number={5},
  pages={1--17},
  year={1964},
  publisher={Elsevier}
}

@article{huang2023abcrr,
      author={Huang, Kun and Zhou, Linli and Pu, Shi},
  journal={IEEE Transactions on Automatic Control}, 
  title={Distributed Random Reshuffling Methods With Improved Convergence}, 
  year={2025},
  volume={70},
  number={9},
  pages={6175-6182},
  doi={10.1109/TAC.2025.3552743}}

@article{xiao2023one,
  title={A One-Sample Decentralized Proximal Algorithm for Non-Convex Stochastic Composite Optimization},
  author={Xiao, Tesi and Chen, Xuxing and Balasubramanian, Krishnakumar and Ghadimi, Saeed},
  journal={arXiv preprint arXiv:2302.09766},
  year={2023}
}

@inproceedings{yuan2021decentlam,
  title={DecentLaM: Decentralized momentum SGD for large-batch deep training},
  author={Yuan, Kun and Chen, Yiming and Huang, Xinmeng and Zhang, Yingya and Pan, Pan and Xu, Yinghui and Yin, Wotao},
  booktitle={Proceedings of the IEEE/CVF International Conference on Computer Vision},
  pages={3029--3039},
  year={2021}
}

@article{wang2021distributed,
  title={Distributed stochastic consensus optimization with momentum for nonconvex nonsmooth problems},
  author={Wang, Zhiguo and Zhang, Jiawei and Chang, Tsung-Hui and Li, Jian and Luo, Zhi-Quan},
  journal={IEEE Transactions on Signal Processing},
  volume={69},
  pages={4486--4501},
  year={2021},
  publisher={IEEE}
}

@article{lei2019stochastic,
  title={Stochastic gradient descent for nonconvex learning without bounded gradient assumptions},
  author={Lei, Yunwen and Hu, Ting and Li, Guiying and Tang, Ke},
  journal={IEEE transactions on neural networks and learning systems},
  volume={31},
  number={10},
  pages={4394--4400},
  year={2019},
  publisher={IEEE}
}

@article{huang2025accelerated,
  title={An accelerated distributed stochastic gradient method with momentum},
  author={Huang, Kun and Pu, Shi and Nedi{\'c}, Angelia},
  journal={Mathematical Programming},
  pages={1--44},
  year={2025},
  publisher={Springer}
}

@article{liu2024decentralized,
  title={Decentralized gradient tracking with local steps},
  author={Liu, Yue and Lin, Tao and Koloskova, Anastasia and Stich, Sebastian U},
  journal={Optimization Methods and Software},
  pages={1--28},
  year={2024},
  publisher={Taylor \& Francis}
}

@article{alghunaim2024local,
  author={Alghunaim, Sulaiman A.},
  journal={IEEE Transactions on Automatic Control}, 
  title={Local Exact-Diffusion for Decentralized Optimization and Learning}, 
  year={2024},
  volume={69},
  number={11},
  pages={7371-7386}
}

@article{yang2024accelerating,
  title={Accelerating Distributed Optimization: A Primal-Dual Perspective on Local Steps},
  author={Yang, Junchi and Yildirim, Murat and Feng, Qiu},
  journal={arXiv preprint arXiv:2407.02689},
  year={2024}
}

@article{li2019communication,
  title={Communication-efficient local decentralized SGD methods},
  author={Li, Xiang and Yang, Wenhao and Wang, Shusen and Zhang, Zhihua},
  journal={arXiv preprint arXiv:1910.09126},
  year={2019}
}

@inproceedings{ge2023gradient,
  title={Gradient tracking with multiple local SGD for decentralized non-convex learning},
  author={Ge, Songyang and Chang, Tsung-Hui},
  booktitle={2023 62nd IEEE Conference on Decision and Control (CDC)},
  pages={133--138},
  year={2023},
  organization={IEEE}
}

@article{du2024unified,
  title={A unified momentum-based paradigm of decentralized SGD for non-convex models and heterogeneous data},
  author={Du, Haizhou and Cheng, Chaoqian and Ni, Chengdong},
  journal={Artificial Intelligence},
  volume={332},
  pages={104130},
  year={2024},
  publisher={Elsevier}
}

@article{gao2020periodic,
  title={Periodic stochastic gradient descent with momentum for decentralized training},
  author={Gao, Hongchang and Huang, Heng},
  journal={arXiv preprint arXiv:2008.10435},
  year={2020}
}

@inproceedings{huang2023computation,
  title={On the computation-communication trade-off with a flexible gradient tracking approach},
  author={Huang, Yan and Xu, Jinming},
  booktitle={2023 62nd IEEE Conference on Decision and Control (CDC)},
  pages={284--289},
  year={2023},
  organization={IEEE}
}

@article{qiu2024convergence,
  title={Convergence of SGD with momentum in the nonconvex case: A time window-based analysis},
  author={Qiu, Junwen and Ma, Bohao and Milzarek, Andre},
  journal={arXiv preprint arXiv:2405.16954},
  year={2024}
}

@article{xin2021stochastic,
  title={A stochastic proximal gradient framework for decentralized non-convex composite optimization: Topology-independent sample complexity and communication efficiency},
  author={Xin, Ran and Das, Subhro and Khan, Usman A and Kar, Soummya},
  journal={arXiv preprint arXiv:2110.01594},
  year={2021}
}

@inproceedings{karimireddy2020scaffold,
  title={Scaffold: Stochastic controlled averaging for federated learning},
  author={Karimireddy, Sai Praneeth and Kale, Satyen and Mohri, Mehryar and Reddi, Sashank and Stich, Sebastian and Suresh, Ananda Theertha},
  booktitle={International conference on machine learning},
  pages={5132--5143},
  year={2020},
  organization={PMLR}
}

@inproceedings{suresh2017distributed,
  title={Distributed mean estimation with limited communication},
  author={Suresh, Ananda Theertha and Felix, X Yu and Kumar, Sanjiv and McMahan, H Brendan},
  booktitle={International conference on machine learning},
  pages={3329--3337},
  year={2017},
  organization={PMLR}
}

@article{chang2020distributed,
  author={Chang, Tsung-Hui and Hong, Mingyi and Wai, Hoi-To and Zhang, Xinwei and Lu, Songtao},
  journal={IEEE Signal Processing Magazine}, 
  title={Distributed Learning in the Nonconvex World: From batch data to streaming and beyond}, 
  year={2020},
  volume={37},
  number={3},
  pages={26-38},
  doi={10.1109/MSP.2020.2970170}}

@ARTICLE{Dimakis2010gossip,
  author={Dimakis, Alexandros G. and Kar, Soummya and Moura, José M. F. and Rabbat, Michael G. and Scaglione, Anna},
  journal={Proceedings of the IEEE}, 
  title={Gossip Algorithms for Distributed Signal Processing}, 
  year={2010},
  volume={98},
  number={11},
  pages={1847-1864},
  doi={10.1109/JPROC.2010.2052531}}

@inproceedings{song2023fedavg,
  title={FedAvg converges to zero training loss linearly for overparameterized multi-layer neural networks},
  author={Song, Bingqing and Khanduri, Prashant and Zhang, Xinwei and Yi, Jinfeng and Hong, Mingyi},
  booktitle={International Conference on Machine Learning},
  pages={32304--32330},
  year={2023},
  organization={PMLR}
}

@inproceedings{ardakani2024slimfit,
  title={SlimFit: Memory-Efficient Fine-Tuning of Transformer-based Models Using Training Dynamics},
  author={Ardakani, Arash and Haan, Altan and Tan, Shangyin and Popovici, Doru Thom and Cheung, Alvin and Iancu, Costin and Sen, Koushik},
  booktitle={Proceedings of the 2024 Conference of the North American Chapter of the Association for Computational Linguistics: Human Language Technologies (Volume 1: Long Papers)},
  pages={6218--6236},
  year={2024}
}

@article{kovalev2020optimal,
  title={Optimal and practical algorithms for smooth and strongly convex decentralized optimization},
  author={Kovalev, Dmitry and Salim, Adil and Richt{\'a}rik, Peter},
  journal={Advances in Neural Information Processing Systems},
  volume={33},
  pages={18342--18352},
  year={2020}
}

@article{qu2019accelerated,
  title={Accelerated distributed Nesterov gradient descent},
  author={Qu, Guannan and Li, Na},
  journal={IEEE Transactions on Automatic Control},
  volume={65},
  number={6},
  pages={2566--2581},
  year={2019},
  publisher={IEEE}
}

@article{reingold2002entropy,
  title={Entropy waves, the zig-zag graph product, and new constant-degree expanders},
  author={Reingold, Omer and Vadhan, Salil and Wigderson, Avi},
  journal={Annals of mathematics},
  pages={157--187},
  year={2002},
  publisher={JSTOR}
}

@article{choo2011elementary,
  title={An elementary proof of the Jury test for real polynomials},
  author={Choo, Younseok},
  journal={Automatica},
  volume={47},
  number={1},
  pages={249--252},
  year={2011},
  publisher={Elsevier}
}

@article{berrut2004barycentric,
  title={Barycentric lagrange interpolation},
  author={Berrut, Jean-Paul and Trefethen, Lloyd N},
  journal={SIAM review},
  volume={46},
  number={3},
  pages={501--517},
  year={2004},
  publisher={SIAM}
}

@article{zhu2024sparkle,
  title={SPARKLE: a unified single-loop primal-dual framework for decentralized bilevel optimization},
  author={Zhu, Shuchen and Kong, Boao and Lu, Songtao and Huang, Xinmeng and Yuan, Kun},
  journal={Advances in Neural Information Processing Systems},
  volume={37},
  pages={62912--62987},
  year={2024}
}

@book{brouwer2011spectra,
  title={Spectra of graphs},
  author={Brouwer, Andries E and Haemers, Willem H},
  year={2011},
  publisher={Springer Science \& Business Media},
  address={Heidelberg}
}

@inproceedings{patel2024limits,
  title={The limits and potentials of local sgd for distributed heterogeneous learning with intermittent communication},
  author={Patel, Kumar Kshitij and Glasgow, Margalit and Zindari, Ali and Wang, Lingxiao and Stich, Sebastian U and Cheng, Ziheng and Joshi, Nirmit and Srebro, Nathan},
  booktitle={The Thirty Seventh Annual Conference on Learning Theory},
  pages={4115--4157},
  year={2024},
  organization={PMLR}
}

@inproceedings{luo2025revisiting,
title={Revisiting Local{SGD} and {SCAFFOLD}: Improved Rates and Missing Analysis},
author={Ruichen Luo and Sebastian U Stich and Samuel Horv{\'a}th and Martin Tak{\'a}{\v{c}}},
booktitle={The 28th International Conference on Artificial Intelligence and Statistics},
year={2025},
url={https://openreview.net/forum?id=DCt4v3Siky}
}

@article{huang2024distributed,
  title={Distributed normal map-based stochastic proximal gradient methods over networks},
  author={Huang, Kun and Pu, Shi and Nedi{\'c}, Angelia},
  journal={arXiv preprint arXiv:2412.13054},
  year={2024}
}

@article{d2021acceleration,
  title={Acceleration methods},
  author={d’Aspremont, Alexandre and Scieur, Damien and Taylor, Adrien},
  journal={Foundations and trends in optimization},
  volume={5},
  number={1-2},
  pages={1--245},
  year={2021},
  publisher={Emerald Publishing Limited}
}

@article{marek2026small,
  title={Small batch size training for language models: When vanilla sgd works, and why gradient accumulation is wasteful},
  author={Marek, Martin and Lotfi, Sanae and Somasundaram, Aditya and Wilson, Andrew and Goldblum, Micah},
  journal={Advances in Neural Information Processing Systems},
  volume={38},
  pages={148837--148862},
  year={2026}
}

\end{document}